\documentclass[10pt,twoside,leqno]{amsart}
\pdfoutput=1
\usepackage{amsfonts}
\usepackage{amsmath}
\usepackage{mathtools}
\usepackage{amsthm}
\usepackage{amssymb}
\usepackage{pifont}
\newcommand{\cmark}{\ding{51}}

\usepackage{mathrsfs}
\usepackage[numbers]{natbib}
\usepackage[fit]{truncate}
\usepackage{geometry}
\usepackage{hyperref}
\usepackage{bm}
\usepackage{tikz-cd}
\usepackage[makeroom]{cancel}
\usepackage{xcolor,colortbl}
\usepackage{hhline}
\usepackage{float}
\usepackage{enumitem} 
\usepackage{stackengine}
\newcommand\xrowht[2][0]{\addstackgap[.5\dimexpr#2\relax]{\vphantom{#1}}}
\usepackage{array, multirow}
\allowdisplaybreaks
\newcommand{\R}{\mathbb{R}}
\newcommand{\C}{\mathbb{C}}

\newcommand{\cm}[1]{\ignorespaces}

\newcommand{\frs}[1]{\mathfrak{s}_{#1}}
\newcommand{\nj}[3]{f^{#3}(N^J(f_{#1},f_{#2}))}

\newcommand{\PreserveBackslash}[1]{\let\temp=\\#1\let\\=\temp}
\newcolumntype{C}[1]{>{\PreserveBackslash\centering}p{#1}}
\newcolumntype{R}[1]{>{\PreserveBackslash\raggedleft}p{#1}}
\newcolumntype{L}[1]{>{\PreserveBackslash\raggedright}p{#1}}

\newcommand{\mlines}[2][c]{%
\begin{tabular}[#1]{@{}c@{}}#2\end{tabular}}

\theoremstyle{plain}
\newtheorem{theorem}{Theorem}[section]
\newtheorem*{theorem*}{Theorem}
\newtheorem{corollary}[theorem]{Corollary}

\newtheorem{proposition}[theorem]{Proposition}
\newtheorem*{proposition*}{Proposition}

\theoremstyle{definition}

\newtheorem{remark}[theorem]{Remark}
\numberwithin{equation}{section}

\usepackage{chngcntr}
\counterwithin{table}{section}

\begin{document}

\title[Hermitian structures on six-dimensional almost nilpotent solvmanifolds]{Hermitian structures on six-dimensional \\ almost nilpotent  solvmanifolds}

\author{Anna Fino}
\address[A. Fino]{Dipartimento di Matematica ``G. Peano''\\
	Universit\`a di Torino\\
	Via Carlo Alberto 10\\
	10123 Torino, Italy \& Department of Mathematics and Statistics, Florida International University, 33199 Miami, Florida, USA} 
\email{annamaria.fino@unito.it, afino@fiu.edu}

\author{Fabio Paradiso}
\address[F. Paradiso]{Dipartimento di Matematica ``G. Peano''\\
	Universit\`a di Torino\\
	Via Carlo Alberto 10\\
	10123 Torino, Italy} \email{fabio.paradiso@unito.it}


\begin{abstract} 
We complete the classification  of  six-dimensional strongly unimodular almost nilpotent Lie algebras  admitting complex structures. For several cases we  describe the space of complex structures  up to isomorphism. As a consequence we determine   the six-dimensional  almost  nilpotent solvmanifolds admitting an invariant complex structure and study the existence of special types of  Hermitian metrics, including SKT, balanced, locally conformally K\"ahler, and  strongly Gauduchon metrics. In particular, we determine new balanced solvmanifolds and confirm a conjecture by the first author and Vezzoni regarding SKT and balanced structures in the six-dimensional strongly unimodular almost nilpotent case. Moreover, we prove some negative results regarding complex structures tamed by symplectic forms, showing in particular that  in every  dimension such structures cannot exist on non-K\"ahler almost abelian Lie algebras. 
\end{abstract} 

\maketitle


\section{Introduction}

A Hermitian structure on a  $2n$-dimensional smooth manifold $M$ is defined as a pair $(J,g)$, where $J$ is an integrable almost complex structure which is compatible with a Riemannian metric $g$ on $M$, namely $g(J\cdot,J\cdot)=g(\cdot, \cdot)$. When the fundamental form $\omega (\cdot, \cdot) \coloneqq g(J\cdot,\cdot)$ is closed, or equivalently  $J$  is  parallel with respect to the Levi-Civita connection relative to $g$,  the Hermitian structure $(J,g)$ is K\"ahler.  An  ever-growing interest has been placed on generalizations of the K\"ahler condition, whose inception stems from both the desire to equip non-K\"ahler complex manifolds with special (hopefully canonical) Hermitian metrics and from their applications in theoretical physics, demanding their investigation.

Two of the most studied ones are the  \emph{strong K\"ahler with torsion} (SKT for short)  or  \emph{pluriclosed} condition, defined by the closure of $H \coloneqq d^c \omega=i(\overline\partial -\partial) \omega = Jd\omega$, and the \emph{balanced} condition, characterized by the coclosure of $\omega$ or equivalently by the closure of $\omega^{n-1}$. The SKT and balanced conditions are incompatible one with the other, in the sense that a Hermitian metric which is both SKT and balanced is necessarily K\"ahler, and  it has even been conjectured (and confirmed in some special cases) that a compact complex manifold admitting both SKT and balanced Hermitian metrics must also admit K\"ahler metrics.

Other important generalizations of the K\"ahler condition include: 
\begin{itemize}
\item the \emph{balanced} condition, characterized by the closure of  $\omega^{n-1}$,
\item the \emph{locally conformally balanced} (LCB) condition, characterized by the closure of the \emph{Lee form} $\theta$, which is  the unique $1$-form satisfying $d\omega^{n-1}=\theta \wedge \omega^{n-1}$,
\item the \emph{locally conformally K\"ahler} (LCK) condition, which holds when $d\omega=\frac{1}{n-1} \theta \wedge \omega$,
\item the \emph{locally conformally SKT} (LCSKT) condition, introduced in \cite{DFFL} and characterized by the condition $dH=\mu \wedge H$, with  $\mu$ a non-zero closed $1$-form,
\item the \emph{1\textsuperscript{st}-Gauduchon} condition, characterized by the condition  $\partial \overline\partial \omega \wedge \omega^{n-2}=0$,
\item the \emph{strongly Gauduchon} condition,  characterized by $\partial \omega^{n-1}$ being $\overline\partial$-exact.
\end{itemize}
We recall that the SKT condition implies the 1\textsuperscript{st}-Gauduchon condition, the LCK or balanced condition implies the LCB condition, while the balanced condition implies the strongly Gauduchon condition.

Another property which has gathered interested in recent literature is the \emph{tamed} condition for complex structures: given a complex structure $J$ and a symplectic form $\Omega$ on $M$, $J$ is said to be tamed by $\Omega$ if the $(1,1)$-part of $\Omega$ is positive. This condition generalizes the K\"ahler condition, since the fundamental form of a K\"ahler structure $(J,g)$ trivially tames $J$. In \cite{EFV}, it has been proven that the existence of a symplectic form $\Omega$ taming $J$ is equivalent to the existence of a $J$-Hermitian metric whose fundamental form $\omega$ satisfies $\partial \omega = \overline\partial \beta$ for some $\partial$-closed $(2,0)$-form $\beta$. In particular, the Hermitian metric has to be SKT.

A natural setting in which to study the previously-mentioned types of Hermitian structures is provided by  \emph{nilmanifolds} or \emph{solvmanifolds}, i.e. by  compact  quotients $\Gamma \backslash G$  of solvable or nilpotent Lie groups by  cocompact \emph{lattices}, since left-invariant structures on $G$ naturally descend to locally homogeneous structures on such quotients. 
A  solvable Lie group $G$  is called  \emph{almost nilpotent}  if  its  nilradical has codimension one.  In particular,  if the nilradical is abelian,  $G$ is  \emph{almost abelian}. Useful criteria  for the existence of lattices of almost nilpotent Lie groups are given in \cite{Bock}.

For a solvable Lie   group $G$ a necessary condition for the existence of cocompact lattices   is the strongly unimodularity of the Lie algebra $\mathfrak{g}$  of $G$, i.e. that   for every $X \in \mathfrak{g}$ and every $k \in \mathbb{N}$, one has $\operatorname{tr} \text{ad}_{X} \rvert_{\mathfrak{n}^k / \mathfrak{n}^{k+1}} =0$, where $\mathfrak{n}$ is the nilradical of $\mathfrak{g}$ (see  \cite{Gar}).

Real dimension six is of particular interest, due to it often being the first non-trivial dimension in which to look for certain special structures and for its theoretical physical applications. Moreover, nilpotent and solvable Lie groups are fully classified in dimension six (see  \cite{Mub1, Mub2, Mub3, SW, Sha}).

In the nilpotent case, most of the previously defined structures have been thoroughly investigated in literature during the last decades (see \cite{DFFL, FPS, OOS, Sal, Saw, Uga}), while only partial results exist thus far in the general solvable case. Recently, almost abelian Lie groups have gathered notable interest, leading to the development of several characterization and classification results for special Hermitian structures (see \cite{LW} for the K\"ahler case, \cite{AL, FP_gk} for the SKT case, \cite{FP_bal, Puj} for the balanced case and \cite{AO, Par} for the LCK and LCB case). Furthermore, in \cite{FS}, the authors obtained results regarding $2$-step solvable Lie groups, of which almost abelian ones are a special class. Another special class of almost nilpotent Lie groups has been studied in \cite{FP_alm}, where characterizations and classification results regarding complex, SKT and balanced structures were obtained for almost nilpotent Lie groups whose nilradical has one-dimensional commutator.  In the paper  we  extend and refine these results, completing the classification of six-dimensional strongly unimodular almost nilpotent Lie algebras admitting complex structures and the aforementioned types of special Hermitian structures.  This involves exploiting the technical machinery developed in the previous literature regarding almost  (nilpotent) abelian  Lie algebras, as well as some different kinds of computations, as we shall explain. In particular, we determine that the previous classification results in \cite{FP_gk, FP_alm} exhaust the six-dimensional strongly unimodular almost nilpotent Lie algebras admitting SKT structures, while, with respect to \cite{FP_alm, FP_bal}, we were able to find two more Lie algebras, up to isomorphisms, admitting balanced structures. Moreover, we proved that, in the six-dimensional almost nilpotent strongly unimodular case, a Lie algebra admitting strongly Gauduchon structures always admits balanced structures: it would be interesting to see if this result holds in higher dimensions as well.

In Section \ref{sec_cpx}, we start by classifying six-dimensional strongly unimodular almost nilpotent Lie algebras admitting complex structures, while, in Section \ref{sec_mod}, we delve deeper into the subject by providing a classification up to automorphisms of complex structures on some of these Lie algebras. In Section \ref{sec_1stG}, we then provide a full classification result regarding 1\textsuperscript{st}-Gauduchon structures and the same is done for SKT and LCSKT structures in Section \ref{sec_SKT}.  In particular, we determine the first example of compact manifold admitting an LCSKT structure with non-degenerate torsion. The results of  Section \ref{sec_SKT}  are then used in Section \ref{sec_tamed} to investigate the existence of complex structures tamed by symplectic forms, showing that such structures can never exist on non-K\"ahler almost abelian Lie algebras and  thus improving the partial results obtained in \cite{FK, FKV}. We then go back to classifying Lie algebras admitting special Hermitian metrics, studying LCB and LCK structures in Section \ref{sec_LCBLCK} and strongly Gauduchon and balanced structures in Section \ref{sec_bal}.

\medskip

\noindent
\emph{Acknowledgments}. The authors would like to thank Adri\'an Andrada for useful comments. The authors are  partially supported by Project PRIN 2017 “Real and complex manifolds: Topology, Geometry and Holomorphic Dynamics” and  by GNSAGA (Indam). The first author is also supported  by a grant from the Simons Foundation (\#944448).

\section{Classification} \label{sec_cpx}
We begin our discussion by completing the classification of six-dimensional strongly unimodular almost nilpotent Lie algebras admitting a complex structure. To fix our notations, we recall that a complex structure $J$ on a Lie algebra $\mathfrak{g}$ is an endomorphism of $\mathfrak{g}$ squaring to $-\text{Id}_{\mathfrak{g}}$ and satisfying the integrability condition $N^J=0$, where $N^J \in \Lambda^2 \mathfrak{g}^* \otimes \mathfrak{g}$ is the \emph{Nijenhuis tensor} associated with $J$, whose expression is
\[
N^J(X,Y)=[X,Y]-J[JX,Y]-J[X,JY]-[JX,JY], \quad X,Y \in \mathfrak{g}.
\]
The left-invariant extension of such $J$ on the Lie group $G$ corresponding to $\mathfrak{g}$ constitutes an integrable complex structure on $G$.

Six-dimensional strongly unimodular almost nilpotent Lie algebras admitting complex structures have already been classified in the cases where the commutator of the five-dimensional nilradical is either zero-dimensional ($\mathfrak{n} \cong \R^5$, namely the \emph{almost abelian} case) or one-dimensional, corresponding to $\mathfrak{n}$ being isomorphic to either $\mathfrak{h}_3 \oplus \R^2$ or $\mathfrak{h}_5$.   Here, for every positive integer $k$, $\mathfrak{h}_{2k+1}=\R\left<x_1,\ldots,x_k,y_1,\ldots,y_k,z\right>$ denotes the $2k+1$-dimensional Heisenberg Lie algebra whose Lie bracket is defined by $[x_j,y_j]=z$, $j=1,\ldots,k$. For the sake of simplicity, we shall refer to such nilradicals as being of \emph{Heisenberg-type}.

Therefore, it remains to analyze the case where the commutator of the nilradical is at least two-dimensional.
Looking at the classification results in \cite{SW}, these Lie algebras can be grouped into eight classes of isomorphism, listed in Table \ref{table-new}. In this table and in the rest of the paper, Lie algebras are indicated using the naming convention of \cite{SW} (except for abelian Lie algebras $\R^k$ and Heisenberg Lie algebras $\mathfrak{h}_{2k+1}$) and are identified via their structure equations. For example, the notation
\[
\mathfrak{s}_{6.140}^{-1}=(f^{35}+f^{16},f^{45}-f^{26},f^{36},-f^{46},0,0)
\]
means that the Lie algebra $\mathfrak{s}_{6.140}^{-1}$ admits a basis $\{f_1,\ldots,f_6\}$ whose dual coframe $\{f^1,\ldots,f^6\}$ satisfies the following equations with respect to the Chevalley-Eilenberg differential:
\begin{gather*}
df^1=f^{35}+f^1 \wedge f^6,\quad df^2=f^4 \wedge f^5-f^2 \wedge f^6, \quad df^3= f^3 \wedge f^6,\\ df^4=-f^4 \wedge f^6,\quad df^5=df^6=0.
\end{gather*}

As one can see, the remaining possible nilradicals are all isomorphic to either
\[
\mathfrak{n}_{5.1}=(f^{35},f^{45},0,0,0), \quad \text{or} \quad
\mathfrak{n}_{5.2}=(f^{35},f^{34},f^{45},0,0).
\]

\begin{remark} \label{rem_152}
Note  that the two Lie algebras $\frs{6.152}^{\varepsilon=\pm 1}$ in \cite{SW} are actually isomorphic one to the other, as an isomorphism between the two is exhibited by changing the sign of the basis elements $e_1$, $e_3$, $e_4$, $e_6$. For this reason, in this paper $\frs{6.152}$ represents only one isomorphism class and has no free parameters. Moreover, such Lie algebra is isomorphic to the Lie algebra
\[
\mathfrak{g}_9=\left( e^{45},e^{15}+e^{36},e^{14}-e^{26}+e^{56},-e^{56},e^{46},0 \right)
\]
in \cite{FOU}, as one can obtain these structure equations starting from the ones of $\frs{6.152}$ in Table \ref{table-new} and setting
\[
e_1=f_3, \quad e_2=-f_2, \quad e_3=f_1, \quad e_4=f_5, \quad e_5=f_2-f_4, \quad e_6=f_6.
\]
It then follows from \cite{FOU} that the Lie algebra $\frs{6.152}$ admits complex structures with closed $(3,0)$-form.
\end{remark}

\begin{table}[H]
\begin{center}
\addtolength{\leftskip} {-2cm}
\addtolength{\rightskip}{-2cm}
\scalebox{0.85}{
\begin{tabular}{|l|l|l|c|}
\hline \xrowht{15pt}
Name& Structure equations & Nilradical & Step of solvability\\
\hline \hline
\xrowht{20pt}
$\mathfrak{s}_{6.140}^{-1}$ & $(f^{35}+f^{16},f^{45}-f^{26},f^{36},-f^{46},0,0)$  & $\mathfrak{n}_{5.1}=(f^{35},f^{45},0,0,0)$ & 2 \\ \hline
\xrowht{20pt}
$\mathfrak{s}_{6.145}^{0}$ & $(f^{35}+f^{26},f^{45}-f^{16},f^{46},-f^{36},0,0)$  & $\mathfrak{n}_{5.1}=(f^{35},f^{45},0,0,0)$ & 2 \\ \hline
\xrowht{20pt}
$\mathfrak{s}_{6.146}^{-1}$ & $(f^{35}+f^{16},f^{45}-f^{16}+f^{36},f^{36},-f^{46},0,0)$  & $\mathfrak{n}_{5.1}=(f^{35},f^{45},0,0,0)$ & 2 \\ \hline
\xrowht{20pt}
$\mathfrak{s}_{6.147}^{0}$ & $(f^{35}+f^{26}+f^{36},f^{45}-f^{16}+f^{46},f^{46},-f^{36},0,0)$  & $\mathfrak{n}_{5.1}=(f^{35},f^{45},0,0,0)$ & 2 \\ \hline
\xrowht{20pt}
$\mathfrak{s}_{6.151}$ & $(f^{35}+f^{16},f^{34}-f^{26}-f^{46},f^{45},-f^{46},f^{56},0)$  & $\mathfrak{n}_{5.2}=(f^{35},f^{34},f^{45},0,0)$ & 3 \\ \hline
\xrowht{20pt}
$\mathfrak{s}_{6.152}$ & $(f^{35}+f^{26},f^{34}-f^{16} + f^{56},f^{45},-f^{56},f^{46},0)$  & $\mathfrak{n}_{5.2}=(f^{35},f^{34},f^{45},0,0)$ & 3 \\ \hline
\xrowht{20pt}
$\mathfrak{s}_{6.154}^0$ & $(f^{35}+f^{26},f^{34}-f^{16},f^{45},-f^{56},f^{46},0)$  & $\mathfrak{n}_{5.2}=(f^{35},f^{34},f^{45},0,0)$ & 3 \\ \hline
\xrowht{20pt}
$\mathfrak{s}_{6.155}^1$ & $(f^{35}+f^{16},f^{34}-f^{26},f^{45},-f^{46},f^{56},0)$  & $\mathfrak{n}_{5.2}=(f^{35},f^{34},f^{45},0,0)$ & 3 \\ \hline
\end{tabular}}
\caption{Six-dimensional strongly unimodular almost nilpotent Lie algebras with nilradical having commutator of dimension at least two.} \label{table-new}
\end{center}
\end{table}

\begin{theorem} \label{cpx-class}
Let $\mathfrak{g}$ be a six-dimensional strongly unimodular almost nilpotent Lie algebra.  Then, $\mathfrak{g}$ admits complex structures if and only if it is isomorphic to one among
\begin{itemize} [leftmargin=0.5cm]
\item[] $\frs{3.3}^0 \oplus \R^3=\left(f^{26},-f^{16},0,0,0,0\right)$, \smallskip
\item[] $\mathfrak{h}_3 \oplus \mathfrak{s}_{3.3}^0=(f^{23},0,0,f^{56},-f^{46},0)$, \smallskip
\item[] $\frs{4.3}^{-\frac{1}{2},-\frac{1}{2}} \oplus \R^2=\left(f^{16},-\frac{1}{2} f^{26},-\frac{1}{2} f^{36},0,0,0\right)$, \smallskip
\item[] $\frs{4.5}^{p,-\frac{p}{2}} \oplus \R^2=\left( pf^{16},-\frac{p}{2} f^{26} + f^{36},-f^{26}-\frac{p}{2} f^{36},0,0,0 \right)$, $p>0$, \smallskip
\item[] $\frs{4.6} \oplus \R^2=(f^{23},f^{26},-f^{36},0,0,0)$, \smallskip
\item[] $\frs{4.7}\oplus \R^2=(f^{23},f^{36},-f^{26},0,0,0)$,\smallskip
\item[] $\frs{5.4}^0 \oplus \R = \left( f^{26},0,f^{46},-f^{36},0,0 \right)$, \smallskip
\item[] $\frs{5.8}^0 \oplus \R = \left( f^{26} +f^{36},-f^{16}+f^{46},f^{46},-f^{36},0,0 \right)$, \smallskip
\item[] $\frs{5.9}^{1,-1,-1} \oplus \R = \left( f^{16},f^{26},-f^{36},-f^{46},0,0 \right)$, \smallskip
\item[] $\frs{5.11}^{p,p,-p} \oplus \R = \left( pf^{16}, pf^{26}, -pf^{36} +f^{46},-f^{36} -pf^{46},0,0 \right)$, $p>0$, \smallskip
\item[] $\frs{5.13}^{p,-p,r} \oplus \R = \left( pf^{16} + f^{26}, -f^{16} +pf^{26}, -pf^{36} +rf^{46},-rf^{36} -pf^{46},0,0 \right)$, $r>0$, \smallskip
\item[] $\frs{5.16} \oplus \R=(f^{23}+f^{46},f^{36},-f^{26},0,0,0)$, \smallskip
\item[] $\frs{6.14}^{-\frac{1}{4},-\frac{1}{4}} = \left( -\frac{1}{4} f^{16} + f^{26}, -\frac{1}{4} f^{26}, -\frac{1}{4} f^{36} + f^{46}, -\frac{1}{4} f^{46}, f^{56}, 0 \right)$, \smallskip
\item[] $\frs{6.16}^{p,-4p}= \left( pf^{16} + f^{26} + f^{36}, -f^{16} +pf^{26} +f^{46}, pf^{36} + f^{46}, -f^{36} +pf^{46}, -4pf^{56}, 0 \right)$, $p<0$, \smallskip
\item[] $\frs{6.17}^{1,q,q,-2(1+q)}=\left( f^{16},f^{26}, qf^{36}, qf^{46}, -2(1+q)f^{56}, 0 \right)$, $0<|q|\leq 1$, $q \neq -1$, \smallskip
\item[] $\frs{6.18}^{1,-\frac{3}{2},-\frac{3}{2}}=\left( f^{16} + f^{26}, f^{26}, f^{36}, -\frac{3}{2} f^{46}, -\frac{3}{2} f^{56}, 0 \right)$, \smallskip
\item[] $\frs{6.19}^{p,p,q,-p-\frac{q}{2}}=\left( pf^{16}, pf^{26}, qf^{36}, -\left(p+\frac{q}{2} \right) f^{46} + f^{56}, -f^{46} -\left(p+\frac{q}{2}\right) f^{56}, 0 \right)$, $p,q \neq 0$, \smallskip
\item[] $\frs{6.20}^{p,p,-\frac{3}{2}p}= \left( pf^{16}+f^{26},pf^{26},pf^{36},-\frac{3}{2}pf^{46}+f^{56},-f^{46}-\frac{3}{2}pf^{56},0 \right)$, $p>0$, \smallskip
\item[] $\frs{6.21}^{p,q,r,-2(p+q)}=\left( pf^{16} +f^{26},-f^{16} + pf^{26}, qf^{36} + rf^{46},-rf^{36} +qf^{46}, -2(p+q)f^{56}, 0 \right)$, $|p| \geq |q|$,
\item[] \hspace{34.2em}$q \neq -p$, $r > 0$, \smallskip
\item[] $\frs{6.25}=(f^{23},f^{36},-f^{26},0,f^{46},0)$, \smallskip
\item[] $\frs{6.44}=(f^{23},f^{36},-f^{26},f^{26}+f^{56},f^{36}-f^{46},0)$, \smallskip
\item[] $\frs{6.51}^{p,0}=(f^{23},pf^{26},-pf^{36},f^{56},-f^{46},0)$, $p>0$, \smallskip
\item[] $\frs{6.52}^{0,q}=(f^{23},f^{36},-f^{26},qf^{56},-qf^{46},0)$, $q>0$, \smallskip
\item[] $\frs{6.145}^{0}=(f^{35}+f^{26},f^{45}-f^{16},f^{46},-f^{36},0,0)$, \smallskip
\item[] $\frs{6.147}^{0}=(f^{35}+f^{26},f^{45}-f^{16}+f^{36},f^{46},-f^{36},0,0)$, \smallskip
\item[] $\frs{6.152}=(f^{35}+f^{26},f^{34}-f^{16} + f^{56},f^{45},-f^{56},f^{46},0)$, \smallskip
\item[] $\frs{6.154}^0=(f^{35}+f^{26},f^{34}-f^{16},f^{45},-f^{56},f^{46},0)$, \smallskip
\item[] $\frs{6.158}=(f^{24}+f^{35},0,f^{36},0,-f^{56},0)$, \smallskip
\item[] $\frs{6.159}=(f^{24}+f^{35},0,-f^{56},0,f^{36},0)$, \smallskip
\item[] $\frs{6.162}^{1}=(f^{24}+f^{35},f^{26},f^{36},-f^{46},-f^{56},0)$, \smallskip
\item[] $\frs{6.164}^p=(f^{24}+f^{35},pf^{26},f^{56},-pf^{46},-f^{36},0)$, $p >0$, \smallskip
\item[] $\frs{6.165}^p=(f^{24}+f^{35},pf^{26}+f^{36},-f^{26}+pf^{36},-pf^{46}+f^{56},-f^{46}-pf^{56},0)$, $p>0$,\smallskip
\item[] $\frs{6.166}^p=(f^{24}+f^{35},-f^{46},-pf^{56},f^{26},pf^{36},0)$, $0<|p|\leq 1$, \smallskip
\item[] $\frs{6.167}=(f^{24}+f^{35},-f^{36},-f^{26},f^{26}+f^{56},f^{36}-f^{46},0)$.
\end{itemize}
\end{theorem}
\begin{proof}
The result concerning almost abelian Lie algebras and Lie algebras with Heisenberg-type nilradical follows from \cite{FP_gk} and \cite{FP_alm}, after properly adapting the names of the Lie algebras involved, based on the conventions in \cite{SW}. The remaining eight isomorphism classes of Lie algebras (see Table \ref{table-new}) can be studied via a case-by-case argument.

Having provided explicit examples of complex structures on $\frs{6.145}^0$, $\frs{6.147}^0$, $\frs{6.152}$ and $\frs{6.154}^0$, it suffices to prove that the remaining four Lie algebras in Table \ref{table-new} do not admit complex structures, which can be achieved via explicit computations, 
considering the generic almost complex structure $J=(J_{jk})_{j,k=1,\ldots,6}$ written as a matrix with respect to the basis $\{f_1,\ldots,f_6\}$ defining the structure equations. We then impose the equations dictated by the condition $J^2=-\text{Id}$ and $N^J=0$, showing that they lead to some contradiction.

We can consider the Lie algebra
\[
\mathfrak{g}=(f^{35}+f^{16}+\varepsilon f^{36},f^{45}-f^{26}-\varepsilon f^{46},f^{36},-f^{46},0,0), \quad \varepsilon \in \{0,1\},
\]
with the cases $\varepsilon=0$ and $\varepsilon=1$ yielding $\mathfrak{s}_{6.140}^{-1}$ and $\mathfrak{s}_{6.146}^{-1}$, respectively.
One computes
\[
f^6(N^J(f_2,f_4))=J_{62}\left( J_{52} - \varepsilon J_{62} \right).
\]
Let us first assume $J_{62} \neq 0$, obtaining $J_{52}=\varepsilon J_{62}$. Now,
\[
f^4 (N^J(f_1,f_2))=-2J_{41}J_{62}, \quad f^6(N^J(f_1,f_2))=-2J_{61}J_{62},
\]
implying $J_{41}=J_{61}=0$. From
\[
f^3(N^J(f_1,f_5))=-(J_{31})^2, \quad f^5(N^J(f_1,f_2))=-J_{51}J_{62},
\]
we deduce $J_{31}=J_{51}=0$. Finally,
$
f^6(N^J(f_1,f_6))=J_{21}J_{62}
$
forces $J_{21}=0$, but now $(J^2)_{11}=(J_{11})^2 \geq 0$, a contradiction.

We should then have $J_{62}=0$. One has
$
f^3(N^J(f_1,f_2))=-2J_{32}J_{61}.
$
We claim $J_{61}=0$. Otherwise we would have
$
f^4(N^J(f_2,f_5))=-(J_{42})^2,
$
yielding $J_{42}=0$. Then,
\[
f^6(N^J(f_2,f_6))=-J_{12}J_{61}, \quad f^6(N^J(f_2,f_3))=J_{52}J_{61}
\]
imply $J_{12}=0=J_{52}=0$, at which point $(J^2)_{22}=(J_{22})^2 \geq 0$, a contradiction similarly to the previous case.

In a similar fashion, we now claim $J_{51}=0$. Suppose that is not the case: then
$
f^3(N^J(f_1,f_3))=J_{31}J_{51}
$
implies $J_{31}=0$, and now
\begin{gather*}
f^1(N^J(f_1,f_2))=-J_{51}J_{32}, \quad f^1(N^J(f_1,f_3))=J_{51}(J_{11}-J_{33}),\\ f^5(N^J(f_1,f_3))=J_{51}(J_{51}-J_{63})
\end{gather*}
imply $J_{32}=0$, $J_{33}=J_{11}$ and $J_{63}=J_{51}$. We have
\[
f^4(N^J(f_2,f_5))=-(J_{42})^2, \quad f^4(N^J(f_1,f_3))=-J_{41}J_{51}, 
\]
which force $J_{41}=J_{42}=0$. Now,
\[
N^J(f_1,f_5)=-J_{35}J_{51}\,f_1 - (2J_{21}J_{65}+J_{45}J_{51})\,f_2 - J_{65}J_{51} \, f_5
\]
vanishes if and only if $J_{35}=J_{45}=J_{65}=0$. We now have
$
f^1(N^J(f_2,f_3))=2J_{51}J_{12},
$
forcing $J_{12}=0$, at which point
$
f^5(N^J(f_2,f_3))=2J_{52}J_{51}
$
implies $J_{52}=0$, leading to the same contradiction as before, $(J^2)_{11}=(J_{11})^2 \geq 0$.

Now, having established $J_{51}=0$, assume $J_{52} \neq 0$: then,
\[
f^1(N^J(f_1,f_2))=J_{31}J_{52}, \quad f^2(N^J(f_1,f_2))=J_{41}J_{52}
\]
imply $J_{31}=0=J_{41}=0$, at which point
$
f^5(N^J(f_1,f_6))=J_{52}J_{21}
$
forces $J_{21}=0$, leading to $(J^2)_{11}=(J_{11})^2 \geq 0$, which cannot occur. Then, we must have $J_{52}=0$.

Continuing to work in the same way, we assume $J_{41}=0$: through
\[
f^4(N^J(f_1,f_3))=-2J_{41}J_{63},
\]
this implies $J_{63}=0$, and then
$
f^2(N^J(f_1,f_3))=J_{41}J_{53}
$
forces $J_{53}=0$. Now,
\begin{alignat*}{2}
&f^5(N^J(f_1,f_6))=J_{54}J_{41}, \quad& &f^4(N^J(f_1,f_4))=-2J_{64}J_{41},\\ &f^3(N^J(f_1,f_5))=-(J_{31})^2-J_{32}J_{41}, \quad& &f^4(N^J(f_1,f_5))=-J_{41}(J_{31}+J_{42}+2J_{65})
\end{alignat*}
imply
\[
J_{54}=J_{64}=0, \quad J_{32}=-\frac{(J_{31})^2}{J_{41}}, \quad J_{42}=-J_{31}-2J_{65}.
\]
In order for
\[
(J^2)_{55}=(J_{55})^2+J_{56}J_{65}
\]
to be negative, we have to assume $J_{65} \neq 0$, so that
\[
f^3(N^J(f_2,f_5))=-4 \frac{(J_{31})^2J_{65}}{J_{41}}
\]
implies $J_{31}=0$. A contradiction is now given by
\[
f^4(N^J(f_2,f_5))=-4(J_{65})^2 \neq 0.
\]
We conclude that we have $J_{41}=0$, necessarily.

Now, we have
\[
f^3(N^J(f_1,f_5))=-(J_{31})^2, \quad f^4(N^J(f_2,f_5))=-(J_{42})^2,
\]
forcing $J_{31}=J_{42}=0$.
A contradiction is now provided by
\[
1=f^1(N^J(f_1,f_6))-(J^2)_{11}=-2(J_{11})^2-1,
\]
which never holds for any value of $J_{11}$. This concludes the proof of the non-existence of complex structures on $\mathfrak{s}_{6.140}^{-1}$ and $\mathfrak{s}_{6.146}^{-1}$.

We now consider the Lie algebra
\[
\mathfrak{g}=(f^{35}+f^{16},f^{34}-f^{26}-\varepsilon f^{46},f^{45},-f^{46},f^{56},0), \quad \varepsilon \in \{0,1\},
\]
encompassing the Lie algebras $\mathfrak{s}_{6.151}$ and $\mathfrak{s}_{6.155}^0$ for $\varepsilon=1$ and $\varepsilon=0$, respectively, and perform the same process as before, writing the generic $J=(J_{jk})_{j,k=1,\ldots,6}$ in terms of $\{f_1,\ldots,f_6\}$ and imposing $J^2=-\text{Id}$ and $N^J=0$.

We start by assuming $J_{61} \neq 0$. Then,
\[
\nj{1}{2}{5}=-2J_{52}J_{61}, \quad \nj{1}{2}{6}=-2J_{62}J_{61}
\]
force $J_{52}=J_{62}=0$, so that
\[
\nj{2}{3}{4}=(J_{42})^2
\]
implies $J_{42}=0$. Now, by
\[
\nj{1}{2}{3}=-J_{32}J_{61}, \quad \nj{2}{6}{6}=-J_{12}J_{61},
\]
we deduce $J_{12}=J_{32}=0$, but this leads to $(J^2)_{22}=(J_{22})^2 \geq 0$, a contradiction. We then must have $J_{61}=0$.

Assume $J_{62} \neq 0$. Then,
\[
\nj{1}{2}{4}=-2J_{41}J_{62}
\]
implies $J_{41}=0$, which leads to
\[
\nj{1}{3}{5}=(J_{51})^2
\]
forcing $J_{51}=0$. Since we have
\[
\nj{1}{2}{3}=-J_{31}J_{62}, \quad \nj{1}{6}{6}=J_{21}J_{62},
\]
we deduce $J_{21}=J_{31}=0$, again yielding a contradiction, since now $(J^2)_{11}=(J_{11})^2 \geq 0$. This means that we have to impose $J_{62}=0$.

Suppose we have $J_{63} \neq 0$. Then,
\[
\nj{1}{4}{6}=J_{51}J_{63}, \quad \nj{1}{5}{6}=-J_{41}J_{63}
\]
imply $J_{41}=J_{51}$, and now
$
\nj{1}{3}{3}=-J_{31}J_{63}
$
yields $J_{31}=0$. A contradiction follows from
\[
\nj{1}{3}{2}=-2J_{21}J_{63},
\]
implying $J_{21}=0$ and consequently $(J^2)_{11}=(J_{11})^2 \geq 0$. We deduce that $J_{63}=0$ holds.

We now assume $J_{64} \neq 0$, so that
\[
\nj{4}{5}{6}=2J_{64}J_{65}
\]
implies $J_{65}=0$. Now, from
\[
\nj{i}{6}{6}=J_{4i}J_{64}, \quad i=1,2,3,5,
\]
we deduce $J_{41}=J_{42}=J_{43}=J_{45}=0$.
Having
\[
\nj{1}{3}{5}=(J_{51})^2, \quad \nj{1}{5}{3}=-(J_{31})^2,
\]
we must have $J_{31}=J_{51}=0$ and now
$
\nj{1}{4}{2}=-2J_{21}J_{64}
$
forces $J_{21}=0$, yielding $(J^2)_{11}=(J_{11})^2 \geq 0$. Then, $J_{64}=0$ necessarily holds.

Now, in order for
\[
(J^2)_{66}=(J_{66})^2+J_{56}J_{65}
\]
to be negative, we must have $J_{65} \neq 0$, so that
\[
(J^2)_{6i}=J_{5i}J_{65}, \quad i=1,2,3,4,
\]
implies $J_{51}=J_{52}=J_{53}=J_{54}=0$. Now,
\[
\nj{2}{3}{4}=-(J_{42})^2,\quad \nj{2}{4}{3}=(J_{32})^2
\]
imply $J_{32}=J_{42}=0$, leading to
\[
\nj{2}{5}{1}=2J_{12}J_{65}
\]
yielding $J_{12}=0$ and the contradiction $(J^2)_{22}=(J_{22})^2 \geq 0$.
\end{proof}

\begin{remark} \label{rem_lattices}
Among the eight Lie algebras of Table \ref{table-new}, we have thus proved that exactly four of them, namely $\frs{6.145}^0$, $\frs{6.147}^0$, $\frs{6.152}$ and $\frs{6.154}^0$, admit complex structures. One can prove that the corresponding simply connected Lie groups admit cocompact lattices, hence they give rise to compact solvmanifolds. For the Lie algebras $\frs{6.145}^0$ and $\frs{6.152}$, this was proved in \cite[Proposition 8.4.9]{Bock} (we note that $\frs{6.145}^0$ is isomorphic to $\mathfrak{g}_{6.70}^{0,0}$ in \cite{Bock}) and \cite[Proposition 2.10]{FOU} (see also Remark \ref{rem_152}), respectively. To prove analogous results for the remaining two Lie algebras, we can exploit the same technique used in \cite{Bock} (see also \cite[Lemma 2.9]{FOU}), by which a lattice on the Lie group corresponding to an $n$-dimensional almost nilpotent Lie algebra $\mathfrak{n} \rtimes \text{span}\left< X \right>$ can be constructed if there exist a nonzero $t$ and a rational basis $\{Y_1,\ldots,Y_{n-1}\}$ of $\mathfrak{n}$ such that the matrix associated with $\text{exp}\left( t\,\text{ad}_X \rvert_{\mathfrak{n}} \right)$ has integer entries. Now, the Lie algebra $\frs{6.147}^0$ can be regarded as the semidirect product $\mathfrak{n}_{5.1} \rtimes \text{span}\left< X \right>$, with $X=f_6-\frac{\pi-1}{\pi} f_5$. If we fix the (rational) basis $\{f_1,\ldots,f_5\}$ for $\mathfrak{n}_{5.1}$, we have
\[
\text{exp}\left(2\pi \,\text{ad}_{X}\rvert_{\mathfrak{n}_{5.1}} \right) = \begin{pmatrix}
1 & 0 & 2 & 0 & 0 \\
0 & 1 & 0 & 2 & 0 \\
0 & 0 & 1 & 0 & 0 \\
0 & 0 & 0 & 1 & 0 \\
0 & 0 & 0 & 0 & 1
\end{pmatrix},
\] 
thus proving the claim. The result regarding $\frs{6.154}^0=\mathfrak{n}_{5.2} \rtimes \text{span}\left<f_6\right>$ easily follows from the fact that
\[
\text{exp}\left( 2 \pi \, \text{ad}_{f_6}\rvert_{\mathfrak{n}_{5.2}} \right) = \text{Id}_{\mathfrak{n}_{5.2}}.
\]
It is not known whether the Lie group corresponding with $\frs{6.152}$ admits cocompact lattices.
\end{remark}

\section{Moduli of complex structures} \label{sec_mod}

In order to obtain some of the results of the next sections, we shall examine some of the Lie algebras of the previous theorem, studying all complex structures on them up to equivalence: we say that two complex structures $J_1$ and $J_2$ on a Lie algebra $\mathfrak{g}$ are equivalent if there exists a Lie algebra automorphism $\varphi \in \text{GL}(\mathfrak{g})$ such that
\[
J_2=\varphi J_1 \varphi^{-1}.
\]
Moreover, two complex structures are equivalent if and only if they admit two respective bases of $(1,0)$-forms having the same structure equations. This is the way in which we are going to tackle this problem. In the next sections, when we study Hermitian structures $(J,g)$ on such Lie algebras, it will be enough to assume such structure equations for some basis of $(1,0)$-forms $\{\alpha^1,\alpha^2,\alpha^3\}$ and then consider the generic Hermitian metric, which, with respect to the previous frame, has fundamental form
\begin{equation} \label{om}
\omega= i (\lambda_1 \alpha^{1 \overline{1}} + \lambda_2 \alpha^{2 \overline{2}} + \lambda_3 \alpha^{3 \overline{3}}) + w_1 \alpha^{2 \overline{3}} - \overline{w}_1 \alpha^{3 \overline{2}} + w_2 \alpha^{1 \overline{3}} - \overline{w}_2 \alpha^{3 \overline{1}} + w_3 \alpha^{1 \overline{2}} - \overline{w}_3 \alpha^{2 \overline{1}},
\end{equation}
with $\lambda_1,\lambda_2,\lambda_3 \in \R_{>0}$ and $w_1,w_2,w_3 \in \C$ satisfying the positivity conditions
\[
\begin{gathered}
\lambda_1 \lambda_2>|w_3|^2, \quad \lambda_2 \lambda_3>|w_1|^2, \quad \lambda_1 \lambda_3>|w_2|^2, \\
 \lambda_1 \lambda_2 \lambda_3+2\Re(i \overline{w}_1 w_2 \overline{w}_3)>\lambda_1|w_1|^2+\lambda_2|w_2|^2+\lambda_3|w_3|^2.
\end{gathered}
\]

We start by analyzing the four Lie algebras of Theorem \ref{cpx-class} having nilradical $\mathfrak{n}$ satisfying $\dim \mathfrak{n}^1 >1$.
Similarly to the previous proof, for each Lie algebra $\mathfrak{g}$, we shall consider the generic endomorphism $J \in \mathfrak{gl}(\mathfrak{g})$, represented by some matrix $(J_{jk})_{j,k=1,\ldots,6}$ with respect to the basis $\{f_1,\ldots,f_6\}$ used in the structure equations. We shall then look at the conditions set by the requirements $N^J=0$ and $J^2=-\text{Id}_{\mathfrak{g}}$ and solve them to obtain the expression for the generic complex structure $J$, which we shall use to extract a $J$-adapted basis of the form $\{e_1=f_{j},e_2=Jf_{j},e_3=f_{k},e_4=Jf_{k},e_5=f_{l},e_6=Jf_{l}\}$ and then to write a suitable basis $\{\alpha^1,\alpha^2,\alpha^3\}$ of $(1,0)$-forms as
\[
\alpha^1=\zeta_1 (e^1+ie^2), \quad \alpha^2=\zeta_2 (e^3+ie^4), \quad \alpha^3=\zeta_3 (e^5+ie^6),
\]
where $\zeta_1,\zeta_2,\zeta_3 \in \C - \{0\}$ will be chosen so as to simplify the associated structure equations as much as possible.

\begin{proposition} \label{cpx-145_147}
For every complex structure $J$ on the Lie algebra $\mathfrak{s}_{6.145}^0$, there exists a basis of $(1,0)$-forms $\{\alpha^1,\alpha^2,\alpha^3\}$ whose structure equations are
\begin{equation} \label{145_J1}
\begin{cases}
d\alpha^1=\alpha^1 \wedge (\alpha^3 - \alpha^{\overline{3}})+e^{i\theta}\alpha^{23} + e^{-i\theta} \alpha^{2 \overline{3}} + \nu \alpha^{3 \overline{3}}, \\
d\alpha^2=\alpha^2 \wedge (\alpha^3 - \alpha^{\overline{3}}),\\
d\alpha^3=0,
\end{cases}
\end{equation}
for some $\theta \in (-\frac{\pi}{2},\frac{\pi}{2})$ and $\nu \in \{0,1\}$.

For every complex structure $J$ on the Lie algebra $\mathfrak{s}_{6.147}^0$, there exists a basis of $(1,0)$-forms $\{\alpha^1,\alpha^2,\alpha^3\}$ satisfying structure equations
\begin{equation} \label{147_J1}
\begin{cases}
d\alpha^1=\alpha^1 \wedge (\alpha^3 - \alpha^{\overline{3}})+(1+z)\alpha^{23} + (1-z) \alpha^{2 \overline{3}} + \nu \alpha^{3 \overline{3}}, \\
d\alpha^2=\alpha^2 \wedge (\alpha^3 - \alpha^{\overline{3}}),\\
d\alpha^3=0,
\end{cases}
\end{equation}
with $z \in \mathbb{C}$, $\Re(z) \neq 0$, and $\nu \in \{0,1\}$, or
\begin{equation} \label{147_J2}
\begin{cases}
d\alpha^1=\alpha^1 \wedge (\alpha^3 - \alpha^{\overline{3}})+z \, \alpha^2 \wedge (\alpha^3 - \alpha^{\overline{3}}) - \alpha^{3 \overline{2}} + \alpha^{3 \overline{3}}, \\
d\alpha^2=- \alpha^2 \wedge (\alpha^3 - \alpha^{\overline{3}}),\\
d\alpha^3=0,
\end{cases}
\end{equation}
for some $z \in \C$, or
\begin{equation} \label{147_J3}
\begin{cases}
d\alpha^1=\alpha^1 \wedge (\alpha^3 - \alpha^{\overline{3}})+x \alpha^2 \wedge (\alpha^3 - \alpha^{\overline{3}}) - \alpha^{3 \overline{2}} \\
d\alpha^2=- \alpha^2 \wedge (\alpha^3 - \alpha^{\overline{3}}),\\
d\alpha^3=0,
\end{cases}
\end{equation}
with $x \in \R_{\geq 0}$.
\end{proposition}
\begin{proof}
In order for the two parts of the statement to share part of the proof, it is easy to start by working on the Lie algebra
\begin{equation} \label{145_147}
\mathfrak{g}=(f^{35}+f^{26},f^{45}-f^{16}+\varepsilon f^{36},f^{46},-f^{36},0,0), \quad \varepsilon \in \{0,1\},
\end{equation}
yielding $\frs{6.145}^0$ for $\varepsilon=0$ and $\frs{6.147}^0$ for $\varepsilon=1$.

First,
\[
f^6(N^J(f_1,f_2))=(J_{61})^2+(J_{62})^2
\]
forces $J_{61}=J_{62}=0$, and now
\[
\nj{3}{4}{6}=(J_{63})^2+(J_{64})^2
\]
yields $J_{63}=J_{64}=0$, at which point
\[
\nj{1}{3}{5}=(J_{51})^2, \quad \nj{2}{4}{5}=(J_{52})^2
\]
imply $J_{51}=J_{52}=0$. Now, we necessarily have $J_{65} \neq 0$, as
\begin{equation} \label{Jsq66}
-1=(J^2)_{66}=J_{66}^2+J_{56}J_{65},
\end{equation}
so that
\[
\nj{3}{5}{5}=J_{65}J_{54}, \quad \nj{4}{5}{5}=-J_{65}J_{53}, \quad (J^2)_{65}=J_{65}(J_{55}+J_{66})
\]
yield $J_{53}=J_{54}=0$, $J_{66}=-J_{55}$. Moreover, by \eqref{Jsq66}, we deduce
\[
J_{56}=-\frac{(J_{66})^2+1}{J_{65}}=-\frac{(J_{55})^2+1}{J_{65}}.
\]
Moving on,
\[
\nj{1}{5}{1}=-J_{11}J_{31}-J_{12}J_{41}+J_{65}J_{12}+J_{65}J_{21}+J_{31}J_{55}
\]
implies
\[
J_{21}=\frac{1}{J_{65}}\left(J_{11}J_{31}+J_{12}J_{41}-J_{12}J_{65}-J_{31}J_{55}\right).
\]
One computes
\begin{align*}
\nj{4}{5}{1}&=-J_{11}J_{34}-J_{12}J_{44}-J_{12}J_{55}-J_{13}J_{65}+J_{24}J_{65}+J_{34}J_{55}, \\
\nj{2}{5}{1}&=-J_{11}J_{32}-J_{11}J_{65}-J_{12}J_{42}+J_{22}J_{65}+J_{32}J_{55},
\end{align*}
yielding
\begin{align*}
J_{13}&=-\frac{1}{J_{65}}(J_{11}J_{34}+J_{12}J_{44}+J_{12}J_{55}-J_{24}J_{65}-J_{34}J_{55}),\\
J_{22}&=\frac{1}{J_{65}}(J_{11}J_{32}+J_{11}J_{65}+J_{12}J_{42}-J_{32}J_{55}),
\end{align*}
at which point
\[
\nj{3}{5}{4}=-\varepsilon J_{42}J_{65}-J_{33}J_{41}-J_{33}J_{65}-J_{41}J_{55}-J_{42}J_{43}+J_{44}J_{65}
\]
forces
\[
J_{44}=\frac{1}{J_{65}}\left(\varepsilon J_{42}J_{65}+J_{33}J_{41}+J_{33}J_{65}+J_{41}J_{55}+J_{42}J_{43}\right),
\]
and now
\begin{align*}
\nj{4}{5}{4}=&-\frac{1}{J_{65}}  (\varepsilon (J_{42})^2J_{65}+J_{33}J_{41}J_{42}+J_{33}J_{42}J_{65}+J_{34}J_{41}J_{65}+J_{34}J_{65}^2+J_{41}J_{42}J_{55}\\&+(J_{42})^2J_{43}+J_{42}J_{55}J_{65}+J_{43}J_{65}^2 )
\end{align*}
implies
\begin{align*}
J_{43}=-\frac{1}{(J_{42})^2+(J_{65})^2}&\left(\varepsilon (J_{42})^2J_{65}+J_{33}J_{41}J_{42}+J_{33}J_{42}J_{65}+J_{34}J_{41}J_{65} \right.\\
&\left.+J_{34}(J_{65})^2+J_{41}J_{42}J_{55}+J_{42}J_{55}J_{65}\right).
\end{align*}
Now, notice we must have $J_{41} \neq J_{65}$, otherwise we would find
\[
\nj{1}{5}{3}=-(J_{31})^2+(J_{65})^2, \quad \nj{1}{5}{4}=-2J_{31}J_{65},
\]
which cannot both vanish, having established $J_{65} \neq 0$.
Now, we can explicitly impose the vanishing of
\begin{align*}
\nj{1}{5}{3}&=-(J_{31})^2-J_{32}J_{41}+J_{32}J_{65}+J_{41}J_{65}, \\
\nj{1}{5}{4}&=-J_{31}J_{41}-J_{31}J_{65}-J_{41}J_{42}+J_{42}J_{65},
\end{align*}
by setting
\[
J_{32}=\frac{J_{31}^2-J_{41}J_{65}}{J_{65}-J_{41}}, \quad J_{42}=\frac{J_{31}(J_{41}+J_{65})}{J_{65}-J_{41}}.
\]

Now, we claim $J_{41}=0$: assuming this is not the case, consider
\[
\nj{1}{5}{2}=\varepsilon J_{31}J_{65}-2J_{11}J_{41}+2J_{12}J_{31}+2J_{41}J_{55},
\]
yielding
\[
J_{55}=J_{11}-\frac{1}{2 J_{41}} \left(\varepsilon J_{31}J_{65}+2J_{12}J_{31}\right).
\]
Now, in order for
\[
\nj{2}{5}{4}=-\frac{2}{(J_{41}-J_{65})^2} J_{65}\left((J_{31})^2J_{41}+(J_{31})^2J_{65}+(J_{41})^3-(J_{41})^2J_{65} \right)
\]
to vanish, we must assume $J_{31} \neq \pm J_{41}$ and impose
\[
J_{65}=-\frac{J_{41}((J_{31})^2+(J_{41})^2)}{(J_{31})^2-(J_{41})^2},
\]
at which point
\[
\nj{2}{5}{3}=\frac{\left((J_{31})^2+(J_{41})^2 \right)^3}{2(J_{31})^3J_{41}}
\]
cannot vanish.

Now, having $J_{41}=0$, we can compute
\[
\nj{2}{5}{4}=-2J_{31}^2,
\]
forcing $J_{31}=0$. Now, we have
\[
(J^2)_{11}=(J_{11})^2-(J_{12}^2), \quad (J^2)_{12}=2J_{11}J_{12}, \quad (J^2)_{33}=(J_{33})^2-(J_{34})^2, \quad (J^2)_{34}=2J_{33}J_{34},
\]
from which we deduce $J_{11}=J_{33}=0$ and $(J_{12})^2=(J_{34})^2=1$.
Then,
\begin{align*}
\nj{3}{5}{1}&=-\varepsilon J_{12}J_{65} + J_{12}J_{34} + J_{14}J_{65}+J_{23}J_{65}-1, \\
\nj{5}{6}{3}&=J_{34}J_{35}-J_{45}J_{55}-J_{46}J_{65}, \\
\nj{5}{6}{4}&=J_{34}J_{45}+J_{35}J_{55}+J_{36}J_{65}
\end{align*}
force
\[
J_{23}=\varepsilon J_{12} -J_{14} + \frac{1-J_{12}J_{34}}{J_{65}}, \quad J_{46} = \frac{ J_{34}J_{35}-J_{45}J_{55}}{J_{65}}, \quad J_{36}=- \frac{J_{34}J_{45}+J_{35}J_{55}}{J_{65}}.
\]

At this point, let us set $\varepsilon=0$ and continue working on the Lie algebra $\frs{6.145}^0$. First, recalling $(J_{12})^2=(J_{34})^2=1$,
\[
\nj{4}{5}{2}=2(J_{12}J_{34}-1)
\]
forces $J_{12}=J_{34}=\delta \in \{-1,1\}$. We now have
\[
\nj{4}{6}{2}=2 \delta J_{24},
\]
yielding $J_{24}=0$, and
\[
\nj{5}{6}{1}=\delta J_{15} + J_{14}J_{35} - J_{25}J_{55} - J_{26}J_{65},
\]
from which we obtain
\[
J_{26}=\frac{1}{J_{65}} \left( \delta J_{15} + J_{14}J_{35} - J_{25}J_{55} \right).
\]
Lastly,
\[
\nj{4}{6}{1}=2\delta J_{14}
\]
forces $J_{14}=0$, so that, in order for
\[
(J^2)_{26}=-\delta J_{15}J_{55} - \delta J_{16}J_{65} - J_{25}
\]
to vanish, we can set
\[
J_{25}=- \delta (J_{15}J_{55} + J_{16}J_{65}).
\]
$J$ now defines a complex structure, since an explicit computation shows that we have $J^2=-\text{Id}$ and $N^J=0$.
One can then see that $\{e_1=f_1,e_2=Jf_1,e_3=f_3,e_4=Jf_3,e_5=f_5,e_6=Jf_5\}$ defines a new basis for $\frs{6.145}^0$. If $(J_{35})^2+(J_{45})^2 \neq 0$, we can consider the basis $\{\alpha^1,\alpha^2,\alpha^3\}$ of $(1,0)$-forms provided by
\[
\alpha^1=- \frac{(J_{65})^2}{2(\delta J_{45}+i J_{35})}(e^1 + ie^2), \quad \alpha^2=  \frac{\delta J_{65} (1+(J_{55})^2)^{\frac{1}{2}}}{2\left( \delta J_{45}+i J_{35} \right)}(e^3+ie^4), \quad \alpha^3=-\frac{\delta J_{65}}{2} (e^5+ie^6),
\]
satisfying the structure equations \eqref{145_J1}, with $\nu=1$ and
\[
e^{i\theta}= \frac{1-iJ_{55}}{1+(J_{55})^2}.
\]
Instead, if $J_{35}=J_{45}=0$, we can set
\[
\alpha^1=f^1+iJf^1,\quad \alpha^2= - \frac{1}{J_{65}} \delta (1+(J_{55})^2)^{\frac{1}{2}}(f^3+iJf^3), \quad \alpha^3=-\frac{\delta J_{65}}{2} (f^5+iJf^5),
\]
obtaining \eqref{145_J1} with $\nu=0$ and $e^{i\theta}$ as before.

We now turn our attention to the Lie algebra $\frs{6.147}^0$, setting $\epsilon=1$ in \eqref{145_147}. Picking up from where we left earlier, we recall that we have $(J_{12})^2=(J_{34})^2=1$ and we start by assuming $J_{34}=J_{12}=\delta \in \{-1,1\}$. Then,
\[
(J^2)_{14}=2\delta J_{24}, \quad \nj{4}{6}{1}=2 \delta J_{14} -1
\]
force $J_{24}=0$, $J_{14}=\frac{\delta}{2}$. Finally, from
\[
(J^2)_{15}=\delta J_{25}+\frac{\delta}{2} J_{45}+J_{15}J_{55}+J_{16}J_{65}, \quad \nj{5}{6}{1}=\delta J_{15}-\frac{\delta}{2}  J_{35}-J_{25}J_{55}-J_{36}J_{65},
\]
we deduce
\[
J_{16}=-\frac{1}{2 J_{65}} \left(2 \delta J_{25}+\delta J_{45}+2 J_{15}J_{55}\right), \quad J_{26}=\frac{1}{2 J_{65}} \left(2 \delta J_{15}-\delta J_{35} - 2J_{25}J_{55} \right).
\]
We now have $(J^2)=-\text{Id}$ and $N^J=0$. A new basis for $\frs{6.147}^0$ is provided by $\{e_1=f_1,e_2=Jf_1,e_3=f_3,e_4=Jf_3,e_5=f_5,e_6=Jf_5\}$, as one can check directly. In the case $(J_{35})^2+(J_{45})^2 \neq 0$, the basis $\{\alpha^1,\alpha^2,\alpha^3\}$ for $(1,0)$-forms defined by
\[
\alpha^1=-\frac{(J_{65})^2}{2(\delta J_{45} + 2i J_{35})} (e^1+ie^2), \quad
\alpha^2=\frac{J_{65}}{2(\delta J_{45} + 2i J_{35})} (e^3+ie^4), \quad
\alpha^3=\frac{\delta}{2} J_{65} (e^5 + i e^6)
\]
yields \eqref{147_J1}, with $z=-\frac{\delta}{2} (J_{65}-iJ_{55})$ and $\nu=1$, while, when $J_{35}=J_{45}=0$, we can take
\[
\alpha^1=f^1 + iJf^1, \quad \alpha^2= \frac{\delta}{J_{65}} (f^3+iJf^3), \quad \alpha^3= \frac{\delta}{2} J_{65} (f^5 + i Jf^5),
\]
giving \eqref{147_J1}, with $z$ as before and $\nu=0$.

We can now assume $J_{34}=-J_{12}=\delta \in \{-1,1\}$, instead. Having
\[
\nj{3}{5}{2}=4\delta J_{55},
\]
we must have $J_{55}=0$. Now,
\[
\nj{3}{6}{2}=-1-2\frac{\delta}{J_{65}}
\]
forces $J_{65}=-2\delta$. Finally, by
\[
(J^2)_{15}=-2\delta J_{16}+\delta J_{25} + J_{14}J_{45} + J_{24}J_{35}, \quad (J^2)_{16}=\frac{1}{2}(J_{14}J_{35}-J_{24}J_{45}+\delta J_{15})+\delta J_{26},
\]
we have
\[
J_{16}=\frac{1}{2} J_{25} + \frac{\delta}{2} (J_{14}J_{45}+J_{24}J_{35}), \quad J_{26}=\frac{1}{2} J_{15} + \frac{\delta}{2} (J_{14}J_{35}-J_{24}J_{45}).
\]
Again, we now have $(J^2)=-\text{Id}$, $N^J=0$ and $\{e_1=f_1,e_2=Jf_1,e_3=f_3,e_4=Jf_3,e_5=f_5,e_6=Jf_5\}$ is a basis for $\frs{6.147}^0$.

If we have $(J_{35})^2+(J_{45})^2 \neq 0$,
\[
\alpha^1=-\frac{2}{\delta J_{45}+i J_{35}} (e^1+ie^2), \quad \alpha^2=-\frac{2}{\delta J_{45}-i J_{35}} (e^3+ie^4), \quad \alpha^3=-e^5-ie^6
\]
define a basis of $(1,0)$-forms satisfying \eqref{147_J2} with
\[
z=\frac{(\delta J_{45} - iJ_{35})^2(-1+2\delta J_{14} - 2i J_{13})}{2((J_{35})^2+(J_{45})^2)}.
\]

In the case $J_{35}=J_{45}=0$, we can consider the basis $\{\alpha^1,\alpha^2,\alpha^3\}$ of $(1,0)$-forms defined by
\[
\alpha^1=e^{-i \frac{\theta}{2}} (e^1+ie^2), \quad \alpha^2=e^{i \frac{\theta}{2}} (e^3+ie^4), \quad \alpha^3=-e^5-ie^6,
\]
with
\[
\theta=
\begin{cases}
\text{arg}(-1+2\delta J_{14}-2i J_{13}), &-1+2\delta J_{14}-2i J_{13} \neq 0,\\
0, &-1+2\delta J_{14}-2i J_{13} = 0,
\end{cases}
\]
yielding \eqref{147_J3}, with
\[
x=\frac{1}{2} \; | -1+2\delta J_{14}-2i J_{13} |. \qedhere
\]
\end{proof}

\begin{proposition} \label{cpx-152_154}
For every complex structure $J$ on the Lie algebra $\mathfrak{s}_{6.152}$, there exists a basis of $(1,0)$-forms $\{\alpha^1,\alpha^2,\alpha^3\}$ whose structure equations are
\begin{equation} \label{152_J1}
\begin{cases}
d\alpha^1=\alpha^{12} - \Re(z_2) \alpha^{13} - \alpha^{1 \overline{2}} + \Re(z_2)\alpha^{1 \overline{3}} + z_1 \alpha^{23} + z_2 \alpha^{2 \overline{2}} - (\Re(z_2)z_2 + i \delta) \alpha^{2 \overline{3}} \\
\hspace{2.8em} +(z_1-\Re(z_2)z_2) \alpha^{3 \overline{2}} + \left( \Re(z_2)^2 z_2 - \Re(z_2)z_1 + \frac{i}{2} \delta \overline{z}_2 \right) \alpha^{3 \overline{3}}, \\
d\alpha^2=\Re(z_2) \alpha^{23} + \Re(z_2) \alpha^{3 \overline{2}} - \left( \Re(z_2)^2 + \frac{i}{2} \delta \right) \alpha^{3 \overline{3}}, \\
d\alpha^3=\alpha^{23} + \alpha^{3 \overline{2}} - \Re(z_2) \alpha^{3 \overline{3}},
\end{cases}
\end{equation}
for some $z_1,z_2 \in \C$, $\delta \in \{-1,1\}$,
or
\begin{equation} \label{152_J2}
\begin{cases}
d\alpha^1=\alpha^{12}-\Re(z_1)\alpha^{13} - \alpha^{1\overline{2}} + \Re(z_1) \alpha^{1\overline{3}} + \frac{1}{2\Im(z_2)^2}(|z_1|^2-\delta \Im(z_2)(z_2+i))\alpha^{23} \\
\hspace{2.8em} + \frac{z_1}{\Im(z_2)^2} \alpha^{2\overline{2}} - \frac{\Re(z_1)z_1}{\Im(z_2)^2} \alpha^{2 \overline{3}} - \frac{1}{2 \Im(z_2)^2}(z_1^2+\delta \Im(z_2)(z_2-i)) \alpha^{3\overline{2}} \\
\hspace{2.8em} + \frac{1}{2 \Im(z_2)^2} \left( \Re(z_1)z_1^2+ \delta \Re(z_1)\Im(z_2) z_2 - \delta \Im(z_1)\Im(z_2)\right) \alpha^{3 \overline{3}}, \\
d\alpha^2=-\Re(z_1) \alpha^{23} - \Re(z_1) \alpha^{3 \overline{2}} + \left(\Re(z_1)^2+ \frac{i}{2} \delta \Im(z_2) \right) \alpha^{3 \overline{3}}, \\
d\alpha^3= - \alpha^{23} - \alpha^{3 \overline{2}} + \Re(z_1) \alpha^{3 \overline{3}},
\end{cases}
\end{equation}
for some $z_1,z_2 \in \C$, $\Im(z_2) \neq 0$, $\delta \in \{-1,1\}$.

For every complex structure $J$ on the Lie algebra $\mathfrak{s}_{6.154}^0$, there exists a basis of $(1,0)$-forms $\{\alpha^1,\alpha^2,\alpha^3\}$ whose structure equations are
\begin{equation} \label{154_J}
\begin{cases}
d\alpha^1=\alpha^{12}-\Re(z)\alpha^{13} - \alpha^{1\overline{2}} + \Re(z) \alpha^{1\overline{3}} + \frac{1}{2x^2}(|z|^2-x(y+i))\alpha^{23} \\
\hspace{2.8em} + \frac{z}{x^2} \alpha^{2\overline{2}} - \frac{\Re(z)z}{x^2} \alpha^{2 \overline{3}} - \frac{1}{2 x^2}(z^2+x(y-i)) \alpha^{3\overline{2}} \\
\hspace{2.8em} + \frac{1}{2 x^2} \left( \Re(z)z^2+ xy\Re(z) -x\Im(z)\right) \alpha^{3 \overline{3}}, \\
d\alpha^2=-\Re(z) \alpha^{23} - \Re(z) \alpha^{3 \overline{2}} + \left(\Re(z_1)^2+ \frac{i}{2}x \right) \alpha^{3 \overline{3}}, \\
d\alpha^3= - \alpha^{23} - \alpha^{3 \overline{2}} + \Re(z) \alpha^{3 \overline{3}},
\end{cases}
\end{equation}
for some $z \in \C$, $x,y \in \R$, $x \neq 0$.
\end{proposition}
\begin{proof}
As done in the proof of the previous theorem, it is convenient to study the Lie algebras $\frs{6.152}$ and $\frs{6.154}^0$ together, by considering the Lie algebra
\[
(f^{35}+f^{26},f^{34}-f^{16} +\varepsilon f^{56},f^{45},-f^{56},f^{46},0), \quad \varepsilon \in \{0,1\}.
\]
We then write the matrix $J_{ij}$ representing the generic $J$ and impose conditions $J^2=-\text{Id}$ and $N^J=0$.

First,
\[
\nj{1}{2}{6}=(J_{61})^2+(J_{62})^2,
\]
forces $J_{61}=J_{62}=0$. From
\[
\nj{4}{5}{6}=-\left((J_{64})^2+(J_{65})^2+J_{63}(J_{44}+J_{55})\right), \quad (J^2)_{66}=J_{36}J_{63}+J_{46}J_{64}+J_{56}J_{65}+(J_{66})^2
\]
we deduce $J_{63} \neq 0$, so that
\[
\nj{k}{4}{6}=J_{5k}J_{63}, \quad \nj{k}{5}{6}=-J_{4k}J_{63}, \quad k \in {1,2},
\]
yield $J_{41}=J_{51}=J_{42}=J_{52}=0$. In the same way,
\[
\nj{2}{3}{3}=-J_{31}J_{63}, \quad \nj{3}{4}{6}=J_{63}(J_{53}+J_{65}), \quad \nj{3}{5}{6}=-J_{63}(J_{43}+J_{64})
\]
force $J_{31}=0$, $J_{65}=-J_{53}$, $J_{64}=-J_{43}$ and
\[
\nj{2}{4}{3}=-(J_{32})^2,\quad \nj{1}{3}{1}=J_{63}(J_{12}+J_{21}), \quad \nj{1}{3}{2}=J_{63}(J_{22}-J_{11})
\]
imply
$J_{32}=0$, $J_{21}=-J_{12}$, $J_{22}=J_{11}$.
Having
\[
\nj{1}{6}{1}=2J_{11}J_{12}, \quad \nj{1}{6}{2}=(J_{11})^2-(J_{12})^2+1,
\]
we must have $J_{11}=0$ and $J_{12}=\delta \in \{-1,1\}$.
We continue by computing
\begin{align*}
\nj{3}{4}{4}&=2J_{43}J_{53}+J_{45}J_{63}+J_{54}J_{63}, \\ \nj{3}{5}{4}&=-(J_{43})^2+(J_{53})^2 - J_{44}J_{63}+J_{55}J_{63}, \\
\nj{3}{6}{4}&=-J_{43}J_{45}+J_{44}J_{53}-J_{53}J_{66}+J_{56}J_{63},
\end{align*}
whose vanishing we can impose by setting
\begin{align*}
J_{54}&=-J_{45}-2 \frac{J_{43}J_{53}}{J_{63}}, \\
J_{55}&=J_{44} +\frac{(J_{43})^2-(J_{53})^2}{J_{63}}, \\
J_{56}&=\frac{1}{J_{63}} \left( J_{43}J_{45}-J_{44}J_{53}+J_{53}J_{66} \right).
\end{align*}
Then, having
\begin{align*}
\nj{3}{4}{3}&=J_{33}J_{53}-J_{43}J_{45}-J_{44}J_{53}+J_{35}J_{63}-2 \frac{(J_{43})^2J_{53}}{J_{63}}, \\ \nj{4}{5}{6}&=-2\left((J_{43})^2+J_{44}J_{63} \right),
\end{align*}
we can set
\[
J_{35}=-\frac{1}{J_{63}}\left(J_{33}J_{53}-J_{43}J_{45}-J_{44}J_{53} \right) + 2 \frac{(J_{43})^2J_{53}}{(J_{63})^2}, \quad J_{44}=-\frac{(J_{43}^2)}{J_{63}}.
\]
Now, we have
\begin{align*}
\nj{3}{5}{3}&=-J_{33}J_{43}-J_{34}J_{63}-J_{45}J_{63}-\frac{J_{43}(J_{53})^2}{J_{63}}, \\
(J^2)_{23}&=-\delta J_{13}+J_{23}J_{33}+J_{24}J_{43}+J_{25}J_{53}+J_{26}J_{63}, \\
(J^2)_{63}&=-(J_{43})^2-(J_{53})^2+J_{33}J_{63}+J_{63}J_{66},
\end{align*}
yielding
\begin{align*}
J_{34}&=-\frac{1}{J_{63}} \left( J_{33}J_{43}+J_{45}J_{53} \right), \\
J_{26}&=\frac{1}{J_{63}} \left( J_{12}J_{13}-J_{23}J_{33}-J_{24}J_{43}-J_{25}J_{53} \right), \\
J_{66}&=-J_{33}- \frac{ (J_{43})^2+(J_{53})^2 }{J_{63}}.
\end{align*}
Using $J^2=-\text{Id}$, we can look at
\begin{align*}
(J^2)_{13}&=J_{12}J_{23}+J_{13}J_{33}+J_{14}J_{43}+J_{15}J_{53}+J_{16}J_{63}, \\
(J^2)_{33}&=(J_{33})^2+J_{36}J_{63} - \frac{J_{33}}{J_{63}} \left((J_{43})^2+(J_{53})^2 \right), \\
(J^2)_{43}&=J_{33}J_{43}+J_{45}J_{53} + J_{46}J_{63} - \frac{(J_{43})^3}{J_{63}},
\end{align*}
obtaining
\begin{align*}
J_{16}&=- \frac{1}{J_{63}} \left( \delta J_{23}+J_{13}J_{33}+J_{14}J_{43}+J_{15}J_{53} \right), \\
J_{36}&=- \frac{1}{J_{63}} \left( (J_{33})^2+1 \right) - \frac{J_{33}}{(J_{63})^2} \left((J_{43})^2+(J_{53})^2 \right), \\
J_{46}&=- \frac{1}{J_{63}}(J_{33}J_{43}+J_{45}J_{53}) + \frac{(J_{43})^3}{(J_{63})^2}.
\end{align*}
We can now consider
\[
\nj{4}{5}{3}=-\frac{1}{(J_{63})^2} (J_{43}+J_{45}J_{63}+J_{63})(J_{43}J_{53}+J_{45}J_{63}-J_{63}),
\]
which we can impose to equal zero by introducing $\eta \in \{-1,1\}$ and setting
\[
J_{45}=\eta - \frac{J_{43}J_{53}}{J_{63}}.
\]
Now,
\[
(J^2)_{15}=\delta J_{25} + \eta J_{14} + \frac{1}{J_{63}} ( \delta J_{23}J_{53} + \eta J_{13}J_{43})
\]
yields
\[
J_{25}=\delta \eta J_{14} - \frac{1}{J_{63}}(J_{23}J_{53} + \delta \eta J_{13}J_{43}).
\]
Since $\delta,\eta \in \{-1,1\}$, we can either have $\eta=\delta$ or $\eta=-\delta$. We start by discussing the former case, considering now
\[
\nj{3}{4}{2}=-2+ \delta \varepsilon J_{63},
\]
which forces $\varepsilon = 1$, meaning that we are now working on the Lie algebra $\frs{6.152}$, and
\[
J_{63}=2 \delta,
\]
while
\[
\nj{3}{5}{2}=2\delta J_{33}- \frac{\delta}{2} ((J_{43})^2+(J_{53})^2)
\]
implies
\[
J_{33}=\frac{\delta}{4}((J_{43})^2+(J_{53})^2).
\]
We have now obtained both $N^J=0$ and $J^2=-\text{Id}$. A basis of $(1,0)$-forms is now provided by
\[
\alpha^1=f^1+iJf^1, \quad \alpha^2=f^3+iJf^3, \quad \alpha^3=f^4+iJf^4,
\]
whose associated structure equations are of the form \eqref{152_J1}, with
\begin{align*}
z_1&=\delta J_{24} - \frac{1}{8} ((J_{43})^2 + (J_{53})^2) + \frac{1}{2} J_{23}J_{43} +\frac{i}{4} \left(  2 \delta J_{13}J_{43} + \delta J_{43}J_{53} + 2 \delta + 4 J_{14}  \right) , \\
z_2&= \frac{1}{2} ( \delta J_{43} + i J_{53} ).
\end{align*}

We now go back and tackle the remaining case $\eta=-\delta$ and consider
\[
\nj{3}{4}{1}=-2J_{23}J_{43}-2J_{24}J_{63}+ \frac{\delta}{J_{63}}((J_{43})^2-(J_{53})^2),
\]
yielding
\[
J_{24}= - \frac{J_{23}J_{24}}{J_{63}}+ \frac{\delta}{2 (J_{63})^2}((J_{43})^2-(J_{53})^2),
\]
while
\[
\nj{3}{4}{2}=-\delta \varepsilon J_{63} + 2J_{13}J_{43}+2J_{14}J_{63}- 2 \frac{\delta J_{43} J_{53}}{J_{63}}
\]
forces
\[
J_{14}= \frac{1}{2} \delta \varepsilon -\frac{J_{13}J_{43}}{J_{63}} + \frac{\delta J_{43}J_{53}}{(J_{63})^2}.
\]
This is enough to guarantee $N^J=0$ and $J^2=-\text{Id}$. A basis of $(1,0)$-forms is given by
\[
\alpha^1=f^1+iJf^1, \quad \alpha^2=\frac{\delta}{2}J_{63}(f^3+iJf^3), \quad \alpha^3=f^4+iJf^4.
\]
In the case $\varepsilon=1$, we get \eqref{152_J2}, with
\[
z_1=\frac{\delta}{2} J_{43} + \frac{i}{2} J_{53}, \quad z_2=J_{33} + \frac{i}{2} \delta J_{63}.
\]
Instead, if $\varepsilon=0$, we obtain \eqref{154_J}, with
\[
z=\frac{\delta}{2} J_{43} + \frac{i}{2} J_{53}, \quad x= \frac{J_{63}}{2}, \quad y=J_{33}. \qedhere
\]
\end{proof}

\begin{remark}
An explicit computation shows that the family \eqref{152_J1} of complex structures on $\frs{6.152}$ satisfies $d \alpha^{123}=0$, so that it corresponds to the complex structure described in \cite[Proposition 3.9]{FOU}, while the complex structures described by \eqref{152_J2} do not admit a nonzero closed $(3,0)$-form. As a consequence, the non-existence results of the next sections regarding Hermitian metrics with respect to \eqref{152_J1} also follow from \cite{FOU}.
\end{remark}

Now, we are interested in obtaining analogous results for complex structures $J$ satisfying $J\mathfrak{n}^1 \not\subset \mathfrak{n}$ on Lie algebras with Heisenberg-type nilradical $\mathfrak{n}$. As we shall see in the next sections, it will not be necessary to study complex structures satisfying $J\mathfrak{n}^1 \subset \mathfrak{n}$ in this way, as there will be workarounds based on the use of algebraic data.

Let $\mathfrak{g}$ be a $2n$-dimensional strongly unimodular almost nilpotent Lie algebra with Heisenberg-type nilradical.
Following \cite{FP_alm}, we can consider the map
\begin{align*}
\varphi \colon \mathfrak{g}/\mathfrak{n} &\to \mathfrak{gl}(\mathfrak{n}/\mathfrak{n}^1), \\
X + \mathfrak{n} &\mapsto \pi(\text{ad}_X \rvert_{\mathfrak{n}}), 
\end{align*}
with $\pi(\text{ad}_X \rvert_{\mathfrak{n}})(Y+\mathfrak{n}^1) \coloneqq [X,Y] + \mathfrak{n}^1$, $Y \in \mathfrak{n}$. Since $\dim \mathfrak{g}/\mathfrak{n}=1$, $\varphi$ induces a well-defined endomorphism $A_{\mathfrak{g}}$ of $\mathfrak{n}/\mathfrak{n}^1$, up to non-zero rescalings.
Now, given a complex structure $J$ on $\mathfrak{g}$ satisfying $J\mathfrak{n}^1 \not\subset \mathfrak{n}$, it is possible to find a $J$-adapted basis for $\mathfrak{g}$, $\{e_1,\ldots,e_{2n}\}$, such that $\mathfrak{n}^1=\text{span}\left<e_1\right>$, $\mathfrak{k} \coloneqq \mathfrak{n} \cap J \mathfrak{n} = \text{span}\left< e_2,\ldots,e_{2n-1}\right>$ and $e_{2n}=Je_1 \notin \mathfrak{n}$.
By \cite{FP_alm}, with respect to such a basis, we have
\begin{gather*}
\text{ad}_{e_{2n}}\rvert_{\mathfrak{n}}=\text{diag}(0,A), \quad A \in \mathfrak{gl}(\mathfrak{n}_1),\;\operatorname{tr} A=0,\;[A,J\rvert_{\mathfrak{k}}]=0, \\
de^1=\eta \in \Lambda^{1,1} \mathfrak{k}^*-\{0\},
\end{gather*}
which encode the whole Lie bracket of $\mathfrak{g}$. Observe that $A$ is a representative of $A_{\mathfrak{g}}$, under the isomorphism $\mathfrak{n}/\mathfrak{n^1} \cong \mathfrak{k}$.  Now, the condition $[A,J\rvert_{\mathfrak{k}}]=0$ implies that, up to a change of basis of $\mathfrak{k}$, we can assume that the matrices representing $A$ and $J\rvert_{\mathfrak{k}}$ are in their respective real Jordan form, namely
\begin{equation} \label{A_J_Jordan}
A=\text{diag}(C^{k_1}_{a_1,b_1},\ldots,C^{k_l}_{a_l,b_l}), \quad
J=\text{diag}\left( J^{k_1}_{\varepsilon_1},\ldots,J^{k_l}_{\varepsilon_l}\right),
\end{equation}
with $k_j \in \mathbb{N}$, $a_j,b_j \in \R$,  $\varepsilon_j \in \{-1,1\}$,  $j=1,\ldots,l$,
where
\[
J^k_{\varepsilon} \coloneqq \text{diag}\left( \begin{pmatrix} 0 & -\varepsilon \\ \varepsilon & 0 \end{pmatrix}, \ldots, \begin{pmatrix} 0 & -\varepsilon \\ \varepsilon & 0 \end{pmatrix} \right) \in \mathfrak{gl}_{2k}.
\]
and
\begin{equation} \label{C^k}
C^k_{a,b}\coloneqq
    \begin{pmatrix}
    a & -b & 1  & 0  &        &         &   &    &    &                                                 \\
    b & a  & 0  & 1  &        &         &   &    &    &                                               \\
      &    & a  & -b &   1    &   0    &   &    &    &                                                \\
      &    & b  & a  &   0    &   1    &   &    &    &                                              \\
      &    &    &    & \ddots & \ddots  &   &    &    &                                                \\
      &    &    &    &        &         & a & -b & 1  & 0                 \\
      &    &    &    &        &         & b & a  & 0  & 1                                             \\
      &    &    &    &        &         &   &    & a  & -b             \\
      &    &    &    &        &         &   &    & b  & a 
  \end{pmatrix} \in \mathfrak{gl}_{2k}, \quad k \in \mathbb{N}, \quad a,b \in \R,
\end{equation}
is a real Jordan block corresponding to the conjugate pair of eigenvalues $a \pm ib$ of $A$, with $\text{Spec}(A)=\{a_1 \pm ib_1, \ldots, a_l \pm ib_l\}$ and $k_1+\ldots+k_l=n-1$. Notice that $C^k_{a,b}$ is similar to $C^k_{a,-b}$ via the change of basis changing the sign of every even-indexed basis vector. We assume our choice of basis yielding \eqref{A_J_Jordan} to have been made by fixing one of the two possible signs for each $b_j$, at the expense of allowing $J$ to feature some diagonal blocks with an inverted sign. To obtain a $J$-adapted basis for $\mathfrak{k}$ of the form $\{e_2,e_3=Je_2,\ldots,e_{2n-2},e_{2n-1}=Je_{2n-2}\}$ we may have to fix the some of some $b_j$'s, based on the value of the corresponding $\varepsilon_j$ in the expression of $J$. By doing so, we now have
\[
A=\text{diag}(C^{k_1}_{a_1,\varepsilon_1b_1},\ldots,C^{k_l}_{a_l,\varepsilon_lb_l}).
\]
Moreover, by rescaling $e_{2n}$ (and $e_1$ accordingly, in order to preserve $Je_1=e_{2n}$), we can assume the eigenvalues of $A$ to be properly normalized in order to match an arbitrary choice of rescaling of $A_{\mathfrak{g}}$, as can be read in the structure equations of each Lie algebra.

Having simplified the expression of $A$, we can then consider the $2$-form $\eta$, writing the generic $(1,1)$-form on $\mathfrak{k}$ with respect to the basis $\{e_2,\ldots,e_{2n-1}\}$, for example
\[
\eta= x_1 e^{23} +x_2 e^{45} + x_3 (e^{24}+e^{35}) + x_4 (e^{25}-e^{34}), \quad x_1,x_2,x_3,x_4 \in \R,
\]
in dimension six, and impose $A^* \eta =0$ to ensure that the Jacobi identity holds (see \cite{FP_alm}) and that $\eta$ has the right rank for the Lie algebra $\mathfrak{g}$ we are considering, by imposing the right power of $\eta$ to be the last non-vanishing one ($\mathfrak{n} \cong \mathfrak{h}_{2k+1} \oplus \R^{2(n-k)}$ forces $\eta^{k} \neq 0$, $\eta^{k+1}=0$). This, together with some further requirements dictated by the isomorphism class of $\mathfrak{g}$ (for example, focusing on the center of $\mathfrak{n}$, see the proof below), yield conditions on the coefficients of $\eta$.

Focusing on the six-dimensional case, we then consider the basis of $(1,0)$-forms given by
\[
\alpha^1=\zeta_1 (e^1+ie^{6}), \quad \alpha^2=\zeta_2 (e^2+ie^3), \quad \alpha^3=\zeta_3 (e^4+ie^5),
\]
where $\zeta_1,\zeta_2,\zeta_3 \in \C - \{0\}$ are chosen as to simplify the associated structure equations as much as possible.

\begin{proposition} \label{heis_cpx}
Let $\mathfrak{g}$ be a six-dimensional strongly unimodular almost nilpotent Lie algebra with nilradical $\mathfrak{n}$ satisfying $\dim \mathfrak{n}^1=1$. Then, every complex structure $J$ on $\mathfrak{g}$ satisfying $J \mathfrak{n}^1 \not\subset \mathfrak{n}$ admits a basis of $(1,0)$-forms $\{\alpha^1,\alpha^2,\alpha^3\}$ whose set of structure equations appears in Table \ref{table-cpxheis}, based on the isomorphism class of $\mathfrak{g}$.
\end{proposition}
\begin{proof}
We only prove the result for the Lie algebra $\mathfrak{g}=\mathfrak{s}_{6.52}^{0,q}$, $q>0$,  as an explicit example of the general procedure described earlier.

Using the structure equations of the Lie algebra, we notice that, under the isomorphisms $\mathfrak{n}/\mathfrak{n^1} \cong \R \left< f_2,f_3,f_4,f_5 \right>$ and $\mathfrak{g}/\mathfrak{n} \cong \R \left<f_6\right>$, the endomorphism $A_{\mathfrak{g}}$ is diagonalizable with spectrum a uniform rescaling of $\{\pm i, \pm ib\}$. Let $J$ be a complex structure on $\mathfrak{g}$. Then, by \cite{FP_alm}, $J$ satisfies $J \mathfrak{n}^1 \not\subset \mathfrak{n}$ and, following the construction above, we can find a basis $\{\tilde{e}_1,\ldots,\tilde{e}_6\}$ of $\mathfrak{g}$ such that $\mathfrak{n}^1 = \R \left<\tilde{e}_1\right>$, $\tilde{e}_6=J\tilde{e}_1$, $\mathfrak{k} \coloneqq \mathfrak{n} \cap J\mathfrak{n}= \text{span}\left<\tilde{e}_2,\tilde{e}_3,\tilde{e}_4,\tilde{e}_5\right>$ and the matrix $A$ representing $\text{ad}_{\tilde{e}_6}\rvert_{\mathfrak{k}}$ and the one representing $J\rvert_{\mathfrak{k}}$ are of the form
\[
A=\text{diag}\left(\begin{pmatrix} 0 & 1 \\ -1 & 0 \end{pmatrix},\begin{pmatrix} 0 & q \\ -q & 0 \end{pmatrix}\right), \quad J\rvert_{\mathfrak{k}}=\text{diag}\left(\begin{pmatrix} 0 & -\varepsilon_1 \\ \varepsilon_1 & 0 \end{pmatrix},\begin{pmatrix} 0 & -\varepsilon_2 \\ \varepsilon_2 & 0 \end{pmatrix}\right), \quad \varepsilon_1,\varepsilon_2 \in \{-1,1\}.
\]
With respect to the new basis $\{e_1=\varepsilon_1 \tilde{e}_1,e_2=\tilde{e}_2,e_3=\varepsilon_1 \tilde{e}_3,e_4=\tilde{e}_4,e_5=\varepsilon_2 \tilde{e}_5,e_6=\varepsilon_1 \tilde{e}_6\}$, the matrix $A$ is now of the form
\[
A=\text{diag}\left(\begin{pmatrix} 0 & 1 \\ -1 & 0 \end{pmatrix},\begin{pmatrix} 0 & \varepsilon q \\ -\varepsilon q & 0 \end{pmatrix}\right),
\]
with $\varepsilon=\varepsilon_1\varepsilon_2$ and $J$ satisfying $Je_1=e_6$, $Je_2=e_3$, $Je_4=e_5$. Now, $de^1=\eta$, with $\eta$ of the form
\[
\eta= x_1 e^{23} +x_2 e^{45} + x_3 (e^{24}+e^{35}) + x_4 (e^{25}-e^{34}), \quad x_1,x_2,x_3,x_4 \in \R,
\]
We compute
\[
A^*\eta=(\varepsilon q-1) \left( x_3 (e^{25}-e^{34}) - x_4 (e^{24}+e^{35}) \right).
\]
Now, if $\varepsilon q \neq 1$, we must set $x_3=x_4=0$ and be left with $\eta= x_1 e^{23} +x_2 e^{45}$, with $x_1x_2 =0$ in order for $\eta^2$ to vanish, as imposed by $\mathfrak{n} \cong \mathfrak{h}_3 \oplus \R^2$. 
Instead, when $\varepsilon q =1$, we obtain $A\rvert_{\mathfrak{k}}=J\rvert_{\mathfrak{k}}$, implying that, after a proper diagonalization, we can assume $\eta$ to be in diagonal form, namely $x_3=x_4=0$, obtaining the same result as the previous case.
Now, knowing that the eigenvectors of $A$ with eigenvalue $\pm i q$ lie in the center of $\mathfrak{n}$, we can set $x_2=0$. Then, the basis of $(1,0)$-forms
\[
\alpha^1=\frac{1}{2}(e^1+ie^6), \quad \alpha^2=\frac{|x_1|^{\frac{1}{2}}}{2}(e^2+ie^3), \quad \alpha^3=\frac{1}{2}(e^4+ie^5),
\]
yields the structure equations in Table \ref{table-cpxheis}, with $\delta=\text{sgn}(x_1)$.
\end{proof}

Lastly, we consider almost abelian algebras, where we can use some of the previous strategies to simplify the expression of the Hermitian metric \eqref{om}, as well as the complex structure equations.
Let $\mathfrak{g}$ be a $2n$-dimensional almost abelian Lie algebras endowed with a Hermitian structure $(J,g)$. Then, as described in \cite{FP_gk}, there exists a basis $\{e_1,\ldots,e_{2n}\}$ such that
$\mathfrak{n}=\text{span}\left<e_1,\ldots,e_{2n-1}\right>$, $\mathfrak{n}^{\perp_g}=\R \left< e_{2n} \right>$, $e_1=-Je_{2n}$, $\lVert e_1 \rVert_g = \lVert e_{2n} \rVert_g =1$ and $\mathfrak{k} \coloneqq \mathfrak{n} \cap J\mathfrak{n} = \text{span}\left<e_2,\ldots,e_{2n-1}\right>$. Here, we are not assuming $\{e_2,\ldots,e_{2n-1}\}$ to be orthonormal. With respect to such a basis, the whole Lie bracket of $\mathfrak{g}$ is encoded in $\text{ad}_{e_{2n}}\rvert_{\mathfrak{n}}$, which, with respect to the previous basis, takes the matrix form
\[
\text{ad}_{e_{2n}}\rvert_{\mathfrak{n}}=\begin{pmatrix} a & 0 \\ v & A \end{pmatrix},
\]
with $a \in \R$, $v \in \mathfrak{k}$ and $A \in \mathfrak{gl}(\mathfrak{k})$, with $[A,J\rvert_{\mathfrak{k}}]=0$. As before, this final condition can be exploited to say that, up to changing basis for $\mathfrak{k}$, we can assume $Je_{2j}=e_{2j+1}$, $j=1,\ldots,n-1$ and $A$ to be in real Jordan form (see \eqref{C^k})
\[
A=\text{diag}(C^{k_1}_{a_1,b_1},\ldots,C^{k_l}_{a_l,b_l}),
\]
with $k_j \in \mathbb{N}$, $a_j,b_j \in \R$, $j=1,\ldots,l$, where $\text{Spec}(A)=\{a_1\pm ib_1,\ldots,a_l \pm ib_l\}$ can be deduced, up to a uniform scaling, from the structure equations of $\mathfrak{g}$, and the signs of $b_1,\ldots,b_l$ in $A$ are determined by the behavior of $J$ and by a possible sign swap for $e_{2n}$ (and, consequently, $e_1$), as in the case of Heisenberg-type nilradical. Finally, we can rescale $e_{2n}$ and $e_1$ to apply a uniform rescaling to $a$ and $A$, in order to better match the original structure equations of $\mathfrak{g}$.

We can consider a basis of $(1,0)$-forms of the kind
\[
\alpha^1=\frac{1}{2}(e^1+ie^{2n}), \quad \alpha^{j}=\frac{1}{2}(e^{2j-2}+ie^{2j-1}),\;j=2,\ldots,n,
\]
and write the corresponding structure equations. Looking at the metric $g$, we can exploit the orthogonal splitting
\[
\mathfrak{g}=\R \left< e_1 \right> \oplus \R\left<e_2,\ldots,e_{2n-1}\right> \oplus \R \left< e_{2n} \right>
\]
to deduce that the fundamental form must be of the form
\[
\omega=i\lambda_1\alpha^{1 \overline{1}} + \omega_{\mathfrak{k}},
\]
with $\lambda_1>0$ and $\omega_{\mathfrak{k}} \in \Lambda^{1,1} \mathfrak{k}^*$ the fundamental form of the restriction of $g$ to $\mathfrak{k}$.

\begin{proposition} \label{almab_cpx}
Let $\mathfrak{g}$ be a six-dimensional unimodular almost abelian Lie algebra. Then, every Hermitian structure $(J,g)$ on $\mathfrak{g}$ admits a basis of $(1,0)$-forms $\{\alpha^1,\alpha^2,\alpha^3 \}$ satisfying the structure equations appearing in Table \ref{table-cpxab} (based on the isomorphism class of $\mathfrak{g}$), and such that
\begin{equation} \label{om_almab}
\omega= i (\lambda_1 \alpha^{1 \overline{1}} + \lambda_2 \alpha^{2 \overline{2}} + \lambda_3 \alpha^{3 \overline{3}}) + w \alpha^{2 \overline{3}} - \overline{w} \alpha^{3 \overline{2}},
\end{equation}
with $\lambda_1,\lambda_2,\lambda_3 \in \R_{>0}$ and $w \in \C$ satisfying the positivity condition
\begin{equation}
\lambda_2 \lambda_3>|w|^2.
\end{equation}
\end{proposition}

\section{1\textsuperscript{st}-Gauduchon structures} \label{sec_1stG} 
We start our list of classification results regarding Hermitian structures in six dimensions with 1\textsuperscript{st}-Gauduchon structures, i.e., Hermitian structures $(J,g)$ satisfying $\partial \overline\partial \omega \wedge \omega=0$. We do so in order to partly simplify our later discussion regarding SKT structures, as we can exploit the trivial fact that Lie algebras not admitting 1\textsuperscript{st}-Gauduchon structures certainly do not admit SKT structures.

\begin{theorem} \label{1stG_class}
Let $\mathfrak{g}$ be a six-dimensional strongly unimodular almost nilpotent Lie algebra. Then, $\mathfrak{g}$ admits 1\textsuperscript{st}-Gauduchon structures $(J,g)$ if and only if it is isomorphic to one among
\begin{alignat*}{3}
&\frs{3.3}^0 \oplus \R^3,                         &\quad &\mathfrak{h}_3 \oplus \mathfrak{s}_{3.3}^0,  &\quad &\frs{4.3}^{-\frac{1}{2},-\frac{1}{2}} \oplus \R^2,\\ 
&\frs{4.5}^{p,-\frac{p}{2}} \oplus \R^2,\;  {\scriptstyle p>0},  &\quad &\frs{4.6} \oplus \R^2,                       &\quad &\frs{4.7}\oplus \R^2, \\
&\frs{5.4}^0 \oplus \R,                           &\quad &\frs{5.13}^{0,0,r} \oplus \R, \; {\scriptstyle r>0},        &\quad &\frs{5.16} \oplus \R,  \\
&\frs{6.14}^{-\frac{1}{4},-\frac{1}{4}},          &\quad &\frs{6.16}^{p,-4p}, \; {\scriptstyle p<0},                  &\quad &\frs{6.17}^{1,q,q,-2(1+q)},\; {\scriptstyle 0<q \leq 1},  \\
&\frs{6.19}^{p,p,q,-p-\frac{q}{2}},\; {\scriptstyle p(2p+q) \leq 0, \, p,q \neq 0},  &\quad &\frs{6.21}^{p,q,r,-2(p+q)},\;{\scriptstyle pq \geq 0,\, |p| \geq |q|,\, p \neq 0,\, r > 0},   &\quad
&\frs{6.25},  \\
&\frs{6.51}^{p,0},\; {\scriptstyle p>0},          &\quad &\frs{6.52}^{0,q}, \; {\scriptstyle q>0},     &\quad &\frs{6.158},  \\
&\frs{6.159},                                     &\quad &\frs{6.164}^p,\; {\scriptstyle a >0},        &\quad &\frs{6.166}^p,\; {\scriptstyle 0<|p|\leq 1, \,p \neq 1}.
\end{alignat*}
Explicit examples of 1\textsuperscript{st}-Gauduchon structures on these Lie algebras are provided in Tables \ref{table-exab} and \ref{table-exheis}.
\end{theorem}
\begin{proof}
We start with the almost abelian case. Comparing the Lie algebras of the statement with the ones of Theorem \ref{cpx-class}, we need to prove the non-existence of 1\textsuperscript{st}-Gauduchon structures on the almost abelian Lie algebras
\begin{alignat*}{3}
&\frs{5.8}^0 \oplus \R, &\quad &\frs{5.9}^{1,-1,-1} \oplus \R, &\quad  &\frs{5.11}^{p,p,-p} \oplus \R,\; {\scriptstyle p>0}, \\
&\frs{5.13}^{p,-p,r} \oplus \R, \;{\scriptstyle p \neq 0, \,r>0}, &\quad &\frs{6.17}^{1,q,q,-2(1+q)},\; {\scriptstyle -1 < q < 0}, &\quad &\frs{6.18}^{1,-\frac{3}{2},-\frac{3}{2}}, \\
&\frs{6.19}^{p,p,q,-p-\frac{q}{2}}, \; {\scriptstyle p(2p+q) > 0,\, q \neq 0}, &\quad &\frs{6.20}^{p,p,-\frac{3}{2}p}, \; {\scriptstyle p>0}, &\quad &\frs{6.21}^{p,q,r,-2(p+q)}, \; {\scriptstyle pq<0, \, |p| \geq |q|,\, q \neq -p,\, r > 0}.
\end{alignat*}
To do so, we consider the generic Hermitian structure, which, by Proposition \ref{almab_cpx}, always admits a basis of $(1,0)$-forms satisfying the structure equations listed in Table \ref{table-cpxab} and making the metric of the form
\[
\omega= i (\lambda_1 \alpha^{1 \overline{1}} + \lambda_2 \alpha^{2 \overline{2}} + \lambda_3 \alpha^{3 \overline{3}}) + w \alpha^{2 \overline{3}} - \overline{w} \alpha^{3 \overline{2}},
\]
with $\lambda_1,\lambda_2,\lambda_3 \in \R_{>0}$, $w \in \C$ and $\lambda_2 \lambda_3>|w|^2$.
By a direct computation we can show that the $6$-form $\partial \overline \partial \omega \wedge \omega$ cannot vanish for each of the Lie algebras in the previous list.

Now, let us focus on Lie algebras having nilradical with one-dimensional commutator. We need to prove that 1\textsuperscript{st}-Gauduchon structures do not exist on the following Lie algebras:
\[
\frs{6.44}, \qquad \frs{6.162}^{1}, \qquad \frs{6.165}^p, \,{\scriptstyle p>0}, \qquad \frs{6.166}^1, \qquad
\frs{6.167}.
\]
By \cite{FP_alm}, every complex structure on these Lie algebras satisfies $J \mathfrak{n}^1 \not\subset \mathfrak{n}$, meaning we can refer to Proposition \ref{heis_cpx} to extract a suitable basis of $(1,0)$-forms satisfying the structure equations in Table \ref{table-exheis} and consider the generic Hermitian metric with fundamental form \eqref{om}.
From the computation of $\partial \overline\partial \omega \wedge \omega$ in each case  we can deduce the claim.

Finally, we have to tackle the Lie algebras with nilradical $\mathfrak{n}$ satisfying $\dim \mathfrak{n}^1 > 1$, namely the ones isomorphic to
\[
\frs{6.145}^{0}, \qquad
\frs{6.147}^{0}, \qquad
\frs{6.152},     \qquad
\frs{6.154}^0.
\]
Similarly to the previous case, if $J$ is a complex structure on one of them, we can refer to Propositions \ref{cpx-145_147} and \ref{cpx-152_154} for the structure equations of a special basis $\{\alpha^1,\alpha^2,\alpha^3\}$ of $(1,0)$-forms and consider the generic Hermitian metric \eqref{om}. For $\frs{6.145}^{0}$, \eqref{145_J1} yields
\[
\partial \overline\partial \omega \wedge \omega = -2 \lambda_1^2 \alpha^{1\overline{1}2\overline{2}3\overline{3}},
\]
while, on $\frs{6.147}^0$, we have
\begin{alignat*}{2}
\partial \overline\partial \omega \wedge \omega &= -2 (1+|z|^2) \alpha^{1\overline{1}2\overline{2}3\overline{3}}, &\quad &\text{if \eqref{147_J1} holds,} \\
\partial \overline\partial \omega \wedge \omega &= \left(-2 (\lambda_1 \Re(z) - 2 \Im(w_3))^2 - 2(\lambda_1 \Im(z)-2 \Re(w_3))^2-\lambda_1^2 \right) \alpha^{1\overline{1}2\overline{2}3\overline{3}}, &\quad &\text{if \eqref{147_J2} holds,} \\
\partial \overline\partial \omega \wedge \omega &= \left(-2 (x\lambda_1 -2 \Im(w_3))^2 - 8 \Re(w_3)^2 - \lambda_1^2\right) \alpha^{1\overline{1}2\overline{2}3\overline{3}}, &\quad &\text{if \eqref{147_J3} holds.}
\end{alignat*}
In the case of $\frs{6.152}$, one has
\begin{alignat*}{2}
\partial \overline\partial \omega \wedge \omega =& \left( -2(\lambda_1 \Re(z_1) + 2 \Im(w_3)\Re(z_2) + 2 \Im(w_2))^2 \right. & & \\
&\left.- 2(\lambda_1 \Im(z_1) + 2 \Re(w_3)\Re(z_2) + 2 \Re(w_2))^2 - \lambda_1^2 \right) \alpha^{1\overline{1}2\overline{2}3\overline{3}}, &\quad &\text{if \eqref{152_J1} holds,} \\
\partial \overline\partial \omega \wedge \omega =& -\frac{\lambda_1^2}{2 \Im(z_2)^4}\left( (\delta \Re(z_2)\Im(z_2) - |z_1|^2 )^2 + \Im(z_2)^2 (1+\Im(z_2)^2) \right) \alpha^{1\overline{1}2\overline{2}3\overline{3}}, &\quad &\text{if \eqref{152_J2} holds,}
\end{alignat*}
and \eqref{154_J} on $\frs{6.154}^0$ yields
\[
\partial \overline\partial \omega \wedge \omega = -\frac{\lambda_1^2}{2 x_1^4}\left( (x_1x_2 - |z|^2 )^2 + x_1^2 \right) \alpha^{1\overline{1}2\overline{2}3\overline{3}}.
\]
The non-vanishing of all these expressions concludes the proof of the theorem. 
\end{proof}

\section{SKT and LCSKT structures} \label{sec_SKT}

As described in the Introduction, the SKT condition for a Hermitian structure $(J,g)$ is characterized by the closure of $H \coloneqq d^c \omega$, which is the torsion $3$-form associated with the \emph{Bismut connection}. The locally conformally SKT (LCSKT) condition was recently introduced in \cite{DFFL} and can be characterized by the condition
$dH=\mu \wedge H,$ for some non-zero closed $1$-form $\mu$. 

In the Lie algebra setting, six-dimensional Lie algebras admitting SKT structures have been classified in the nilpotent case \cite{FPS}, the almost abelian case \cite{FP_gk}, while some partial results were obtained in the strongly unimodular almost nilpotent case when the nilradical is of Heisenberg type \cite{FP_alm}. Full characterizations in the almost abelian case and in the six-dimensional $2$-step solvable case have been proved in \cite{AL} and \cite{FS}, respectively.

Instead, \cite{DFFL} and \cite{BF} provide a full classification result for six-dimensional nilpotent and almost abelian Lie algebras admitting LCSKT structures.

The next theorem builds on some of these results to complete the classification of six-dimensional strongly unimodular almost nilpotent Lie algebras admitting SKT or LCSKT structures.

\begin{theorem} \label{SKT_LCSKT_class}
Let $\mathfrak{g}$ be a six-dimensional strongly unimodular almost nilpotent Lie algebra. Then, $\mathfrak{g}$ admits:
\begin{itemize} [leftmargin=0.5cm]
\item[{\normalfont (i)}] SKT structures if and only if it is isomorphic to one among
\begin{alignat*}{4}
&\frs{3.3}^0 \oplus \R^3,                          &\qquad &\mathfrak{h}_3 \oplus \mathfrak{s}_{3.3}^0,             &\qquad &\frs{4.3}^{-\frac{1}{2},-\frac{1}{2}} \oplus \R^2, &\qquad &\frs{4.5}^{p,-\frac{p}{2}} \oplus \R^2,\, {\scriptstyle p>0}, \\
&\frs{4.6} \oplus \R^2,                            &\qquad &\frs{4.7}\oplus \R^2,                                   &\qquad &\frs{5.4}^0 \oplus \R,                             &\qquad &\frs{5.13}^{0,0,r} \oplus \R,\, {\scriptstyle r>0}, \\
&\frs{6.19}^{p,p,-2p,0},\, {\scriptstyle p \neq 0}, &\qquad &\frs{6.21}^{p,0,r,-2p},\, {\scriptstyle p \neq 0,\, r >0},&\qquad &\frs{6.25},                                        &\qquad &\frs{6.51}^{p,0},\, {\scriptstyle p>0}, \\
&\frs{6.52}^{0,q},\, {\scriptstyle q>0},            &\qquad &\frs{6.158},                                            &\qquad &\frs{6.164}^p,\, {\scriptstyle p>0},
\end{alignat*}
\item[{\normalfont (ii)}] LCSKT structures if and only if it is isomorphic to one among
\begin{alignat*}{4}
&\frs{3.3}^0 \oplus \R^3, &\qquad &\mathfrak{h}_3 \oplus \mathfrak{s}_{3.3}^0, &\qquad &\frs{5.4}^0 \oplus \R, &\qquad &\frs{5.13}^{0,0,r} \oplus \R,\, {\scriptstyle r>0}, \\
&\frs{6.17}^{1,1,1,-4}, &\qquad &\frs{6.19}^{p,p,-4p,p},\,{\scriptstyle p \neq 0},  &\qquad &\frs{6.21}^{p,p,r,-4p},\, {\scriptstyle p \neq 0, \,r>0}, &\qquad &\frs{6.154}^0.
\end{alignat*}
\end{itemize}
In particular, if a six-dimensional strongly unimodular almost nilpotent Lie algebra admits SKT structures, then it is either almost abelian or its nilradical has one-dimensional commutator.
\end{theorem}
\begin{proof}
The almost abelian case was treated in \cite{FP_gk} and \cite{BF}.
For the non-almost abelian case, we start by proving that the Lie algebras not appearing in the statement do not admit \emph{twisted SKT} structures, namely Hermitian structures $(J,g)$ satisfying $dd^c \omega = \mu \wedge d^c \omega$, with the $1$-form $\mu$ being closed and being allowed to vanish: these structures encompass SKT structures and LCSKT structures as special cases ($\mu=0$ and $\mu \neq 0$, respectively). The Lie algebras we need to work with are the following:
\begin{alignat*}{5}
&\frs{5.16} \oplus \R, &\qquad
&\frs{6.44},           &\qquad
&\frs{6.145}^{0},      &\qquad
&\frs{6.147}^{0},      &\qquad
&\frs{6.152}, \\
&\frs{6.159},          &\qquad
&\frs{6.162}^{1},      &\qquad
&\frs{6.165}^p,\,{\scriptstyle p>0}, &\qquad
&\frs{6.166}^p,\,{\scriptstyle 0<|p|\leq 1}, &\qquad
&\frs{6.167}.
\end{alignat*}
By \cite{FP_alm}, it is already known that the Lie algebra $\frs{5.16} \oplus \R$ does not admits SKT structures, as all its complex structures satisfy $J\mathfrak{n}^1 \subset \mathfrak{n}$, $\mathfrak{n}$ being the nilradical of $\frs{5.16} \oplus \R$: hence, it only remains to prove that this Lie algebra does not admit LCSKT structures, and we shall treat it together with the Lie algebras $\frs{4.7} \oplus \R^2$ and $\frs{6.25}$ later in the proof. For the remaining Lie algebras, instead, we can exploit Propositions \ref{cpx-145_147}, \ref{cpx-152_154} and \ref{heis_cpx}, consider the complex structures defined by the structure equations in Table \ref{table-cpxheis}, exhibit the generic real closed $1$-form $\mu$, consider the generic Hermitian metric \eqref{om} and compute $dd^c \omega - \mu \wedge d^c \omega$, showing that it cannot vanish. The computations proving the claim are summarized in Table \ref{SKT_LCSKT_comp}.

\begin{table}[H]
\begin{center}
\addtolength{\leftskip} {-2cm}
\addtolength{\rightskip}{-2cm}
\scalebox{0.8}{
\begin{tabular}{|l|l|l|}
\hline \xrowht{15pt}
Lie algebra & Generic real closed $1$-form $\mu$ & $dd^c \omega  - \mu \wedge d^c \omega= \frac{1}{4} \xi_{jk\overline{l}\overline{m}} \alpha^{jk\overline{l}\overline{m}} + \xi_{123\overline{j}} \alpha^{123\overline{j}} + \xi_{j\overline{1}\overline{2}\overline{3}} \alpha^{j\overline{1}\overline{2}\overline{3}}$\\
\hline 
\hline \xrowht{25pt}
$\frs{6.44}$ & $iy(\alpha^1-\alpha^{\overline{1}})$ & \hspace{-0.8em} $\begin{array}{l} \xi_{12\overline{1}\overline{3}}=2y \lambda_2 \; \implies \, y=0, \\ \xi_{13\overline{1}\overline{3}}=4 \lambda_2 \end{array}$ \\
\hline \xrowht{20pt}
$\frs{6.145}^0$ & $\zeta \alpha^3 + \overline{\zeta} \alpha^{\overline{3}}$ & $\xi_{23\overline{2}\overline{3}}-2\Im(w_3)\xi_{13\overline{2}\overline{3}}=4\lambda_1$ \\
\hline \xrowht{20pt}
\multirow{3}{1em}{\parbox{1\linewidth}{\vspace{3.5em}\hbox{$\frs{6.147}^0$}}} & $\zeta \alpha^3 + \overline{\zeta} \alpha^{\overline{3}}$ & $\lambda_1 \xi_{23\overline{2}\overline{3}} -2\Im(w_3 \xi_{13\overline{2}\overline{3}})=4(1+|z|^2)\lambda_1^2$ \\
\cline{2-3} \xrowht{35pt}
 & $\zeta \alpha^3 + \overline{\zeta} \alpha^{\overline{3}}$ & \hspace{-0.8em} $\begin{array}{l} \xi_{123\overline{3}}=\lambda_1 \zeta \; \implies \, \zeta=0, \\ \xi_{13\overline{2}\overline{3}}=4 \lambda_1 \overline{z} + 8i w_3 \; \implies \, w_3=\frac{i}{2} \lambda_1 \overline{z}, \\ \xi_{23\overline{2}\overline{3}}=2\lambda_1 \end{array}$ \\
\cline{2-3} \xrowht{20pt}
 & $\zeta \alpha^3 + \overline{\zeta} \alpha^{\overline{3}}$ & \hspace{-0.8em} $\begin{array}{l} \xi_{123\overline{3}}=\lambda_1 \zeta \; \implies \, \zeta=0, \\ \xi_{23\overline{2}\overline{3}}-x \Re(\xi_{13\overline{2}\overline{3}})=2\lambda_1 \end{array}$ \\
\hline \xrowht{35pt}
\multirow{2}{1em}{\vspace{-1.2em}\parbox{1\linewidth}{\hbox{$\frs{6.152}$}}} & $iy(\alpha^2-\alpha^{\overline{2}})-iy\Re(z_2)(\alpha^3-\alpha^{\overline{3}})$ & \hspace{-0.8em} $\begin{array}{l} \xi_{123\overline{2}}=-\delta \lambda_1 y \; \implies \, y=0 \\ \xi_{12\overline{2}\overline{3}}=4\lambda_1 \overline{z}_1 - 8i\Re(z_2)w_3-8iw_2 \; \implies \, w_2=-\frac{i}{2} \lambda_1 \overline{z}_1 - \Re(z_2)w_3 \\ \xi_{23\overline{2}\overline{3}}=2\lambda_1 \end{array}$  \\
\cline{2-3} \xrowht{20pt} & $iy(\alpha^2-\alpha^{\overline{2}})-iy\Re(z_1)(\alpha^3-\alpha^{\overline{3}})$ &  \hspace{-0.8em} $\begin{array}{l} \Im(\xi_{12\overline{2}\overline{3}})=-\delta \lambda_1 y \; \implies \, y=0 \\ \xi_{23\overline{2}\overline{3}}=\frac{\lambda_1}{\Im(z_2)^4} \left( (\delta \Re(z_2)\Im(z_2)-|z_1|^2)^2+\Im(z_2)^2(1+\Im(z_2)^2) \right) \end{array}$ \\
\hline \xrowht{35pt}
$\frs{6.159}$ & $iy(\alpha^1-\alpha^{\overline{1}})+\zeta \alpha^3 + \overline{\zeta} \alpha^{\overline{3}}$ & \hspace{-0.8em} $\begin{array}{l} \xi_{12\overline{1}\overline{2}}=2\delta \lambda_1 y \; \implies \, y=0, \\ \xi_{12\overline{2}\overline{3}}=-i\delta \lambda_1 \overline{\zeta} \; \implies \, \zeta=0, \\ \xi_{23\overline{2}\overline{3}}=-4\delta \varepsilon \lambda_1 \end{array}$\\
\hline \xrowht{20pt}
$\frs{6.162}^1$ & $iy(\alpha^1-\alpha^{\overline{1}})$ & $\xi_{23\overline{2}\overline{3}}=4\lambda_1$ \\
\hline \xrowht{20pt}
$\frs{6.165}^p$, $p>0$ & $iy(\alpha^1-\alpha^{\overline{1}})$ & $\xi_{23\overline{2}\overline{3}}=4\lambda_1$ \\
\hline \xrowht{20pt}
$\frs{6.166}^p$, $0<|p|\leq 1$ & $iy(\alpha^1-\alpha^{\overline{1}})$ & $\xi_{23\overline{2}\overline{3}}=-4 \varepsilon \lambda_1$ \\
\hline \xrowht{20pt}
$\frs{6.167}^1$ & $iy(\alpha^1-\alpha^{\overline{1}})$ & $\xi_{23\overline{2}\overline{3}}=4\lambda_1$ \\
\hline
\end{tabular}}
\caption{Six-dimensional strongly unimodular non-almost abelian almost nilpotent Lie algebras admitting neither SKT nor LCSKT structures (except $\frs{5.16} \oplus\R$).} \label{SKT_LCSKT_comp}
\end{center}
\end{table}

Comparing the two lists in the statement, it is easy to prove that $\frs{6.154}^0$ does not admit SKT structures, since, by Theorem \ref{1stG_class}, it does not admit 1\textsuperscript{st}-Gauduchon structures.

It remains to prove that the following Lie algebras (which admit SKT structures, with the exception of $\frs{5.16} \oplus \R$) do not admit LCSKT structures:
\[
\frs{4.6} \oplus \R^2, \quad
\frs{4.7}\oplus \R^2,\quad
\frs{5.16} \oplus \R, \quad
\frs{6.25}, \quad
\frs{6.51}^{p,0},\,{\scriptstyle p>0}, \quad
\frs{6.52}^{0,q},\,{\scriptstyle q>0}, \quad
\frs{6.158}, \quad
\frs{6.164}^p,\,{\scriptstyle p >0}.
\]
All these Lie algebras have nilradical with one-dimensional commutator. First, we work with complex structures satisfying $J\mathfrak{n}^1 \subset \mathfrak{n}$, meaning we are ruling out $\frs{6.52}^{0,q}$, by \cite{FP_alm}. Assume $(\mathfrak{g},J,g)$ is an LCSKT Lie algebra, with $\mathfrak{g}$ isomorphic to one of the previous Lie algebras and $J$ satisfying $J\mathfrak{n}^1 \subset \mathfrak{n}$. By \cite{FP_alm}, $\mathfrak{g}$ admits an orthonormal basis $\{e_1,\ldots,e_6\}$, with $\mathfrak{n}=\R\left<e_1,\ldots,e_5\right>$, $\mathfrak{n}_1=\R\left<e_1\right>$ and $Je_1=e_2$, $Je_3=e_4$, $Je_5=e_6$. With respect to this basis, one has
\[
\begin{gathered}
\text{ad}_{e_6}\rvert_{\mathfrak{n}}=\left( \begin{array}{c|c|c|c}
		0 & 0 & \gamma_1 & v_1 \\ \hline
		0 & m_1 & J\gamma_1 + \gamma_2 & v_2 \\ \hline
		0 & 0 & \begin{matrix} -\tfrac{1}{2}(m_1+m_2) & q \\ -q & -\tfrac{1}{2}(m_1+m_2) \end{matrix} & v \\ \hline
		0 & 0 & 0 & m_2
	\end{array} \right), \\ de^1=\eta=c\,e^{34} + \gamma_2 \wedge e^{5} + m_1 \, e^{25},
\end{gathered}
\]
with
\[
m_1,m_2,q,c,v_1,v_2 \in \R, \quad v \in \mathfrak{k}, \quad \gamma_1,\gamma_2 \in \mathfrak{k}^*,
\]
with $\mathfrak{k} \coloneqq \R \left<e_3,e_4\right>$, satisfying the following conditions imposed by the Jacobi identity:
\begin{align}
&m_1(m_1+m_2)=0, \label{sub1} \\
&c(m_1+m_2)=0, \label{sub2} \\
&(m_1+m_2)\gamma_2 + m_1 J\gamma_1 + A^*\gamma_2 -cJv^\flat =0, \label{sub3}
\end{align}
where
\[
A=\begin{pmatrix} -\tfrac{1}{2}(m_1+m_2) & q \\ -q & -\tfrac{1}{2}(m_1+m_2) \end{pmatrix}.
\]
and $(\cdot)^\flat$ is the musical isomorphism between $\mathfrak{g}$ and $\mathfrak{g}^*$ induced by the metric $g$. Actually, one can prove that $m_1+m_2=0$ holds: if this were not case, \eqref{sub1} and \eqref{sub2} would imply $m_1=c=0$, so that \eqref{sub3} reads $A^* \gamma_2=0$, which cannot occur, since the degeneracy of $A$ would imply $m_2=q=0$. In the end, one obtains 
\begin{equation} \label{sub_adapted}
\text{ad}_{6}\rvert_{\mathfrak{n}}=\left( \begin{array}{c|c|c|c}
		0 & 0 & \gamma_1 & v_1 \\ \hline
		0 & m & J\gamma_1 + \gamma_2 & v_2 \\ \hline
		0 & 0 & \begin{matrix} 0 & q \\ -q & 0 \end{matrix} & v \\ \hline
		0 & 0 & 0 & -m
	\end{array} \right), \quad de^1=\eta=c\,e^{34} + \gamma_2 \wedge e^{5} + m \, e^{25},
\end{equation}
with
\begin{equation} \label{sub_adapted_condition}
m \gamma_1 + q \gamma_2 - cv^\flat=0,
\end{equation}
exploiting $A=-qJ\rvert_{\mathfrak{k}}$.
Following the techniques in \cite{FP_alm}, imposing $\mathfrak{g}$ to be isomorphic to a specific Lie algebra induces further conditions on the parameters involved. We do this one Lie algebra at a time, proving that the vanishing of
\[
\xi \coloneqq dd^c \omega - \mu \wedge d^c \omega,
\]
with $\mu$ a generic non-zero closed $1$-form, produces a contradiction:
\begin{itemize} [leftmargin=0.5cm]
\item $\frs{4.7} \oplus \R^2$, $\frs{5.16} \oplus \R$, $\frs{6.25}$: necessarily, $m=0$, $q \neq 0$. Now, the generic closed $1$-form is
\begin{equation} \label{mu}
\mu=y_1e^1+y_2e^2+\mu_{\mathfrak{k}} + y_5e^5 + y_6e^6, \quad y_1,y_2,y_5,y_6 \in \R,\; \mu_{\mathfrak{k}} \in \mathfrak{k}^*,
\end{equation}
with
\begin{equation} \label{4.7_LCSKT_closed}
cy_1=0, \quad y_1 \nu =0, \quad y_1v_1+y_2v_2 + \mu_{\mathfrak{k}}(v)=0, \quad
\mu_{\mathfrak{k}}=\frac{1}{q} \left( y_1 J\gamma_1 - y_2 \gamma_1 + y_2 J\gamma_2 \right).
\end{equation}
For a generic $X \in \mathfrak{k}$, a computation yields
\begin{align}
\label{4.7_LCSKT_1} \xi(e_1,e_2,e_3,e_4)&=cy_2, \\
\label{4.7_LCSKT_2} \xi(e_1,e_2,X,e_5)&=-y_1(\gamma_1-J\gamma_2)(X) - y_2 J\gamma_1(X), \\
\label{4.7_LCSKT_3} \xi(e_1,e_2,X,e_6)&=y_2 J\gamma_2(X), \\
\label{4.7_LCSKT_4} \xi(e_2,X,e_5,e_6)&=q(J\gamma_1+\gamma_2)(X) + y_6(\gamma_1-J\gamma_2)(X) - y_2 g(v,X) + v_2 \mu_{\mathfrak{k}}(X).
\end{align}
First, we notice that $y_1 = y_2= 0$ must hold, otherwise \eqref{4.7_LCSKT_closed}, \eqref{4.7_LCSKT_1}, \eqref{4.7_LCSKT_2} and \eqref{4.7_LCSKT_3} would imply $c=0$, $\gamma_2=0$, annihilating $\eta=de^1$, a contradiction.
Then, \eqref{4.7_LCSKT_closed} forces $\mu_{\mathfrak{k}}=0$, so that \eqref{4.7_LCSKT_4} now implies $\gamma_2=-J\gamma_1$. We now compute
\[
\xi(e_1,e_3,e_4,e_5)=cy_5, \quad \xi(e_1,e_3,e_4,e_6)=cy_6
\]
and, since we want $\mu=y_5e^5+y_6e^6$ not to vanish, we must set $c=0$. Now,
\[
\xi(e_3,e_4,e_5,e_6)=-2\vert \gamma_1 \rvert^2
\]
implies $\gamma_1=0$, in turn forcing $\eta=0$;
\item $\frs{4.6} \oplus \R^2$, $\frs{6.51}^{p,0}$, $\frs{6.158}$, $\frs{6.164}^p$: we need to require $m \neq 0$. We write the generic $1$-form \eqref{mu} and compute
\[
d\mu(e_2,e_5)=y_1m, \quad d\mu(e_2,e_6)=y_2m,
\]
implying $y_1=y_2=0$. Now, a computations shows
\begin{align*}
\xi(e_1,e_2,X,e_5)&=m \mu_{\mathfrak{k}} (X), \quad X \in \mathfrak{k}, \\
\xi(e_1,e_2,e_5,e_6)&=-my_6,
\end{align*}
so that we must set $\mu_{\mathfrak{k}}=0$ and $y_6=0$. Now, $\mu$ is reduced to $\mu=y_5e^5$ and it cannot be closed unless it vanishes, since $d\mu=-my_5e^{56}$.
\end{itemize}

To end the proof, we need to study the Lie algebras $\frs{4.7} \oplus \R^2$ and $\frs{6.52}^{0,q}$ in the case of a complex structure $J$ satisfying $J\mathfrak{n}^1 \not\subset \mathfrak{n}$. We proceed exactly as in the case of Lie algebras admitting neither SKT nor LCSKT structures, this time assuming $\mu \neq 0$. The computations proving the claim are summarized in Table \ref{SKT_LCSKT_comp2}.\end{proof}

\begin{table}[H]
\begin{center}
\addtolength{\leftskip} {-2cm}
\addtolength{\rightskip}{-2cm}
\scalebox{0.85}{
\begin{tabular}{|l|l|l|}
\hline \xrowht{15pt}
Lie algebra & Generic real closed $1$-form $\mu$ & $dd^c \omega  - \mu \wedge d^c \omega= \frac{1}{4} \xi_{jk\overline{l}\overline{m}} \alpha^{jk\overline{l}\overline{m}} + \xi_{123\overline{j}} \alpha^{123\overline{j}} + \xi_{j\overline{1}\overline{2}\overline{3}} \alpha^{j\overline{1}\overline{2}\overline{3}}$\\
\hline 
\hline \xrowht{25pt}
$\frs{4.7} \oplus \R^2$ & $iy(\alpha^1-\alpha^{\overline{1}}) + \zeta \alpha^3 + \overline{\zeta} \alpha^{\overline{3}}$ & \hspace{-0.8em} $\begin{array}{l} \xi_{12\overline{1}\overline{2}}=2\varepsilon y \lambda_1 \; \implies \, y=0, \\ \xi_{23\overline{1}\overline{2}}=i\varepsilon \zeta \lambda_1 \end{array}$ \\
\hline
\xrowht{20pt}
$\frs{6.52}^{0,q}$, $q>0$ & $iy(\alpha^1-\alpha^{\overline{1}})$ & $\xi_{12\overline{1}\overline{2}}=2\delta y \lambda_1$  \\
\hline
\end{tabular}}
\caption{Six-dimensional strongly unimodular non-almost abelian almost nilpotent Lie algebras admitting SKT structures but no LCSKT structures (excluding those with nilradical having one-dimensional commutator and only admitting complex structures satisfying $J\mathfrak{n}^1 \subset \mathfrak{n}$).} \label{SKT_LCSKT_comp2}
\end{center}
\end{table}

\begin{remark}
We note that we have not obtained any new isomorphism classes of Lie algebras admitting SKT structures with respect to the ones already known by the classification results in \cite{FP_gk, FP_alm}. Instead, $\mathfrak{h}_3 \oplus \frs{3.3}^0$ and $\frs{6.154}^0$ provide new examples of Lie algebras admitting LCSKT structures, with the latter being of particular interest, as it does cannot carry SKT structures. The simply connected Lie groups associated with these two Lie algebras admit cocompact lattices (see \cite[Remark 4.8]{FP_alm}, for example, and Remark \ref{rem_lattices}), yielding new examples of compact LCSKT manifolds. Moreover, in the case of $\frs{6.154}^0$, such examples can have non-degenerate torsion form $H=d^c\omega$, in the sense that the contraction $\iota_X H$ of $H$ by any non-zero tangent vector $X$ yields a non-zero $2$-form: for example, one can consider the one induced by the example of LCSKT structure on $\frs{6.154}^0$ from Table \ref{table-exnew}, having torsion form
\[
H=f^{146}-f^{256}-f^{345}.
\]
The claim follows from noticing that the $2$-forms $\iota_{f_i}H$, $i=1,\ldots,6$, are all linearly independent.
\end{remark}

\section{Existence of tamed complex structures} \label{sec_tamed}

In this  section we provide a negative result regarding the existence of symplectic forms taming complex structures on almost nilpotent Lie algebras. In the almost abelian, we can drop both the six-dimensional and the unimodular hypothesis and prove a general result. Notice that the result has already been proven in \cite{FKV} under the hypothesis that the Lie algebra is not of type I, namely, that some of the eigenvalues of the adjoint action of a non-nilpotent element on the nilradical are not imaginary.

\begin{theorem} \label{tamed_almab}
Let $\mathfrak{g}$ be a $2n$-dimensional almost abelian Lie algebra endowed with a complex structure $J$. Then, there do not exist symplectic structures $\Omega$  on $\mathfrak{g}$ taming $J$, unless $(\mathfrak{g},J)$ admits K\"ahler metrics.
\end{theorem}
\begin{proof}
Using the characterization we have just recalled, we assume $\mathfrak{g}$ admits an SKT structure $(J,g)$ such that $\partial \omega = \overline\partial \beta$ for some $\partial$-closed $(2,0)$-form $\beta$.

Following \cite{FP_gk}, we now fix a codimension-one abelian ideal and observe that $J\mathfrak{n}^{\perp_g} \subset \mathfrak{n}$ since $g$ is $J$-Hermitian. We also denote by $\mathfrak{k} \coloneqq \mathfrak{n} \cap J\mathfrak{n}=(\mathfrak{n}^{\perp_g} \oplus J\mathfrak{n}^{\perp_g})^{\perp_g}$ the maximal $J$-invariant subspace of $\mathfrak{n}$.

By \cite{AL}, one can then find an orthonormal basis $\{e_1,\ldots,e_{2n}\}$ of $\mathfrak{g}$ adapted to the splitting $\mathfrak{g}=J\mathfrak{n}^{\perp_g} \oplus \mathfrak{k} \oplus \mathfrak{n}^{\perp_g}$, such that $Je_1=e_{2n}$, $Je_{2k-1}=e_{2k}$, $k=1,\ldots,n-1$ and such that, with respect to the basis $\{e_1,\ldots,e_{2n-1}\}$ for $\mathfrak{n}$, the matrix associated with the endomorphism $\text{ad}_{e_{2n}}\rvert_{\mathfrak{n}}$ is of the form
\[
\text{ad}_{e_{2n}}\rvert_{\mathfrak{n}}=\begin{pmatrix} a & 0 \\ v & A \end{pmatrix},
\]
with $a \in \R$, $v \in \mathfrak{k}_1$ and $A$ of the form
\begin{equation} \label{A}
A=\text{diag}\left( \begin{pmatrix} -\frac{a}{2} & -b_1 \\ b_1 & -\frac{a}{2} \end{pmatrix}, \ldots, \begin{pmatrix} -\frac{a}{2} & -b_h \\ b_h & -\frac{a}{2} \end{pmatrix} , \begin{pmatrix} 0 & -b_{h+1} \\ b_{h+1} & 0 \end{pmatrix},\ldots,\begin{pmatrix} 0 & -b_{n-1} \\ b_{n-1} & 0 \end{pmatrix} \right),
\end{equation}
for some $b_k \in \R$, $k=1,\ldots,n-1$, $h \in \{0,\ldots,n-1\}$, $|b_1| \geq \ldots \geq |b_h|$, $|b_{h+1}| \geq \ldots \geq |b_{n-1}|$.

Now, denote by $\mathfrak{g}^{1,0},\mathfrak{g}^{0,1} \subset \mathfrak{g} \otimes \C$ the eigenspaces of $J$ with respect to the eigenvalues $\pm i$.
Bases for $\mathfrak{g}^{1,0}$ and $\mathfrak{g}^{0,1}$ are given, respectively, by $\{Z_1,\ldots,Z_n\}$ and $\{\overline{Z}_1,\ldots,\overline{Z}_n\}$, where 
\[
Z_1=\frac{1}{2}(e_1-ie_{2n}), \quad Z_k=\frac{1}{2}(e_{2k-1}-ie_{2k}),\;k=2,\ldots,n.
\]
Their respective dual bases for $(\mathfrak{g}^{1,0})^*$ and $(\mathfrak{g}^{0,1})^*$ are $\{\alpha^1,\ldots,\alpha^n\}$ and $\{\overline{\alpha}^1,\ldots,\overline{\alpha}^n\}$, with
\[
\alpha^1=e^1+ie^{2n}, \quad \alpha^k=e^{2k-1}+ie^{2k},\,k=2,\ldots,n.
\]

Recall that, by the integrability of $J$, the exterior differential $d$ splits into the sum $\partial + \overline\partial$, with
\[
\partial \colon \Lambda^{p,q}\mathfrak{g}^* \to \Lambda^{p+1,q}\mathfrak{g}^*,\quad \overline\partial \colon \Lambda^{p,q}\mathfrak{g}^* \to \Lambda^{p,q+1}\mathfrak{g}^*,
\]
for all $p,q=1,\ldots,n$, where $\Lambda^{p,q}\mathfrak{g}^* \coloneqq \Lambda^p(\mathfrak{g}^{1,0})^* \otimes \Lambda^q(\mathfrak{g}^{0,1})^* \subset \Lambda^{p+q}(\mathfrak{g}^* \otimes \C)$.

In the setup above, it is easy to see that $d \mathfrak{g}^* \subset \mathfrak{g}^* \wedge e^{2n}$, so that one has
\[
\partial (\Lambda^{1,1}\mathfrak{g}^*) \subset \Lambda^{1,1}\mathfrak{g}^* \wedge \alpha^1,\qquad \overline \partial (\Lambda^{2,0} \mathfrak{g}^*) \subset \Lambda^{2,0} \mathfrak{g}^* \wedge \overline{\alpha}^1.
\]
Therefore, since, by hypothesis, $\partial \omega \in \partial (\Lambda^{1,1}\mathfrak{g}^*) \cap \overline \partial (\Lambda^{2,0} \mathfrak{g}^*)$, we must have $\partial \omega \in \Lambda^{1,0}\mathfrak{n}_1^* \wedge \alpha^1 \wedge \overline{\alpha}^1$. Then, $\partial \omega$ has no components inside $\Lambda^{1,1} \mathfrak{n}_1^* \wedge \alpha^1$, so that, for all $X,Y \in \mathfrak{n}_1$, one must have
\[
d\omega(Z_1,X-iJX,Y+iJY)=0.
\]
Exploiting $[\mathfrak{n},\mathfrak{n}]=\{0\}$ and $[A,J|_{\mathfrak{k}}]=0$, we compute
\begin{align*}
d\omega(Z_1,X-iJX,Y+iJY)&=-\omega([Z_1,X-iJX],Y+iJY)+\omega([Z_1,Y+iJY],X-iJX) \\
                        &=\frac{i}{2}\omega([e_{2n},X-iJX],Y+iJY)-\frac{i}{2}\omega([e_{2n},Y+iJY],X-iJX) \\
                        &=\frac{i}{2}\omega(AX-iJAX,Y+iJY)-\frac{i}{2}\omega(AY+iJAY,X-iJX) \\
                        &=-\frac{i}{2}g(AX-iJAX,J(X+iJY))+\frac{i}{2}g(AY+iJAY,J(X-iJX)) \\
                        &=-\frac{1}{2}g(AX-iJAX,Y+iJY)-\frac{1}{2}g(AY+iJAY,X-iJX) \\
                        &=-(g(AX,Y)+g(X,AY))-i(g(AX,JY)-g(JX,AY)) \\
                        &=-g((A+A^t)X,Y)-ig((A+A^t)X,JY),
\end{align*}
which vanishes for all $X,Y \in \mathfrak{k}$ if and only if $A$ is skew-symmetric. In particular, we obtain $h=0$ in \eqref{A} and the eigenvalues of $A$ are all purely imaginary.

Now, let us assume $a \neq 0$. Then, $a$ cannot be an eigenvalue of $A$, so that $A-a\text{Id}_{\mathfrak{k}}$ is non-singular, ensuring the existence of a unique $X \in \mathfrak{k}$ such that $\mathfrak{k} \ni v = AX-aX$.

Consider then the new $J$-Hermitian metric
\[
g^\prime=g\rvert_{\mathfrak{k}}+({e^1}^\prime)^2+({e^{2n}}^\prime)^2,
\]
where $e_1^\prime=e_1-X$, $e_{2n}^\prime=Je_1^\prime=e_{2n}-JX$ and duals are taken with respect to the splitting $\mathfrak{g}=\text{span}\left<e_1^\prime\right> \oplus \mathfrak{k} \oplus \text{span}\left<e_{2n}^\prime\right>$. Then, with respect to the new adapted unitary basis $\{e_1^\prime,e_2,\ldots,e_{2n-1}\}$ for $\mathfrak{n}$, we have
\[
\text{ad}_{e_{2n}^\prime}\rvert_{\mathfrak{n}}= \begin{pmatrix} a & 0 \\ 0 & A \end{pmatrix},
\]
with $a$ and $A$ as above: observe that $A$ is still skew-symmetric with respect to $g^{\prime}$, since $g\rvert_{\mathfrak{k}}=g^\prime\rvert_{\mathfrak{k}}$. By the characterization of K\"ahler almost abelian Lie algebras \cite{LW}, $(\mathfrak{g},J,g^{\prime})$ is K\"ahler.

We can then assume $a=0$ and the previous method can only work if we manage to prove that $v$ lies in the image of $A$, which corresponds to the orthogonal complement to $\ker A$ inside $\mathfrak{k}$. $A$ is now of the form
\[
A=\text{diag}\left( \begin{pmatrix} 0 & -b_1 \\ b_1 & 0 \end{pmatrix}, \ldots, \begin{pmatrix} 0 & -b_t \\ b_t & 0 \end{pmatrix} , 0,\ldots,0 \right),
\]
for some $b_k \in \R-\{0\}$, $k \in \{1,\ldots,t\}$, and some $t \in \{0,\ldots,n-1\}$.

Notice that $\ker A$ and its orthogonal complement inside $\mathfrak{k}_1$ are both $J$-invariant

Now, using again $\partial \omega = \overline\partial \beta \in \Lambda^{1,0}\mathfrak{n}_1^* \wedge \alpha^1 \wedge \overline{\alpha}^1$, we consider $k \in \{2,\ldots,t+1\}$, $l \in \{t+2,\ldots,n\}$ (implying $Z_k \in \Lambda^{1,0}(\ker A)^{\perp_g}$ and $Z_l \in \Lambda^{1,0}\ker A$) and we have
\begin{align*}
0=d \beta (Z_{\overline{1}},Z_k,Z_l)&=-\beta([Z_{\overline{1}},Z_k],Z_l) + \beta ([Z_{\overline{1}},Z_l],Z_k), \\
&=i \beta([e_{2n},Z_k],Z_l) - i \beta([e_{2n},Z_l],Z_k), \\
&=b_k \beta(Z_l,Z_k).
\end{align*}
We then must have $\beta(Z_k,Z_l)$, meaning
\begin{equation} \label{beta_split}
\beta\left(\Lambda^{1,0} \ker A, \Lambda^{1,0} \left((\ker A)^{\perp_g} \cap \mathfrak{k}\right) \right)=0.
\end{equation}
It is easy to see that, in our situation, given a $2$-form $\gamma \in \Lambda^2 \mathfrak{g}^*$, one has
\[
\left(\iota_{Z_{\overline{1}}} \iota_{Z_1} d\gamma \right) \rvert_{\ker A} = \frac{i}{2} (\iota_v \gamma) \rvert_{\ker A}.
\]
As a consequence, the condition that $\partial \omega$ and $\overline \partial \beta$ coincide along their respective $\alpha^1 \wedge \alpha^{\overline{1}} \wedge \Lambda^{1,0}(\ker A)^*$-components boils down to
\begin{equation} \label{tamed}
(\iota_v \beta)\rvert_{\Lambda^{1,0}(\ker A)} = (\iota_v \omega)\rvert_{\Lambda^{1,0}(\ker A)}.
\end{equation}
To exploit this, we first write
\[
v=w + \tilde{w}, \quad w \in \ker A, \; \tilde{w} \in (\ker A)^{\perp_g} \cap \mathfrak{k}=\operatorname{Im} A,
\]
and then
$
w = w^{1,0} + \overline{w^{1,0}}, w^{1,0} \in \Lambda^{1,0} (\ker A).
$
Evaluating the left-hand side of \eqref{tamed} on $w^{1,0}$, we obtain
\[
\beta\left(v,w^{1,0}\right)=\beta\left(w^{1,0},w^{1,0}\right)=0,
\]
where we used that $\beta$ is of type $(2,0)$ and condition \eqref{beta_split}. Instead, the right-hand side reads
\[
\omega\left(v,w^{1,0}\right)=-\frac{i}{2} \left\lvert w^{1,0} \right\rvert^2.
\]
It follows that $w=0$, so that $v \in (\ker A)^{\perp_g} \cap \mathfrak{k}=\operatorname{Im} A$. The claim then follows.
\end{proof}

Having dealt with the almost abelian case in full generality, we restrict to the six-dimensional strongly unimodular case and analyze the remaining almost nilpotent Lie algebras.

\begin{theorem} \label{tamed_almnip}
Let $\mathfrak{g}$ be a six-dimensional strongly unimodular almost nilpotent Lie algebra endowed with a complex structure $J$. Then, $\mathfrak{g}$ admits symplectic structures $\Omega$ taming $J$ if and only if $(\mathfrak{g},J)$ admits K\"ahler metrics.
\end{theorem}
\begin{proof}
In the almost abelian case, the results follows from Theorem \ref{tamed_almab}. Now the Lie algebra $\mathfrak{h}_3 \oplus \frs{3.3}^0$ is the only Lie algebra of Theorem \ref{cpx-class} admitting both SKT structures and symplectic structures, both of whose existence is necessary for the existence of a complex structure tamed by a symplectic form: this is achieved by considering the Lie algebras of Theorem \ref{cpx-class} and comparing with \cite{Mac}, or by performing simple explicit computations, considering the generic closed $2$-form on each Lie algebra and checking if it can be non-degenerate. We can then take the generic complex structure on $\mathfrak{h}_3 \oplus \frs{3.3}^0$ and consider the basis $\{\alpha^1,\alpha^2,\alpha^3\}$ of $(1,0)$-forms yielding the structure equations in Table \ref{table-cpxheis}, by Proposition \ref{heis_cpx}. In such a basis, the generic $\partial$-closed $(2,0)$-form is
\[
\beta=z_1 \alpha^{12} + z_2 \alpha^{13}, \quad z_1,z_2 \in \C
\]
and, letting $\{Z_1,Z_2,Z_3,Z_{\overline{1}},Z_{\overline{2}},Z_{\overline{3}}\}$ denote the basis of $(\mathfrak{h}_3 \oplus \frs{3.3}^0) \otimes \C$ which is dual with respect to $\{\alpha^1,\alpha^2,\alpha^3,\alpha^{\overline{1}},\alpha^{\overline{2}},\alpha^{\overline{3}}\}$, one can compute
\[
(\partial \omega - \overline\partial \beta)(Z_1,Z_3,Z_{\overline{3}})=\varepsilon \lambda_1 \neq 0,
\]
negating the tamed condition and concluding the proof.
\end{proof}

\section{LCB and LCK structures} \label{sec_LCBLCK}

In \cite{Par}, the LCB condition was studied on almost abelian algebras, yielding a characterization result in all dimensions and a full classification result in dimension six.

The LCK condition has been studied more widely in literature: in \cite{Saw}, it was proven that the only (non-abelian) nilpotent Lie algebras admitting LCK structure are of the form $\mathfrak{h}_{2k+1} \oplus \R$, where $\mathfrak{h}_{2k+1}$ denotes the $(2k+1)$-dimensional Heisenberg Lie algebra we have defined in Section \ref{sec_cpx}. In the almost abelian case, the characterization results of \cite{AO} were used in \cite{Par} to obtain a full classification in dimension six.

The following results extend the previously-mentioned results to the complete six-dimensional strongly unimodular almost nilpotent case.

\begin{theorem}
Let $\mathfrak{g}$ be a six-dimensional strongly unimodular almost nilpotent Lie algebra. Then, $\mathfrak{g}$ admits LCB structures if and only if it is isomorphic to one of the Lie algebras of Theorem \ref{cpx-class}, with the exception of
\[
\frs{5.4}^0 \oplus \R, \qquad
\frs{6.18}^{1,-\frac{3}{2},-\frac{3}{2}}, \qquad
\frs{6.20}^{p,p,-\frac{3}{2}p},\, {\scriptstyle p>0}, \qquad
\frs{6.25}.
\]
Explicit examples of LCB structures on the remaining Lie algebras are provided in Tables \ref{table-exab}, \ref{table-exheis} and \ref{table-exnew}.
\end{theorem}
\begin{proof}
The result in the almost abelian case follows from \cite{Par}. It remains to prove that $\frs{6.25}$ does not admit LCB structures. This Lie algebra has nilradical $\mathfrak{n}$ isomorphic to $\mathfrak{h}_3 \oplus \R^2$ and all its complex structures satisfy $J\mathfrak{n}^1 \subset \mathfrak{n}$, by \cite{FP_alm}. Following the same techniques of the proof of Theorem \ref{SKT_LCSKT_class}, let us assume $\frs{6.25}$ admits an LCB structure $(J,g)$; then, there exists an orthonormal basis $\{e_1,\ldots,e_6\}$, with $\mathfrak{n}=\R\left<e_1,\ldots,e_5\right>$, $\mathfrak{n}_1=\R\left<e_1\right>$ and $Je_1=e_2$, $Je_3=e_4$, $Je_5=e_6$, with respect to which one has \eqref{sub_adapted}, with the further conditions $m=0$, $q \neq 0$. Now, \eqref{sub_adapted_condition}
implies $\gamma_2=\frac{c}{q} v^\flat$ must hold. By \cite{FP_alm}, the Lee form is
\begin{equation} \label{theta_sub}
\theta= -v_2e^1+(c+v_1)e^2+Jv^\flat - m e^6,
\end{equation}
and a computation shows
$d\theta(e_3,e_4)=-cv_2,$
meaning $v_2=0$, since $c \neq 0$ is forced by the non-vanishing of $\eta$. Moving on, we have
\begin{equation} \label{dthetaXe6}
d\theta(X,e_6)=\frac{1}{q}\left( (c+v_1)(cv^\flat + q J\alpha) -q^2v^\flat \right), \quad X \in \mathfrak{k}.
\end{equation}
If we assume $v_1=-c$, \eqref{dthetaXe6} implies $v=0$, and now the change of basis
\begin{gather*}
f_1=e_1,\;f_2=\frac{1}{q} (\gamma_1(e_4)e_1 + \gamma_1(e_3)e_2) +e_3, \; f_3=\frac{1}{q}(-\gamma_1(e_3)e_1+\gamma_1(e_4)e_2) + e_4, \\ f_4=-\frac{q}{c} e_5, \; f_5=e_2, \; f_6=\frac{1}{q} e_6
\end{gather*}
provides an explicit isomorphism with $\frs{5.16} \oplus \R$, a contradiction. This means that we must have $c+v_1 \neq 0$, instead. Equation \eqref{dthetaXe6} then forces
\[
\alpha=\frac{c(c+v_1)-q^2}{q(c+v_1)} Jv^\flat.
\]
Now, if
\begin{equation} \label{condition_LCB}
(c(c+v_1)-q^2)|v|^2-q^2v_1(c+v_1)
\end{equation}
vanishes, the change of basis
\begin{gather*}
f_1=e_1, \quad f_2=\frac{c(c+v_1)-q^2}{q^2(c+v_1)} g(v,e_3) e_1 + \frac{g(v,e_4)}{c+v_1}e_2 + e_3, \\ f_3=\frac{c(c+v_1)-q^2}{q^2(c+v_1)} g(v,e_4) e_1 - \frac{g(v,e_3)}{c+v_1}e_2 + e_4, \quad f_4=-\frac{1}{q} Jv + e_5, \quad f_5=e_2, \quad f_6=\frac{1}{q}e_6
\end{gather*}
induces an isomorphism with the Lie algebra $\frs{4.7} \oplus \R^2$, a contradiction. Instead, if \eqref{condition_LCB} is non-zero, a similar change of basis, but with
\[
f_4=\frac{q^2(c+v_1)}{(c(c+v_1)-q^2)|v|^2-q^2v_1(c+v_1)} \left( Jv -qe_5\right),
\]
does the same with respect to the Lie algebra $\frs{5.16} \oplus \R$, again a contradiction, concluding the proof.
\end{proof}

\begin{theorem}
Let $\mathfrak{g}$ be a six-dimensional strongly unimodular almost nilpotent Lie algebra. Then, $\mathfrak{g}$ admits LCK structures if and only if it is isomorphic to one among
\begin{alignat*}{4}
&\frs{3.3}^0 \oplus \R^3, &\qquad
&\frs{5.13}^{0,0,r} \oplus \R,\, {\scriptstyle r>0}, &\qquad
&\frs{5.16} \oplus \R, &\qquad
&\frs{6.17}^{1,1,1,-4}, \\
&\frs{6.19}^{p,p,-4p,p}, \,{\scriptstyle p \neq 0},  &\qquad
&\frs{6.21}^{p,p,r,-4p}, \,{\scriptstyle p \neq 0,\, r>0}, &\qquad
&\frs{6.159},   &\qquad
&\frs{6.166}^p,\, {\scriptstyle 0<|p|\leq 1}.
\end{alignat*}
Explicit examples of LCK structures on these Lie algebras are provided in Tables \ref{table-exab} and \ref{table-exheis}.
\end{theorem}
\begin{proof}
In the almost abelian case, the result follows from \cite{Par}. If the nilradical $\mathfrak{n}$ of $\mathfrak{g}$ has one-dimensional commutator and $\mathfrak{g}$ admits an LCK structure $(J,g)$, we first assume that the complex structure satisfies $J\mathfrak{n}^1 \subset \mathfrak{n}$: then, following the steps in the proof of Theorem \ref{SKT_LCSKT_class}, by \cite{FP_alm}, $\mathfrak{g}$ admits an orthonormal basis $\{e_1,\ldots,e_6\}$, with $\mathfrak{n}=\R\left<e_1,\ldots,e_5\right>$, $\mathfrak{n}_1=\R\left<e_1\right>$ and $Je_1=e_2$, $Je_3=e_4$, $Je_5=e_6$, with respect to which one has \eqref{sub_adapted}, satisfying \eqref{sub_adapted_condition}.
The Lee form $\theta$ is given by \eqref{theta_sub} and its closure implies $v_2=0$, since a computation shows
\[
d\theta(e_2,e_5)=-v_2 m, \quad d\theta(e_3,e_4)=-v_2c, \quad d\theta(X,e_5)=-v_2 \gamma_2(X),\; X \in \mathfrak{k},
\]
and we want $\eta$ not to vanish. Now, we can explicitly compute
\begin{align*}
d\omega - \frac{\theta}{2} \wedge \omega =& -\frac{m}{2} e^{126} + \frac{c-v_1}{2} e^2 \wedge (e^{34}-e^{56}) - (J\gamma_1 - \gamma_2) \wedge e^{16} \\ &- \gamma_1 \wedge e^{26} - \gamma_2 \wedge e^{25} - Jv^\flat \wedge (e^{12}-e^{56}),
\end{align*}
so that the LCK condition forces
\[
m=0, \quad v_1=c, \quad v=0, \quad \alpha=\nu=0.
\]
We also obtain $c \neq 0$, in order for $\eta$ not to vanish.
Now, an explicit isomorphism with $\frs{5.16} \oplus \R$ is provided by the change of basis
\[
f_1=\frac{1}{c} e_1, \quad f_2=e_3, \quad f_3=e_4, \quad f_4=\frac{q}{c} e_5, \quad f_5=e_2, \quad f_6=\frac{1}{q}e_6.
\]
This rules out the existence of LCK structures on all the other Lie algebras of Theorem with nilradical with one-dimensional commutator and only admitting complex structures satisfying $J\mathfrak{n}^1 \subset \mathfrak{n}$, namely $\frs{4.6} \oplus \R^2$, $\frs{6.25}$ (here, the non-existence of LCB structures was already enough), $\frs{6.51}^{p,0}$, $\frs{6.158}$ and $\frs{6.164}^p$.

Now, we only need to study Lie algebras with Heisenberg-type nilradical with respect to complex structures satisfying $J\mathfrak{n}^1 \not\subset \mathfrak{n}$ and $\frs{6.145}^0$, $\frs{6.147}^0$, $\frs{6.152}$, $\frs{6.154}^0$, whose nilradical has at least two-dimensional commutator. As we have already done in the previous proofs, we do this by considering the structure equations of Table \ref{table-cpxheis}, by virtue of Proposition \ref{heis_cpx}, and those of Propositions \ref{cpx-145_147} and \ref{cpx-152_154}, together with the generic Hermitian metric \eqref{om}. Instead of computing the associated Lee form $\theta$, imposing its closure and the LCK condition $d\omega=\frac{\theta}{2} \wedge \omega$, we first exhibit the generic real closed $1$-form $\mu$ and then impose the condition $d\omega=\mu \wedge \omega$, proving that it always leads to a contradiction. Notice that, apart from $\mathfrak{h}_3 \oplus \frs{3.3}^0$, a contradiction is reached as soon as $\mu$ is required to vanish, as $\mu=0$ would imply $d\omega=0$ and we already know that none of these Lie algebras (again, except $\mathfrak{h}_3 \oplus \frs{3.3}^0$) admit K\"ahler structures: $\frs{6.164}^0$ does not admit 1\textsuperscript{st}-Gauduchon structures, while the other ones do not admit LCSKT structures, and both are generalizations of the K\"ahler condition. The computations are summarized in Table \ref{LCK_comp}.
\end{proof}

\begin{table}[H]
\begin{center}
\addtolength{\leftskip} {-2cm}
\addtolength{\rightskip}{-2cm}
\scalebox{0.85}{
\begin{tabular}{|l|l|l|}
\hline \xrowht{15pt}
Lie algebra & Generic real closed $1$-form $\mu$ & $d \omega - \mu \wedge \omega=\frac{1}{2} \xi_{jk\overline{m}} \alpha^{jk\overline{m}} + \frac{1}{2} \overline{\xi_{jk\overline{m}}} \alpha^{m \overline{jk}}$\\
\hline 
\hline \xrowht{25pt}
$\mathfrak{h}_3 \oplus \frs{3.3}^0$ & $iy(\alpha^1-\alpha^{\overline{1}})+\zeta \alpha^3 + \overline{\zeta} \alpha^{\overline{3}}$ & \hspace{-0.8em} $\begin{array}{l} \xi_{12\overline{2}}=y\lambda_2, \; \xi_{23\overline{2}}=i\zeta \lambda_2 \; \implies \, y=\zeta=0, \\ \xi_{13\overline{3}}=\varepsilon \lambda_1 \end{array}$ \\
\hline \xrowht{25pt}
$\frs{4.7} \oplus \R^2$ & $iy(\alpha^1-\alpha^{\overline{1}})+\zeta \alpha^3 + \overline{\zeta} \alpha^{\overline{3}}$ & \hspace{-0.8em} $\begin{array}{l} \xi_{12\overline{2}}=\varepsilon \lambda_1 + y\lambda_2, \; \xi_{13\overline{1}}=i(\zeta \lambda_1 + y \overline{w}_2) \; \implies \, y=-\varepsilon \frac{\lambda_1}{\lambda_2},\; \zeta= \frac{\varepsilon}{\lambda_2} \overline{w}_2, \\ \xi_{13\overline{3}}=-\frac{\varepsilon}{\lambda_2}(\lambda_1\lambda_3-|w_2|^2) \end{array}$ \\
\hline \xrowht{20pt}
$\frs{6.44}$ & $iy(\alpha^1-\alpha^{\overline{1}})$ & $\xi_{12\overline{2}}=y\lambda_2$ \\
\hline \xrowht{20pt}
$\frs{6.52}^{0,q}$, $q>0$ & $iy(\alpha^1-\alpha^{\overline{1}})$ & $\xi_{13\overline{3}}=y\lambda_3$ \\
\hline \xrowht{20pt}
$\frs{6.145}^0$ & $\zeta \alpha^3 + \overline{\zeta} \alpha^{\overline{3}}$ & $\xi_{13\overline{1}}=i\zeta \lambda_1$ \\
\hline \xrowht{20pt}
\multirow{3}{1em}{\parbox{1\linewidth}{\vspace{3.5em}\hbox{$\frs{6.147}^0$}}} & $\zeta \alpha^3 + \overline{\zeta} \alpha^{\overline{3}}$ & $\xi_{13\overline{1}}=i\zeta \lambda_1$ \\
\cline{2-3} \xrowht{20pt}
 & $\zeta \alpha^3 + \overline{\zeta} \alpha^{\overline{3}}$ & $\xi_{12\overline{3}}=-i\lambda_1$ \\
\cline{2-3} \xrowht{20pt}
 & $\zeta \alpha^3 + \overline{\zeta} \alpha^{\overline{3}}$ & $\xi_{12\overline{3}}=-i\lambda_1$ \\
\hline \xrowht{20pt}
\multirow{2}{1em}{\vspace{-1.2em}\parbox{1\linewidth}{\hbox{$\frs{6.152}$}}} & $iy(\alpha^2-\alpha^{\overline{2}})-iy\Re(z_2)(\alpha^3-\alpha^{\overline{3}})$ & $\xi_{12\overline{1}}= -y\lambda_1$ \\
\cline{2-3} \xrowht{20pt} & $iy(\alpha^2-\alpha^{\overline{2}})-iy\Re(z_1)(\alpha^3-\alpha^{\overline{3}})$ & $\xi_{12\overline{1}}= -y\lambda_1$ \\
\hline \xrowht{20pt}
$\frs{6.154}^0$ & $iy(\alpha^2-\alpha^{\overline{2}})-iy\Re(z_1)(\alpha^3-\alpha^{\overline{3}})$ & $\xi_{12\overline{1}}=-y\lambda_1$ \\
\hline \xrowht{25pt}
$\frs{6.162}^1$ & $iy(\alpha^1-\alpha^{\overline{1}})$ & \hspace{-0.8em} $\begin{array}{l} \xi_{12\overline{2}}=(y-2)\lambda_2 \; \implies \, y=2, \\ \xi_{13\overline{3}}=4\lambda_3 \end{array}$ \\
\hline \xrowht{25pt}
$\frs{6.165}^p$, $p>0$ & $iy(\alpha^1-\alpha^{\overline{1}})$ & \hspace{-0.8em} $\begin{array}{l} \xi_{12\overline{2}}=(y-2p)\lambda_2 \; \implies \, y=2p, \\ \xi_{13\overline{3}}=4p\lambda_3 \end{array}$ \\
\hline \xrowht{20pt}
$\frs{6.167}$ & $iy(\alpha^1-\alpha^{\overline{1}})$ & $\xi_{12\overline{2}}=y\lambda_2$ \\
\hline
\end{tabular}}
\caption{Six-dimensional strongly unimodular (non-almost abelian) almost nilpotent Lie algebras not admitting LCK structures (excluding those with Heisenberg-type nilradical and only admitting complex structures satisfying $J\mathfrak{n}^1 \subset \mathfrak{n}$).} \label{LCK_comp}
\end{center}
\end{table}

\section{Strongly Gauduchon and balanced structures} \label{sec_bal}
A  Hermitian structure $(J,g)$ on a $2n$-dimensional smooth manifold is called Gauduchon when it satisfies $\partial \overline\partial \omega=0$. Clearly, the Gauduchon condition generalizes the balanced condition $d \omega^{n-1}=0$.

In \cite{Xiao}, the author introduced the \emph{strongly Gauduchon} condition, defined by the condition $\partial \omega^{n-1} = \overline\partial \beta$, for some $(n,n-2)$-form $\beta$. The strongly Gauduchon sits in between the Gauduchon and the balanced conditions, strengthening the former and generalizing the latter.

Speaking of balanced structures in the Lie algebra setting, we mention the classification result in the six-dimensional nilpotent case obtained in \cite{Uga} and the characterization results for the almost abelian case in \cite{FP_bal}, yielding a full classification result in six dimensions. Partial results in the almost nilpotent case with Heisenberg-type nilradical were obtained in \cite{FP_alm}.

In the next theorem, we provide a full classification result for six-dimensional strongly unimodular almost nilpotent Lie algebras admitting balanced structures, showing that the generalization provided by the strongly Gauduchon condition is not reflected in this classification result.

\begin{theorem} \label{StrG_bal_class}
Let $\mathfrak{g}$ be a six-dimensional strongly unimodular almost nilpotent Lie algebra. Then, $\mathfrak{g}$ admits strongly Gauduchon structures if and only if it admits balanced structures if and only if it is isomorphic to one among
\begin{alignat*}{4}
&\frs{3.3}^0 \oplus \R^3,   &\qquad
&\frs{5.8}^0 \oplus \R,  &\qquad
&\frs{5.9}^{1,-1,-1} \oplus \R,  &\qquad
&\frs{5.11}^{p,p,-p} \oplus \R,\, {\scriptstyle p>0},   \\
&\frs{5.13}^{0,0,r} \oplus \R, \, {\scriptstyle r>0},   &\qquad
&\frs{5.16} \oplus \R,  &\qquad
&\frs{6.145}^{0}, &\qquad
&\frs{6.147}^{0}, \\
&\frs{6.159},  &\qquad
&\frs{6.162}^{1},   &\qquad
&\frs{6.165}^p, {\scriptstyle p>0},  &\qquad
&\frs{6.166}^p,\, {\scriptstyle 0<|p|\leq 1},  \\
&\frs{6.167}.
\end{alignat*}
Explicit examples of balanced structures on such Lie algebras are provided in Tables \ref{table-exab}, \ref{table-exheis} and \ref{table-exnew}. 
\end{theorem}
\begin{proof}
We start with the almost abelian case, proving that a six-dimensional (not necessarily unimodular) almost abelian Lie algebra $\mathfrak{g}$ admits strongly Gauduchon structures if and only if it admits balanced structures. Assume $(J,g)$ is a strongly Gauduchon structure on $\mathfrak{g}$. Then, by \cite{FP_alm}, $\mathfrak{g}$ admits an orthonormal basis $\{e_1,\ldots,e_6\}$ such that $\mathfrak{n}=\R \left< e_1,\ldots,e_5\right>$ is an abelian ideal and $J$ satisfies $Je_1=e_6$, $Je_2=e_3$, $Je_4=e_5$. In such a basis, the Lie bracket of $\mathfrak{g}$ is encoded by the adjoint action of $e_{6}$ on $\mathfrak{n}$, which has the following matrix form, with respect to the chosen basis:
\[
\text{ad}_{e_6}\rvert_{\mathfrak{n}}=\begin{pmatrix}
a & 0 & 0 & 0 & 0 \\
v_1 & A_{11} & A_{12} & A_{13} & A_{14} \\
v_2 & -A_{12} & A_{11} & -A_{14} & A_{13} \\
v_3 & A_{31} & A_{32} & A_{33} & A_{34} \\
v_4 & -A_{32} & A_{31} & -A_{34} & A_{33} \end{pmatrix}, 
\]
with $a,v_k,A_{jk}$, $k=1,2,3,4$, $j=1,3$. We denote $\mathfrak{k} \coloneqq \mathfrak{n} \cap J\mathfrak{n}=\R \left<e_2,e_3,e_4,e_5\right>$, $v \coloneqq (v_1,v_2,v_3,v_4)^t \in \mathfrak{k}$, $A \coloneqq \text{ad}_{e_6}\rvert_{\mathfrak{k}} \in \mathfrak{gl}(\mathfrak{k})$. We consider the basis of $(1,0)$-forms defined by
\[
\alpha^1=\frac{1}{2}(e^1+ie^6), \quad \alpha^2=\frac{1}{2}(e^2+ie^3), \quad \alpha^3=\frac{1}{2}(e^4+ie^5),
\]
so that the fundamental form
\[
\omega=\frac{i}{2} \alpha^{1\overline{1}} + \frac{i}{2} \alpha^{2\overline{2}} + \frac{i}{2} \alpha^{3\overline{3}}
\]
satisfies
\begin{equation} \label{delom2}
\partial \omega^2 = \frac{i}{4} \alpha^{123} \wedge \left( -(v_3-iv_4) \alpha^{\overline{1}\overline{2}} + (v_1-iv_2) \alpha^{\overline{1}\overline{3}} + (\operatorname{tr} A) \alpha^{\overline{2} \overline{3}} \right),
\end{equation}
while the generic $(3,1)$-form
\[
\beta= \alpha^{123} \wedge (z_1 \alpha^{\overline{1}} + z_2 \alpha^{\overline{2}} + z_3 \alpha^{\overline{3}} ), \quad z_1,z_2,z_3 \in \C,
\]
satisfies
\begin{equation} \label{delbbeta}
\begin{aligned}
\overline\partial \beta=&\frac{i}{2} \alpha^{123\overline{1}} \wedge \left( \left( (A_{33}-iA_{14}+2A_{11}+a)z_2 + (A_{31}+iA_{32})z_3 \right) \alpha^{\overline{2}} \right.\\
&\left. + \left( (A_{13}+iA_{14})z_2 + (A_{11}-iA_{12}+2A_{22}+a)z_3 \right) \alpha^{\overline{3}} \right)
\end{aligned}
\end{equation}
Equating \eqref{delom2} and \eqref{delbbeta} yields $\operatorname{tr} A=0$ and
\begin{align*}
v=\begin{pmatrix} v_1 \\ v_2 \\ v_3 \\v_4 \end{pmatrix} &= \begin{pmatrix}
2(a-A_{11})\Re(z_3) + 2A_{12}\Im(z_3) + 2A_{13}\Re(z_2)-2A_{14}\Im(z_2) \\
-2(a-A_{11})\Im(z_3) + 2A_{12}\Re(z_3) -2A_{13}\Im(z_2) -2A_{14}\Re(z_2) \\
-2(a+A_{11})\Re(z_2) -2A_{31} \Re(z_3) + 2A_{32}\Im(z_3)-2A_{34}\Im(z_2) \\
2(a+A_{11})\Im(z_2) + 2A_{31}\Im(z_3) + 2A_{32}\Re(z_3)-2A_{34}\Re(z_2)
\end{pmatrix} \\
&= (A-a\,\text{Id}\rvert_{\mathfrak{k}})X, \;X= \begin{pmatrix} -2\Re(z_3) \\ 2\Im(z_3) \\ 2\Re(z_2) \\ -2\Im(z_2) \end{pmatrix},
\end{align*}
so that the new $J$-Hermitian metric for which the basis $\{e_1-X,e_2,e_3,e_4,e_5,e_6-JX\}$ is balanced, as the matrix associated with $\text{ad}_{e_6-JX}\rvert_{\mathfrak{k}}$ is of the form
\[
\text{ad}_{\tilde{e_6}}\rvert_{\mathfrak{k}} = \begin{pmatrix} a & 0 \\ 0 & A \end{pmatrix},
\]
with $a$ and $A$ being the same as before: in particular, $\operatorname{tr} A=0$. The classification of six-dimensional almost abelian Lie algebras admitting balanced structures was obtained in \cite{FP_bal}: the unimodular ones are listed in the statement.

We now turn our attention to Lie algebras with Heisenberg-type nilradical $\mathfrak{n}$, assuming the existence of a strongly Gauduchon structure $(J,g)$ with complex structure satisfying $J\mathfrak{n}^1 \subset \mathfrak{n}$: following the steps in the proof of Theorem \ref{SKT_LCSKT_class}, by \cite{FP_alm}, $\mathfrak{g}$ admits an orthonormal basis $\{e_1,\ldots,e_6\}$, with $\mathfrak{n}=\R\left<e_1,\ldots,e_5\right>$, $\mathfrak{n}_1=\R\left<e_1\right>$ and $Je_1=e_2$, $Je_3=e_4$, $Je_5=e_6$, with respect to which one has \eqref{sub_adapted}, with \eqref{sub_adapted_condition}.
In what follows, we write
\[
v=v_3e_3+v_4e_4, \quad \gamma_1=\gamma_{1,3}e^3+\gamma_{1,4}e^4, \quad \gamma_2=\gamma_{2,3}e^3+\gamma_{2,4}e^4, \quad v_k,\gamma_{j,k} \in \R,\;j=1,2,\,k=3,4.
\]
We now consider the basis of $(1,0)$-forms provided by
\[
\alpha^1=\frac{1}{2}(e^1+ie^2), \quad \alpha^2=\frac{1}{2}(e^3+ie^4), \quad \alpha^3=\frac{1}{2}(e^5+ie^6).
\]
We can compute
\begin{equation} \label{delom2_sub}
\partial\omega^2=\frac{i}{4} \alpha^{123} \wedge \left( m \alpha^{\overline{1}\overline{2}} + (v_3 - i v_4) \alpha^{\overline{1}\overline{3}} + (-v_1+iv_2-c) \alpha^{\overline{2}\overline{3}} \right).
\end{equation}
Taking the generic $(3,1)$-form
\[
\beta= \alpha^{123} \wedge (z_1 \alpha^{\overline{1}} + z_2 \alpha^{\overline{2}} + z_3 \alpha^{\overline{3}} ), \quad z_1,z_2,z_3 \in \C,
\]
we have
\begin{equation} \label{delbbeta_sub}
\overline\partial \beta=\frac{1}{2} \alpha^{123\overline{3}} \left( qz_1 \alpha^{\overline{1}} + \left( (\gamma_{2,3}-\gamma_{1,4} + i(\gamma_{2,4}+\gamma_{1,3}))z_1 - imz_2 \right) \alpha^{\overline{2}} \right).
\end{equation}
Equating \eqref{delom2_sub} and \eqref{delbbeta_sub} yields
\begin{gather*}
m=0, \quad v_1=-2mx_2+2\gamma_{1,3}x_1+2\gamma_{2,4}x_1-c, \quad v_2=-2\gamma_{1,4}x_1+2\gamma_{2,3}x_1,\\ v_3=0, \quad v_4=-2qx_1.
\end{gather*}
\begin{gather*}
m=0, \quad v_1-iv_2=-2i(\gamma_{2,3}-\gamma_{1,4} + i(\gamma_{2,4}+\gamma_{1,3}))z_1 -c, \quad
v_3-iv_4=2iqz_1.
\end{gather*}
The condition \eqref{sub_adapted_condition} now reads
\[
2cqz_1+iq(\gamma_{2,3}-i\gamma_{2,4})=0.
\]
We observe that $q$ cannot vanish, otherwise $\mathfrak{g}$ would be nilpotent, meaning we must set
\[
\gamma_{2,3}-i\gamma_{2,4}=-2icz_1.
\]
To recap, we now have
\begin{gather*}
\text{ad}_{e_6}\rvert_{\mathfrak{n}}=\begin{pmatrix}
		0 & 0 & \gamma_{1,3} & \gamma_{1,4} & -4c|z_1|^2+2\gamma_{1,3}\Re(z_1)-2\gamma_{1,4}\Im(z_1)-c \\
		0 & 0 & -2c\Im(z_1)-\gamma_{1,4} & -2c\Re(z_1)+\gamma_{1,3} & -2\gamma_{1,3}\Im(z_1)-2\gamma_{1,4}\Re(z_1) \\
		0 & 0 & 0 & q & -2q\Im(z_1) \\
		0 & 0 & -q & 0 & -2q\Re(z_1) \\
        0 & 0 & 0 & 0 & 0
	\end{pmatrix}, \\ \eta=c\,e^{34} -2c\Im(z_1)e^{35} -2c\Re(z_1) e^{45}.
\end{gather*}
Now, observing that $c$ cannot vanish (otherwise, $\eta=0$), we can perform a change of basis, making apparent the isomorphism between $\mathfrak{g}$ and $\frs{5.16} \oplus \R$, which admits balanced structures, by \cite{FP_alm}:
\begin{align*}
f_1=&e_1, \\
f_2=&\frac{1}{q} (\gamma_{1,4}e_1 + (\gamma_{1,3}-2c\Re(z_1))e_2) + e_3, \\
f_3=&\frac{1}{q} (-\gamma_{1,3}e_1 + (\gamma_{1,4}+2c\Im(z_1)e_2)+ e_4, \\
f_4=&\frac{q}{c(4|z_1|^2+1)} (2\Re(z_1)e_3-2\Im(z_1)e_4-e_5), \\ f_5=&e_2, \\
f_6=&\frac{1}{q} e_6.
\end{align*}

It remains to examine Lie algebras with Heisenberg-type nilradical with respect to complex structures satisfying $J\mathfrak{n}^1 \not\subset \mathfrak{n}$ and Lie algebras with nilradical satisfying $\dim \mathfrak{n}^1 > 1$. Following what we have done in the previous proofs, we can consider the structure equations of Table \ref{table-cpxheis}, thanks to Proposition \ref{heis_cpx}, and those of Propositions \ref{cpx-145_147} and \ref{cpx-152_154}, together the generic Hermitian metric of the form \eqref{om}. Focusing on the Lie algebras not appearing in the statement, we impose the strongly Gauduchon condition $\partial (\omega^2) = \overline\partial \beta$ for the generic $(3,1)$-form
\[
\beta= \alpha^{123} \wedge (z_1 \alpha^{\overline{1}} + z_2 \alpha^{\overline{2}} + z_3 \alpha^{\overline{3}} ), \quad z_1,z_2,z_3 \in \R,
\]
and show that we get to a contradiction. Table \ref{StrG_comp} summarizes these computations.
\end{proof}

\begin{table}[H]
\begin{center}
\addtolength{\leftskip} {-2cm}
\addtolength{\rightskip}{-2cm}
\scalebox{0.85}{
\begin{tabular}{|l|l|}
\hline \xrowht{15pt}
Lie algebra & $\partial \omega^2 - \overline\partial \beta = \frac{1}{2} \xi_{123\overline{j}\overline{k}} \alpha^{123\overline{j}\overline{k}}$\\
\hline
\hline \xrowht{20pt}
$\mathfrak{h}_3 \oplus \frs{3.3}^0$ & $\xi_{123\overline{2}\overline{3}}=-2i\varepsilon (\lambda_1\lambda_2 -|w_3|^2)$ \\
\hline \xrowht{20pt}
$\frs{4.7} \oplus \R^2$ & $\xi_{123\overline{2}\overline{3}}=-2i\varepsilon (\lambda_1 \lambda_3 - |w_2|^2)$ \\
\hline \xrowht{20pt}
$\frs{6.44}$ & $\xi_{123\overline{2}\overline{3}}=-2i\varepsilon (\lambda_1 \lambda_2 - |w_3|^2)$ \\
\hline \xrowht{20pt}
$\frs{6.52}^{0,q}$, $q>0$ & $\xi_{123\overline{2}\overline{3}}=-2i\delta (\lambda_1 \lambda_3 - |w_2|^2)$ \\
\hline \xrowht{20pt}
\multirow{2}{1em}{\vspace{0em} \parbox{1\linewidth}{\hbox{$\frs{6.152}$}}} & $\Re(z_2) \xi_{123\overline{1}\overline{2}} + \xi_{123\overline{1}\overline{3}}=-i\delta(\lambda_1 \lambda_2 - |w_3|^2)$  \\
\cline{2-2} \xrowht{20pt} &  $\Re(z_1) \xi_{123\overline{1}\overline{2}} + \xi_{123\overline{1}\overline{3}}=i\delta\Im(z_2)(\lambda_1 \lambda_2 - |w_3|^2)$  \\
\hline \xrowht{20pt}
$\frs{6.154}^0$ & $\Re(z)\xi_{123\overline{1}\overline{2}}+\xi_{123\overline{1}\overline{3}}=ix (\lambda_1 \lambda_2 - |w_3|^2)$ \\
\hline
\end{tabular}}
\caption{Six-dimensional strongly unimodular non-almost abelian almost nilpotent Lie algebras not admitting strongly Gauduchon structures (excluding those with Heisenberg-type nilradical and only admitting complex structures satisfying $J\mathfrak{n}^1 \subset \mathfrak{n}$).} \label{StrG_comp}
\end{center}
\end{table}

\begin{remark}
We note that it is possible to prove that all the Lie algebras of Theorem \ref{StrG_bal_class}, beside admitting balanced structures, also admit non-balanced strongly Gauduchon structures.
\end{remark}

\begin{remark}
The previous result provides two new classes of Lie algebras admitting balanced structures, namely $\frs{6.145}^0$ and $\frs{6.147}^0$, with the respect to the ones previously known in literature (see \cite{FP_alm, FP_bal}). Thanks to Remark \ref{rem_lattices}, this yields new examples of compact solvmanifolds admitting invariant balanced structures.
\end{remark}

Having studied the SKT condition and the balanced condition, we can now say something about K\"ahler condition.

\begin{corollary}
Let $\mathfrak{g}$ be a six-dimensional strongly unimodular almost nilpotent Lie algebra. Then, $\mathfrak{g}$ admits K\"ahler structures if and only if it admits both SKT and balanced structures if and only if it is isomorphic to one among
\begin{itemize} [leftmargin=0.5cm]
\item[] $\frs{3.3}^0 \oplus \R^3=\left(f^{26},-f^{16},0,0,0,0\right)$, \smallskip
\item[] $\frs{5.13}^{0,0,r} \oplus \R = \left(f^{26}, -f^{16},rf^{46},-rf^{36},0,0 \right)$, $r>0$.
\end{itemize}
Explicit examples of K\"ahler structures on these Lie algebras are exhibited in Table \ref{table-exab}. 
\end{corollary}
\begin{proof}
Comparing Theorems \ref{SKT_LCSKT_class} and \ref{StrG_bal_class}, the two Lie algebras of the statement are the only six-dimensional strongly unimodular almost nilpotent Lie algebras, up to isomorphism, which admit both SKT and balanced structures, which is a necessary condition for the existence of K\"ahler structures. The existence of K\"ahler structures on them was established in \cite{FP_gk} and is confirmed by the explicit examples in Table \ref{table-exab}, proving the claim.
\end{proof}

In \cite{FV}, the authors formulated a conjecture, according to which a compact complex manifold admitting both SKT and balanced metrics necessarily admits K\"ahler metrics. So far, such conjecture has been confirmed in several cases  \cite{AN, Chi, Fei, FGV, FKV, FP_alm, FV, FV1, FS, FLY, Oti, Ver}. Moreover, in \cite{FS2}, the authors provided the first examples of complex structures on non-unimodular Lie groups admitting both SKT and balanced metrics but no K\"ahler metrics.

We note that the previous theorem confirms this conjecture for six-dimensional almost nilpotent solvmanifolds.

\section{Appendix} \label{appendix}
This section features the tables mentioned throughout the article. Table \ref{table-str} summarizes the main results, providing, at a glance, a list of all six-dimensional strongly unimodular almost nilpotent Lie algebras admitting complex structures, up to isomorphism, along with their structure equations and the types of special Hermitian structures they admit. Parentheses indicate trivial results: for example, the Lie algebra $\frs{5.8}^0 \oplus \R$ admits balanced structures, so it trivially admits LCB structures.
Tables \ref{table-cpxheis} and \ref{table-cpxab} were referenced in Propositions \ref{heis_cpx} and \ref{almab_cpx} and provide the classification of complex structures, up to automorphisms, on some of the Lie algebras of Table \ref{table-str} (with some further conditions on the possible Hermitian metrics in the almost abelian case). Finally, Tables \ref{table-exab}, \ref{table-exheis} and \ref{table-exnew} provide explicit examples of special Hermitian structures on each of the Lie algebras of Table \ref{table-str}. In order to simplify these final three tables, we recall the inclusions among the different types of special Hermitian metrics mentioned in the Introduction: for instance, the  balanced structure on $\frs{5.8}^0 \oplus \R$ is obviously also an example of LCB structure, but we omit writing it.

\newpage

\begin{table}[H]
\begin{center}
\addtolength{\leftskip} {-2cm}
\addtolength{\rightskip}{-2cm}
\scalebox{0.6}{
\begin{tabular}{|l|l|l|l|l|l|l|l|l|l|}
\hline \xrowht{15pt}
Name & Structure equations & Nilradical & K\"ahler & SKT & Balanced & LCK & LCSKT & LCB & 1\textsuperscript{st}-Gauduchon
 \\
\hline \hline  \xrowht{20pt}
$\frs{3.3}^0 \oplus \R^3$ &  $(f^{26},-f^{16},0,0,0,0)$ & $\R^5$ & \cmark & (\cmark) & (\cmark) & (\cmark) & (\cmark) & (\cmark) & (\cmark) 
\\ \hline  \xrowht{20pt}
$\frs{4.3}^{-\frac{1}{2},-\frac{1}{2}} \oplus \R^2$ &  $\left(f^{16},-\frac{1}{2}f^{26},-\frac{1}{2}f^{36},0,0,0\right)$ & $\R^5$ & -- & \cmark & -- & -- & -- & \cmark & (\cmark) 
   \\ \hline  \xrowht{20pt}
$\frs{4.5}^{p,-\frac{p}{2}} \oplus \R^2$ &  $\left(pf^{16},-\frac{p}{2}f^{26}+f^{36},-f^{26}-\frac{p}{2}f^{36},0,0,0\right)$, $p>0$ & $\R^5$ & -- & \cmark & -- & -- & -- & \cmark & (\cmark) 
   \\ \hline  \xrowht{20pt}
$\frs{5.4}^{0} \oplus \R$ &  $\left(f^{26},0,f^{46},-f^{36},0,0\right)$ & $\R^5$ & -- & \cmark & -- & -- & \cmark & -- & (\cmark) 
   \\ \hline  \xrowht{20pt}
$\frs{5.8}^{0} \oplus \R$ &  $\left(f^{26}+f^{36},-f^{16}+f^{46},f^{46},-f^{36},0,0\right)$ & $\R^5$ & -- & -- & \cmark & -- & -- & (\cmark) & --
   \\ \hline  \xrowht{20pt}
$\frs{5.9}^{1,-1,-1} \oplus \R$ &  $\left(f^{16},f^{26},-f^{36},-f^{46},0,0\right)$ & $\R^5$ & -- & -- & \cmark & -- & -- & (\cmark) & --
   \\ \hline  \xrowht{20pt}
$\frs{5.11}^{p,p,-p} \oplus \R$ &  $\left(pf^{16},pf^{26},-pf^{36}+f^{46},-f^{36}-pf^{46},0,0\right)$, $p>0$ & $\R^5$ & -- & -- & \cmark & -- & -- & (\cmark) & --
   \\ \hline  \xrowht{20pt}
$\frs{5.13}^{p,-p,r} \oplus \R$ &  $\left(pf^{16}+f^{26},-f^{16}+pf^{26},-pf^{36}+rf^{46},-rf^{36}-pf^{46},0,0\right)$, $r>0$ & $\R^5$ & $p=0$ & ($p=0$) & \cmark & ($p=0$) & ($p=0$) & (\cmark) & ($p=0$)
\\ \hline  \xrowht{20pt}
$\frs{6.14}^{-\frac{1}{4},-\frac{1}{4}}$ &  $\left(-\frac{1}{4}f^{16}+f^{26},-\frac{1}{4}f^{26},-\frac{1}{4}f^{36}+f^{46},-\frac{1}{4}f^{46},f^{56},0\right)$ & $\R^5$ & -- & -- & -- & -- & -- & \cmark & \cmark
\\ \hline  \xrowht{20pt}
$\frs{6.16}^{p,-4p}$ &  $\left(pf^{16}+f^{26}+f^{36},-f^{16}+pf^{26}+f^{46},pf^{36}+f^{46},-f^{36}+pf^{46},-4pf^{56},0\right)$, $p<0$ & $\R^5$ & -- & -- & -- & -- & -- & \cmark & \cmark
\\ \hline  \xrowht{20pt}
$\frs{6.17}^{1,q,q,-2(1+q)}$ &  $\left(f^{16},f^{26},qf^{36},qf^{46},-2(1+q)f^{56},0\right)$, $0<|q| \leq 1$, $q \neq 1$ & $\R^5$ & -- & -- & -- & $q=1$ & $q=1$ & \cmark & $q>0$
\\ \hline  \xrowht{20pt}
$\frs{6.18}^{1,-\frac{3}{2},-\frac{3}{2}}$ &  $\left(f^{16}+f^{26},f^{26},f^{36},-\frac{3}{2}f^{46},-\frac{3}{2}f^{56},0\right)$ & $\R^5$ & -- & -- & -- & -- & -- & -- & --
\\ \hline  \xrowht{20pt}
$\frs{6.19}^{p,p,q,-p-\frac{q}{2}}$ &  $\left(pf^{16},pf^{26},qf^{36},-\left(p+\frac{q}{2}\right)f^{46}+f^{56},-f^{46}-\left(p+\frac{q}{2}\right)f^{56},0\right)$, $p,q \neq 0$ & $\R^5$ & -- & $q=-2p$ & -- & $q=-4p$ & $q=-4p$ & \cmark & $p(2p+q) \leq 0$
\\ \hline  \xrowht{20pt}
$\frs{6.20}^{p,p,-\frac{3}{2}p}$ &  $\left(pf^{16}+f^{26},pf^{26},pf^{36},-\frac{3}{2}pf^{46}+f^{56},-f^{46}-\frac{3}{2}pf^{56},0\right)$, $p>0$ & $\R^5$ & -- & -- & -- & -- & -- & -- & --
\\ \hline  \xrowht{20pt}
$\frs{6.21}^{p,q,r,-2(p+q)}$ &  $\left(pf^{16}+f^{26},-f^{16}+pf^{26},qf^{36}+rf^{46},-rf^{36}+qf^{46},-2(p+q)f^{56},0\right)$, $|p| \geq |q|$, $q \neq -p$, $r>0$ & $\R^5$ & -- & $q=0$ & -- & $q=p$ & $q=p$ & \cmark & $pq \geq 0$
   \\ \hline  \xrowht{20pt}
$\mathfrak{h}_3 \oplus \frs{3.3}^{0}$ & $(f^{23},0,0,f^{56},-f^{46},0)$ & $\mathfrak{h}_3 \oplus \R^2$ & -- & \cmark & -- & -- & \cmark & \cmark & (\cmark)
   \\ \hline \xrowht{20pt}
$\frs{4.6} \oplus \R^2$ & $(f^{23},f^{26},-f^{36},0,0,0)$ & $\mathfrak{h}_3 \oplus \R^2$ & -- & \cmark & -- & -- & -- & \cmark & (\cmark)
   \\ \hline \xrowht{20pt}
$\frs{4.7} \oplus \R^2$ & $(f^{23},f^{36},-f^{26},0,0,0)$ & $\mathfrak{h}_3 \oplus \R^2$ & -- &\cmark & -- & -- & -- & \cmark & (\cmark)
   \\ \hline \xrowht{20pt}
$\frs{5.16} \oplus \R$ & $(f^{23}+f^{46},f^{36},-f^{26},0,0,0)$ & $\mathfrak{h}_3 \oplus \R^2$ &-- & -- & \cmark & \cmark & -- & (\cmark) & \cmark
\\ \hline  \xrowht{20pt}
$\frs{6.25}$ & $(f^{23},f^{36},-f^{26},0,f^{46},0)$ & $\mathfrak{h}_3 \oplus \R^2$ &-- & \cmark & -- & -- & -- & -- & (\cmark)
\\ \hline  \xrowht{20pt}
$\frs{6.44}$ & $(f^{23},f^{36},-f^{26},f^{26}+f^{56},f^{36}-f^{46},0)$ & $\mathfrak{h}_3 \oplus \R^2$ &-- & -- & -- & -- & -- & \cmark & --
\\ \hline \xrowht{20pt}
$\frs{6.51}^{p,0}$ & $(f^{23},pf^{26},-pf^{36},f^{56},-f^{46},0)$, $p>0$ & $\mathfrak{h}_3 \oplus \R^2$ &-- & \cmark & -- & -- & -- & \cmark & (\cmark)
\\ \hline \xrowht{20pt}
$\frs{6.52}^{0,q}$ & $(f^{23},f^{36},-f^{26},qf^{56},-qf^{46},0)$, $q>0$ & $\mathfrak{h}_3 \oplus \R^2$ &-- & \cmark & -- & -- & -- & \cmark & (\cmark)
\\ \hline \xrowht{20pt}
$\frs{6.158}$ & $(f^{24}+f^{35},0,f^{36},0,-f^{56},0)$ & $\mathfrak{h}_5$ &-- & \cmark & -- & -- & -- & \cmark & (\cmark)
\\ \hline \xrowht{20pt}
$\frs{6.159}$ & $(f^{24}+f^{35},0,f^{56},0,-f^{36},0)$ & $\mathfrak{h}_5$ &-- & -- & \cmark & \cmark & -- & (\cmark) & \cmark
\\ \hline \xrowht{20pt}
$\frs{6.162}^1$ & $(f^{24}+f^{35},f^{26},f^{36},-f^{46},-f^{56},0)$ & $\mathfrak{h}_5$ &-- & -- & \cmark & -- & -- & (\cmark) & --
\\ \hline \xrowht{20pt}
$\frs{6.164}^p$ & $(f^{24}+f^{35},pf^{26},f^{56},-pf^{46},-f^{36},0)$, $p>0$ & $\mathfrak{h}_5$ &-- & \cmark & -- & -- & -- & \cmark & (\cmark)
\\ \hline \xrowht{20pt}
$\frs{6.165}^p$ & $(f^{24}+f^{35},pf^{26}+f^{36},-f^{26}+pf^{36},-pf^{46}+f^{56},-f^{46}-pf^{56},0)$, $p>0$ & $\mathfrak{h}_5$ &-- & -- & \cmark & -- & -- & (\cmark) & --
\\ \hline \xrowht{20pt}
$\frs{6.166}^p$ & $(f^{24}+f^{35},f^{46},pf^{56},-f^{26},-pf^{36},0)$, $0<|p| \leq 1$ & $\mathfrak{h}_5$ &-- & -- & \cmark & \cmark & -- & (\cmark) & $p \neq 1$
\\ \hline \xrowht{20pt}
$\frs{6.167}$ & $(f^{24}+f^{35},f^{36},-f^{26},f^{26}+f^{56},f^{36}-f^{46},0)$ & $\mathfrak{h}_5$ &-- & -- & \cmark & -- & -- & (\cmark) & --
\\ \hline \xrowht{20pt}
$\mathfrak{s}_{6.145}^{0}$ & $(f^{35}+f^{26},f^{45}-f^{16},f^{46},-f^{36},0,0)$ &$\mathfrak{n}_{5.1}$ &-- & -- & \cmark & -- & -- & (\cmark) & --
\\ \hline \xrowht{20pt}
$\mathfrak{s}_{6.147}^{0}$ & $(f^{35}+f^{26}+f^{36},f^{45}-f^{16}+f^{46},f^{46},-f^{36},0,0)$ &$\mathfrak{n}_{5.1}$   &-- & -- & \cmark & -- & -- & (\cmark) & --
\\ \hline \xrowht{20pt}
$\mathfrak{s}_{6.152}$ & $\left( f^{35}+f^{26},f^{34}-f^{16}+f^{56},f^{45},-f^{56},f^{46},0 \right)$ &$\mathfrak{n}_{5.2}$  &-- & -- & -- & -- & -- & \cmark & --
\\ \hline \xrowht{20pt}
$\mathfrak{s}_{6.154}^0$ & $\left( f^{35}+f^{26},f^{34}-f^{16},f^{45},-f^{56},f^{46},0 \right)$ &$\mathfrak{n}_{5.2}$   & -- &-- & -- & -- & \cmark & \cmark & --
\\ \hline
\end{tabular}}
\caption{Six-dimensional strongly unimodular almost nilpotent Lie algebras admitting complex structures.} \label{table-str}
\end{center}
\end{table}

\newpage

\begin{table}[H]
\begin{center}
\addtolength{\leftskip} {-2cm}
\addtolength{\rightskip}{-2cm}
\scalebox{1}{
\begin{tabular}{|l|l|l|}
\hline \xrowht{15pt}
Lie algebra & Complex structure equations & Conditions\\
\hline \hline
\xrowht{50pt}
$\mathfrak{h}_3 \oplus \frs{3.3}^0$ & $ \begin{cases} 
                                        d\alpha^1=i\varepsilon \alpha^{3 \overline{3}}, \\
                                        d\alpha^2=-\alpha^2 \wedge (\alpha^1-\alpha^{\overline{1}}), \\
                                        d\alpha^3=0.
                                        \end{cases}  $                                                        & $\varepsilon \in \{-1,1\}$ \\
\hline
\xrowht{50pt}
$\frs{4.7} \oplus \R^2$ & $ \begin{cases} 
                            d\alpha^1=i\varepsilon \alpha^{2 \overline{2}}, \\
                            d\alpha^2=-\alpha^2 \wedge (\alpha^1-\alpha^{\overline{1}}), \\
                            d\alpha^3=0.
                            \end{cases}  $                                                        & $\varepsilon \in \{-1,1\}$ \\
\hline
\xrowht{50pt}
$\frs{6.44}$ & $ \begin{cases} 
                 d\alpha^1=i \varepsilon \alpha^{3 \overline{3}}, \\
                 d\alpha^2=-(\alpha^2+i\alpha^3) \wedge (\alpha^1-\alpha^{\overline{1}}), \\
                 d\alpha^3=-\alpha^3 \wedge (\alpha^1-\alpha^{\overline{1}}).
                 \end{cases}  $                                                                  & $\varepsilon \in \{-1,1\}$ \\
\hline
\xrowht{50pt}
$\frs{6.52}^{0,q}$, $q>0$ & $ \begin{cases} 
                                      d\alpha^1=i\delta \alpha^{2 \overline{2}}, \\
                                      d\alpha^2=-\alpha^2 \wedge (\alpha^1-\alpha^{\overline{1}}), \\
                                      d\alpha^3=-\varepsilon q \alpha^3 \wedge (\alpha^1-\alpha^{\overline{1}}).
                                      \end{cases}  $                                                                & $\varepsilon \in \{-1,1\}$ \\
\hline
\xrowht{50pt}
$\frs{6.159}$             & $\begin{cases} 
                             d\alpha^1=  i \delta \alpha^{2 \overline{2}} + i \varepsilon \alpha^{3 \overline{3}}, \\
                             d\alpha^2=-\alpha^2 \wedge (\alpha^1-\alpha^{\overline{1}}), \\
                             d\alpha^3= 0.
                             \end{cases}  $                                                                & $\delta,\varepsilon \in \{-1,1\}$ \\
\hline
\xrowht{50pt}
$\frs{6.162}^1$             & $\begin{cases} 
                             d\alpha^1=  \alpha^{2 \overline{3}} -  \alpha^{3 \overline{2}}, \\
                             d\alpha^2=-i\alpha^2 \wedge (\alpha^1-\alpha^{\overline{1}}), \\
                             d\alpha^3= i\alpha^3 \wedge (\alpha^1-\alpha^{\overline{1}}).
                             \end{cases}  $                                                                & -- \\
\hline
\xrowht{50pt}
$\frs{6.165}^p$, $p>0$     & $\begin{cases} 
                             d\alpha^1=  \alpha^{2 \overline{3}} -  \alpha^{3 \overline{2}}, \\
                             d\alpha^2=-(1+ip)\,\alpha^2 \wedge (\alpha^1-\alpha^{\overline{1}}), \\
                             d\alpha^3= -(1-ip)\,\alpha^3 \wedge (\alpha^1-\alpha^{\overline{1}}).
                             \end{cases}  $                                                                & -- \\
\hline
\xrowht{50pt}
$\frs{6.166}^p$, $0<|p|\leq 1$ & $\begin{cases} 
                             d\alpha^1=  i\delta \alpha^{2 \overline{2}}  +i\varepsilon \delta \alpha^{3 \overline{3}}, \\
                             d\alpha^2=- \alpha^2 \wedge (\alpha^1-\alpha^{\overline{1}}), \\
                             d\alpha^3= -\varepsilon p \alpha^3 \wedge (\alpha^1-\alpha^{\overline{1}}).
                             \end{cases}  $                                                                & $\delta,\varepsilon \in \{-1,1\}$ \\
\hline
\xrowht{50pt}
$\frs{6.167}$ & $\begin{cases} 
                 d\alpha^1=  \varepsilon \alpha^{2 \overline{3}} - \varepsilon \alpha^{3 \overline{2}} +ix  \alpha^{3 \overline{3}}, \\
                 d\alpha^2=-(\alpha^2+i\alpha^3) \wedge (\alpha^1-\alpha^{\overline{1}}), \\
                 d\alpha^3=-\alpha^3 \wedge (\alpha^1-\alpha^{\overline{1}}).
                 \end{cases}  $                                                                & $\varepsilon \in \{-1,1\}$, $x \in \R$ \\
\hline
\end{tabular}}
\caption{Complex structures satisfying $J\mathfrak{n}^1 \not\subset \mathfrak{n}$, up to automorphisms, on six-dimensional strongly unimodular almost nilpotent Lie algebras with nilradical having one-dimensional commutator.} \label{table-cpxheis}
\end{center}
\end{table}

\newpage

\begin{table}[H]
\begin{center}
\addtolength{\leftskip} {-2cm}
\addtolength{\rightskip}{-2cm}
\scalebox{0.65}{
\begin{tabular}{|l|l|l|}
\hline \xrowht{15pt}
Lie algebra & Complex structure equations & Conditions\\
\hline
\hline \xrowht{50pt}
$\frs{3.3}^0 \oplus \R^3$ & $\begin{cases} 
                              d\alpha^1=0, \\
                              d\alpha^2=-i\alpha^2 \wedge (\alpha^1-\alpha^{\overline{1}})+z \alpha^{1 \overline{1}}, \\
                              d\alpha^3=0.
                              \end{cases}  $                                                        & $z \in \C$ \\
\hline \xrowht{50pt}
$\frs{4.3}^{-\frac{1}{2},-\frac{1}{2}} \oplus \R^3$ & $\begin{cases} 
                                                      d\alpha^1=i \alpha^{1 \overline{1}}, \\
                                                      d\alpha^2=\frac{i}{2} \alpha^2 \wedge (\alpha^1-\alpha^{\overline{1}})+z_1 \alpha^{1 \overline{1}}, \\
                                                      d\alpha^3=z_2 \alpha^{1 \overline{1}}.
                                                      \end{cases}  $                                                        & $z_1,z_2 \in \C$ \\
\hline \xrowht{50pt}
$\frs{4.5}^{p,-\frac{p}{2}} \oplus \R^2$, $p>0$ & $\begin{cases} 
                                            d\alpha^1=i p \alpha^{1 \overline{1}}, \\
                                            d\alpha^2=-\left(1+\frac{i}{2}p\right) \alpha^2 \wedge (\alpha^1-\alpha^{\overline{1}})+z_1 \alpha^{1 \overline{1}}, \\
                                            d\alpha^3=z_2 \alpha^{1 \overline{1}}.
                                            \end{cases}$                                                        & $z_1,z_2 \in \C$ \\
\hline \xrowht{50pt}
$\frs{5.4}^{0} \oplus \R$ & $\begin{cases} 
                                                      d\alpha^1=0, \\
                                                      d\alpha^2= -\alpha^2 \wedge (\alpha^1-\alpha^{\overline{1}})+z_1 \alpha^{1 \overline{1}}, \\
                                                      d\alpha^3=z_2 \alpha^{1 \overline{1}}.
                                                      \end{cases}  $                                                        & $z_1,z_2 \in \C$, $z_2 \neq 0$ \\
\hline \xrowht{50pt}
$\frs{5.8}^{0} \oplus \R$ & $\begin{cases} 
                                                      d\alpha^1=0, \\
                                                      d\alpha^2= -(\alpha^2+i\alpha^3) \wedge (\alpha^1-\alpha^{\overline{1}})+z_1 \alpha^{1 \overline{1}}, \\
                                                      d\alpha^3=- \alpha^3 \wedge (\alpha^1-\alpha^{\overline{1}}) + z_2 \alpha^{1 \overline{1}}.
                                                      \end{cases}  $                                                        & $z_1,z_2 \in \C$ \\
\hline \xrowht{50pt}
$\frs{5.9}^{1,-1,-1} \oplus \R$ & $\begin{cases} 
                                                      d\alpha^1=0, \\
                                                      d\alpha^2= -\alpha^2 \wedge (\alpha^1-\alpha^{\overline{1}})+z_1 \alpha^{1 \overline{1}}, \\
                                                      d\alpha^3= i \alpha^3 \wedge (\alpha^1-\alpha^{\overline{1}}) + z_2 \alpha^{1 \overline{1}}.
                                                      \end{cases}  $                                                        & $z_1,z_2 \in \C$ \\
\hline \xrowht{50pt}
$\frs{5.11}^{p,p,-p} \oplus \R$, $p>0$ & $\begin{cases} 
                                                      d\alpha^1=0, \\
                                                      d\alpha^2= -ip\alpha^2 \wedge (\alpha^1-\alpha^{\overline{1}})+z_1 \alpha^{1 \overline{1}}, \\
                                                      d\alpha^3= (-1+ip) \alpha^3 \wedge (\alpha^1-\alpha^{\overline{1}}) + z_2 \alpha^{1 \overline{1}}.
                                                      \end{cases}  $                                                        & $z_1,z_2 \in \C$ \\
\hline \xrowht{50pt}
$\frs{5.13}^{p,-p,r} \oplus \R$, $r>0$ & $\begin{cases} 
                                                      d\alpha^1=0, \\
                                                      d\alpha^2= -(1+ip) \alpha^2 \wedge (\alpha^1-\alpha^{\overline{1}})+z_1 \alpha^{1 \overline{1}}, \\
                                                      d\alpha^3= (-\varepsilon r + ip) \alpha^3 \wedge (\alpha^1-\alpha^{\overline{1}}) + z_2 \alpha^{1 \overline{1}}.
                                                      \end{cases}  $                                                        & $z_1,z_2 \in \C$, $\varepsilon \in \{-1,1\}$ \\
\hline \xrowht{50pt}
$\frs{6.14}^{-\frac{1}{4},-\frac{1}{4}}$ & $\begin{cases} 
                                                      d\alpha^1=i \alpha^{1\overline{1}}, \\
                                                      d\alpha^2= i \left(\frac{1}{4} \alpha^2 - \alpha^3 \right) \wedge (\alpha^1-\alpha^{\overline{1}})+z_1 \alpha^{1 \overline{1}}, \\
                                                      d\alpha^3= \frac{i}{4} \alpha^3 \wedge (\alpha^1-\alpha^{\overline{1}}) + z_2 \alpha^{1 \overline{1}}.
                                                      \end{cases}  $                                                        & $z_1,z_2 \in \C$ \\
\hline \xrowht{50pt}
$\frs{6.16}^{p,-4p}$, $p<0$ & $\begin{cases} 
                                                      d\alpha^1=-4ip \alpha^{1\overline{1}}, \\
                                                      d\alpha^2=- i \left( (p-i) \alpha^2 + \alpha^3 \right) \wedge (\alpha^1-\alpha^{\overline{1}})+z_1 \alpha^{1 \overline{1}}, \\
                                                      d\alpha^3= -(1+ip) \alpha^3 \wedge (\alpha^1-\alpha^{\overline{1}}) + z_2 \alpha^{1 \overline{1}}.
                                                      \end{cases}  $                                                        & $z_1,z_2 \in \C$ \\
\hline \xrowht{50pt}
$\frs{6.17}^{1,q,q,-2(1+q)}$, $0<|q| \leq 1$, $q \neq -1$ & $\begin{cases} 
                                                      d\alpha^1=-2i(1+q) \alpha^{1\overline{1}}, \\
                                                      d\alpha^2= -i \alpha^2 \wedge (\alpha^1-\alpha^{\overline{1}})+z_1 \alpha^{1 \overline{1}}, \\
                                                      d\alpha^3= -iq \alpha^3 \wedge (\alpha^1-\alpha^{\overline{1}}) + z_2 \alpha^{1 \overline{1}}.
                                                      \end{cases}  $                                                        & $z_1,z_2 \in \C$, $z_2=0$ if $q=-\frac{2}{3}$ \\
\hline \xrowht{50pt}
$\frs{6.18}^{1,-\frac{3}{2},-\frac{3}{2}}$ & $\begin{cases} 
                                                      d\alpha^1=i \alpha^{1\overline{1}}, \\
                                                      d\alpha^2= -i \alpha^2  \wedge (\alpha^1-\alpha^{\overline{1}})+z_1 \alpha^{1 \overline{1}}, \\
                                                      d\alpha^3= \frac{3}{2} i \alpha^3 \wedge (\alpha^1-\alpha^{\overline{1}}) + z_2 \alpha^{1 \overline{1}}.
                                                      \end{cases}  $                                                        & $z_1,z_2 \in \C$, $z_1 \neq 0$ \\
\hline \xrowht{50pt}
$\frs{6.19}^{p,p,q,-p-\frac{q}{2}}$, $p,q \neq 0$ & $\begin{cases} 
                                                      d\alpha^1=iq \alpha^{1\overline{1}}, \\
                                                      d\alpha^2= -ip  \alpha^2 \wedge (\alpha^1-\alpha^{\overline{1}})+z_1 \alpha^{1 \overline{1}}, \\
                                                      d\alpha^3= i \left(p+\frac{q}{2} + i \right) \alpha^3 \wedge (\alpha^1-\alpha^{\overline{1}}) + z_2 \alpha^{1 \overline{1}}.
                                                      \end{cases}  $                                                        & $z_1,z_2 \in \C$, $z_1=0$ if $p=q$ \\
\hline \xrowht{50pt}
$\frs{6.20}^{p,p,-\frac{3}{2}p}$, $p>0$ & $\begin{cases} 
                                                      d\alpha^1=ip \alpha^{1\overline{1}}, \\
                                                      d\alpha^2= -ip \alpha^2 \wedge (\alpha^1-\alpha^{\overline{1}})+z_1 \alpha^{1 \overline{1}}, \\
                                                      d\alpha^3= -\left( 1 + \frac{3}{2}ip \right) \alpha^3 \wedge (\alpha^1-\alpha^{\overline{1}}) + z_2 \alpha^{1 \overline{1}}.
                                                      \end{cases}  $                                                        & $z_1,z_2 \in \C$, $z_1 \neq 0$ \\
\hline \xrowht{50pt}
$\frs{6.21}^{p,q,r,-2(p+q)}$, $q \neq -p$, $r>0$ & $\begin{cases} 
                                                      d\alpha^1=-2i(p+q) \alpha^{1\overline{1}}, \\
                                                      d\alpha^2= -(\varepsilon + ip) \alpha^2 \wedge (\alpha^1-\alpha^{\overline{1}})+z_1 \alpha^{1 \overline{1}}, \\
                                                      d\alpha^3= -(r+ iq) \alpha^3 \wedge (\alpha^1-\alpha^{\overline{1}}) + z_2 \alpha^{1 \overline{1}}.
                                                      \end{cases}  $                                                        & $z_1,z_2 \in \C$, $\varepsilon \in \{-1,1\}$ \\
\hline
\end{tabular}}
\caption{Hermitian structures up to equivalence on six-dimensional unimodular almost abelian Lie algebras, with Hermitian metric \eqref{om_almab}.} \label{table-cpxab}
\end{center}
\end{table}

\newpage

\begin{table}[H]
\begin{center}
\addtolength{\leftskip} {-2cm}
\addtolength{\rightskip}{-2cm}
\scalebox{0.45}{
\begin{tabular}{|l|l|l|c|c|}
\hline \xrowht{15pt}
Lie algebra & Conditions & Structure equations & Structure type & Example \\ \hline
\hline \xrowht{48pt}
$\frs{3.3} \oplus \R^3$ & -- & $(f^{26},-f^{16},0,0,0,0)$ & K\"ahler & \mlines{$Jf_1=f_2$, $Jf_3=f_4$, $Jf_5=f_6$ \\ $\omega=f^{12}+f^{34}+f^{56}$ \medskip \\ $\mu=f^6$} \\
\hline \xrowht{40pt}
$\frs{4.3}^{-\frac{1}{2},-\frac{1}{2}} \oplus \R^2$ & \mlines{--} & $\left(f^{16},-\frac{1}{2}f^{26},-\frac{1}{2}f^{36},0,0,0\right)$ & \mlines{SKT \\ LCB} & \mlines{$Jf_1=f_6$, $Jf_2=f_3$, $Jf_4=f_5$ \\ $\omega=f^{16}+f^{23}+f^{45}$} \\
\hline \xrowht{40pt}
$\frs{4.5}^{p,-\frac{p}{2}} \oplus \R^2$ & $p>0$ & $\left(pf^{16},-\frac{p}{2}f^{26}+f^{36},-f^{26}-\frac{p}{2}f^{36},0,0,0\right)$ & \mlines{SKT \\ LCB} & \mlines{$Jf_1=f_6$, $Jf_2=f_3$, $Jf_4=f_5$ \\ $\omega=f^{16}+f^{23}+f^{45}$} \\
\hline \xrowht{48pt}
$\frs{5.4}^{0} \oplus \R$ & -- & $\left(f^{26},0,f^{46},-f^{36},0,0\right)$ & \mlines{SKT \\ LCSKT} & \mlines{$Jf_1=f_5$, $Jf_2=f_6$, $Jf_3=f_4$ \\ $\omega=f^{15}+f^{26}+f^{34}$ \medskip \\ $\mu=f^6$} \\
\hline \xrowht{40pt}
$\frs{5.8}^{0} \oplus \R$ & -- & $\left(f^{26}+f^{36},-f^{16}+f^{46},f^{46},-f^{36},0,0\right)$ & \mlines{Balanced} & \mlines{$Jf_1=f_2$, $Jf_3=f_4$, $Jf_5=f_6$ \\ $\omega=f^{12}+f^{34}+f^{56}$} \\
\hline \xrowht{40pt}
$\frs{5.9}^{1,-1,-1} \oplus \R$ & -- & $\left(f^{16},f^{26},-f^{36},-f^{46},0,0\right)$ & \mlines{Balanced} & \mlines{$Jf_1=f_2$, $Jf_3=f_4$, $Jf_5=f_6$ \\ $\omega=f^{12}+f^{34}+f^{56}$} \\
\hline \xrowht{40pt}
$\frs{5.11}^{p,-p,r} \oplus \R$ & $p>0$ & $\left(pf^{16},pf^{26},-pf^{36}+f^{46},-f^{36}-pf^{46},0,0\right)$ & \mlines{Balanced} & \mlines{$Jf_1=f_2$, $Jf_3=f_4$, $Jf_5=f_6$ \\ $\omega=f^{12}+f^{34}+f^{56}$} \\
\hline \xrowht{48pt}
\multirow{2}{*}{\vspace{-2em}$\frs{5.13}^{p,-p,r} \oplus \R$} & $p=0$, $r>0$ & $(f^{26},-f^{16},rf^{46},-rf^{36},0,0)$ & \mlines{K\"ahler} & \mlines{$Jf_1=f_2$, $Jf_3=f_4$, $Jf_5=f_6$ \\ $\omega=f^{12}+f^{34}+f^{56}$ \medskip \\ $\mu=f^6$} \\
\cline{2-5} \xrowht{40pt}
 & $p \neq 0$, $r>0$ & $(pf^{16}+f^{26},-f^{16}+pf^{26},-pf^{36}+rf^{46},-rf^{36}-pf^{46},0,0,0)$ & \mlines{Balanced} & \mlines{$Jf_1=f_2$, $Jf_3=f_4$, $Jf_5=f_6$ \\ $\omega=f^{12}+f^{34}+f^{56}$} \\
\hline \xrowht{40pt}
$\frs{6.14}^{-\frac{1}{4},-\frac{1}{4}}$ & -- & $\left(-\frac{1}{4}f^{16}+f^{26},-\frac{1}{4}f^{26},-\frac{1}{4}f^{36}+f^{46},-\frac{1}{4}f^{46},0,0\right)$ & \mlines{LCB \\ 1\textsuperscript{st}-Gauduchon} & \mlines{$Jf_1=f_3$, $Jf_2=f_4$, $Jf_5=f_6$ \\ $\omega=f^{13}+4f^{24}+f^{56}$} \\
\hline \xrowht{40pt}
$\frs{6.16}^{p,4p}$ & $p<0$ & $\left(pf^{16}+f^{26}+f^{36},-f^{16}+pf^{26}+f^{46},pf^{36}+f^{46},-f^{36}+pf^{46},-4pf^{56},0\right)$ & \mlines{LCB \\ 1\textsuperscript{st}-Gauduchon} & \mlines{$Jf_1=f_2$, $Jf_3=f_4$, $Jf_5=f_6$ \\ $\omega=4p^2f^{12}+f^{34}+f^{56}$} \\
\hline \xrowht{40pt}
\multirow{3}{*}{\vspace{-7em}$\frs{6.17}^{1,q,q,-2(1+q)}$} & $-1<q<0$ & $(f^{16},f^{26},qf^{36},qf^{46},-2(1+q)f^{56},0)$ & LCB & \mlines{$Jf_1=f_2$, $Jf_3=f_4$, $Jf_5=f_6$ \\ $\omega=f^{12}+f^{34}+f^{56}$} \\
\cline{2-5} \xrowht{40pt}
 & $0<q<1$ & $(f^{16},f^{26},qf^{36},qf^{46},-2(1+q)f^{56},0)$ & \mlines{ LCB \\ 1\textsuperscript{st}-Gauduchon} & \mlines{$Jf_1=f_2$, $Jf_3=f_4$, $Jf_5=f_6$ \\ $\omega=f^{12}+4f^{13}+4f^{24}+4\frac{(1+q)^2}{q}f^{34}+2f^{56}$} \\
\cline{2-5} \xrowht{48pt}
 & $q=1$ & $(f^{16},f^{26},f^{36},f^{46},-4f^{56},0)$ & \mlines{ LCK \\ LCSKT \\ 1\textsuperscript{st}-Gauduchon} & \mlines{$Jf_1=f_2$, $Jf_3=f_4$, $Jf_5=f_6$ \\ $\omega=f^{12}+4f^{13}+4f^{24}+16f^{34}+2f^{56}$ \medskip \\ $\mu=2f^6$} \\
\hline \xrowht{40pt}
$\frs{6.18}^{1,-\frac{3}{2},-\frac{3}{2}}$ & -- & $\left(f^{16}+f^{26},f^{26},f^{36},-\frac{3}{2}f^{46},-\frac{3}{2}f^{56},0\right)$ & Complex & $Jf_1=f_3$, $Jf_2=f_6$, $Jf_4=f_5$ \\
\hline \xrowht{40pt}
\multirow{4}{*}{\vspace{-11em}$\frs{6.19}^{p,p,q,-p-\frac{q}{2}}$} & $q=-2p \neq 0$ & $(pf^{16},pf^{26},-2pf^{36},f^{56},-f^{46},0)$ & \mlines{SKT \\ LCB} & \mlines{$Jf_1=f_2$, $Jf_3=f_6$, $Jf_4=f_5$ \\ $\omega=f^{12}+f^{36}+f^{45}$} \\
\cline{2-5} \xrowht{48pt}
 & $q=-4p \neq 0$ & $(pf^{16},pf^{26},-4pf^{36},pf^{46} + f^{56},-f^{46}+pf^{56},0)$ & \mlines{ LCK \\ LCSKT \\ 1\textsuperscript{st}-Gauduchon} & \mlines{$Jf_1=f_2$, $Jf_3=f_6$, $Jf_4=f_5$ \\ $\omega=f^{12}+f^{36}+f^{45}$ \medskip \\ $\mu=2pf^6$} \\
\cline{2-5} \xrowht{40pt}
 & $p(2p+q)<0$, $q \neq -4p$ & $(pf^{16},pf^{26},qf^{36},-(p+\frac{q}{2})f^{46} + f^{56},-f^{46}-(p+\frac{q}{2})f^{56},0)$ & \mlines{ LCB \\ 1\textsuperscript{st}-Gauduchon} & \mlines{$Jf_1=f_2$, $Jf_3=f_6$, $Jf_4=f_5$ \\ $\omega=f^{12}+8f^{14}+8f^{25}+4f^{36}-8\frac{q^2+4}{p(2p+q)}f^{45}$} \\
\cline{2-5} \xrowht{40pt}
 & $p(2p+q) > 0$, $q \neq 0$ & $(pf^{16},pf^{26},qf^{36},-(p+\frac{q}{2})f^{46} + f^{56},-f^{46}-(p+\frac{q}{2})f^{56},0)$ & LCB & \mlines{$Jf_1=f_2$, $Jf_3=f_6$, $Jf_4=f_5$ \\ $\omega=f^{12}+f^{36}+f^{45}$} \\
\hline \xrowht{40pt}
$\frs{6.20}^{p,p,-\frac{3}{2}p}$ & $p>0$ & $\left(pf^{16}+f^{26},pf^{26},pf^{36},-\frac{3}{2}pf^{46}+f^{56},-f^{56}-\frac{3}{2}pf^{56},0\right)$ & Complex & $Jf_1=f_3$, $Jf_2=f_6$, $Jf_4=f_5$ \\
\hline \xrowht{40pt}
\multirow{5}{*}{\vspace{-14em}$\frs{6.21}^{p,q,r,-2(p+q)}$} & $q=0$, $p \neq 0$, $r>0$ & $(pf^{16}+f^{26},-f^{16}+pf^{26},rf^{46},-rf^{36},-2pf^{56},0)$ & \mlines{SKT \\ LCB} & \mlines{$Jf_1=f_2$, $Jf_3=f_4$, $Jf_5=f_6$ \\ $\omega=f^{12}+f^{34}+f^{56}$} \\
\cline{2-5} \xrowht{48pt}
 & \multirow{2}{*}{\vspace{-4em}$q=p \neq 0$, $r>0$}  & \multirow{2}{*}{\vspace{-4em}$(pf^{16}+f^{26},-f^{16}+pf^{26},pf^{36}+rf^{46},-rf^{36}+pf^{46},-4pf^{56},0)$} & \mlines{LCK \\ LCSKT} & \mlines{$Jf_1=f_2$, $Jf_3=f_4$, $Jf_5=f_6$ \\ $\omega=f^{12}+f^{34}+f^{56}$ \medskip \\ $\mu=2pf^6$} \\
\cline{4-5} \xrowht{40pt}
 & & & 1\textsuperscript{st}-Gauduchon & \mlines{$Jf_1=f_2$, $Jf_3=f_4$, $Jf_5=f_6$ \\ $\omega=f^{12}+2f^{13}+2f^{24}+\frac{4p^2+(r-1)^2}{p^2}f^{34}+f^{56}$} \\
\cline{2-5} \xrowht{40pt}
 & $pq>0$, $|p|>|q|$, $r>0$ & $(pf^{16}+f^{26},-f^{16}+pf^{26},qf^{36}+rf^{46},-rf^{36}+qf^{46},-2(p+q)f^{56},0)$ & \mlines{ LCB \\ 1\textsuperscript{st}-Gauduchon} & \mlines{$Jf_1=f_2$, $Jf_3=f_4$, $Jf_5=f_6$ \\ $\omega=f^{12}+2f^{13}+2f^{24}+\frac{(p+q)^2+(r-1)^2}{pq}f^{34}+f^{56}$} \\
\cline{2-5} \xrowht{40pt}
 & $pq<0$, $|p|>|q|$, $r>0$ & $(pf^{16}+f^{26},-f^{16}+pf^{26},qf^{36}+rf^{46},-rf^{36}+qf^{46},-2(p+q)f^{56},0)$ & LCB & \mlines{$Jf_1=f_2$, $Jf_3=f_4$, $Jf_5=f_6$ \\ $\omega=f^{12}+f^{34}+f^{56}$} \\
\hline
\end{tabular}}
\caption{Examples of complex and special Hermitian structures on six-dimensional unimodular almost abelian Lie algebras.} \label{table-exab}
\end{center}
\end{table}

\newpage

\begin{table}[H]
\begin{center}
\addtolength{\leftskip} {-2cm}
\addtolength{\rightskip}{-2cm}
\scalebox{0.45}{
\begin{tabular}{|l|l|l|c|c|}
\hline \xrowht{15pt}
Lie algebra & Conditions & Structure equations & Structure type & Example \\ \hline
\hline \xrowht{48pt}
$\mathfrak{h}_3 \oplus \frs{3.3}^0$ & -- & $\left( f^{23},0,0,f^{56},-f^{46},0 \right)$ & \mlines{SKT \\ LCSKT \\ LCB} & \mlines{$Jf_1=f_6$, $Jf_2=f_3$, $Jf_4=f_5$ \\ $\omega=f^{16}+f^{23}+f^{45}$ \medskip\\ $\mu=f^2$} \\
\hline \xrowht{40pt}
$\frs{4.6} \oplus \R^2$ & -- & $\left( f^{23},f^{26},-f^{36},0,0,0 \right)$ & \mlines{SKT \\ LCB} & \mlines{$Jf_1=f_2$, $Jf_3=f_6$, $Jf_4=f_5$ \\ $\omega=f^{12}+f^{36}+f^{45}$} \\
\hline \xrowht{40pt}
$\frs{4.7} \oplus \R^2$ & -- & $\left( f^{23},f^{36},-f^{26},0,0,0 \right)$ & \mlines{SKT \\ LCB} & \mlines{$Jf_1=f_6$, $Jf_2=f_3$, $Jf_4=f_5$ \\ $\omega=f^{16}+f^{23}+f^{45}$} \\
\hline \xrowht{40pt}
\multirow{3}{*}{\vspace{-7em}$\frs{5.16} \oplus \R$} & \multirow{3}{*}{\vspace{-7em}--} & \multirow{3}{*}{\vspace{-7em}$\left( f^{23}+f^{46},f^{36},-f^{26},0,0,0 \right)$} & \mlines{Balanced} & \mlines{$Jf_1=f_5$, $Jf_2=f_3$, $Jf_4=-f_6$ \\ $\omega=f^{15}+f^{23}-f^{46}$} \\
\cline{4-5} \xrowht{40pt}
& & & LCK & \mlines{$Jf_1=f_5$, $Jf_2=f_3$, $Jf_4=f_6$ \\ $\omega=f^{15}+f^{23}+f^{46}$} \\
\cline{4-5} \xrowht{40pt}
& & & 1\textsuperscript{st}-Gauduchon  & \mlines{$Jf_1=f_5$, $Jf_2=f_3$, $Jf_4=f_6$ \\ $\omega=f^{13}+f^{15}+2f^{23}+f^{25}+f^{46}$} \\
\hline \xrowht{40pt}
$\frs{6.25}$ & -- & $\left( f^{23},f^{36},-f^{26},0,f^{46},0 \right)$ & \mlines{SKT} & \mlines{$Jf_1=f_5$, $Jf_2=f_3$, $Jf_4=f_6$ \\ $\omega=f^{15}+f^{23}+f^{46}$} \\
\hline \xrowht{40pt}
$\frs{6.44}$ & -- & $\left( f^{23},f^{36},-f^{26},f^{26}+f^{56},f^{36}-f^{46},0 \right)$ & \mlines{LCB} & \mlines{$Jf_1=f_6$, $Jf_2=f_3$, $Jf_4=f_5$ \\ $\omega=f^{16}+f^{23}+f^{45}$} \\
\hline \xrowht{40pt}
$\frs{6.51}^{p,0}$ & $p>0$ & $\left( f^{23},pf^{26},-pf^{36},f^{56},-f^{46},0 \right)$ & \mlines{SKT \\ LCB} & \mlines{$Jf_1=pf_2$, $Jf_3=f_6$, $Jf_4=f_5$ \\ $\omega=pf^{12}+f^{36}+f^{45}$} \\
\hline \xrowht{40pt}
$\frs{6.52}^{0,q}$ & $q>0$ & $\left( f^{23},f^{36},-f^{26},qf^{56},-qf^{46},0 \right)$ & \mlines{SKT \\ LCB} & \mlines{$Jf_1=f_6$, $Jf_2=f_3$, $Jf_4=f_5$ \\ $\omega=f^{16}+f^{23}+f^{45}$} \\
\hline \xrowht{40pt}
\multirow{2}{*}{\vspace{-5em}$\frs{6.158}$} & \multirow{2}{*}{\vspace{-5em}--} & \multirow{2}{*}{\vspace{-5em}$\left( f^{24}+f^{35},0,f^{36},0,-f^{56},0 \right)$} & \mlines{SKT} & \mlines{$Jf_1=f_3$, $Jf_2=f_4$, $Jf_5=f_6$ \\ $\omega=f^{13}+f^{24}+f^{56}$} \\
\cline{4-5} \xrowht{55pt}
& & & LCB & \mlines{$J=\left( \begin{smallmatrix}
0 & 0 & -1 & 0 & 0 & 1 \\
0 & 0 & 0 & -1 & 0 & 0 \\
1 & 0 & 0 & 0 & 1 & 0 \\
0 & 1 & 0 & 0 & 0 & 0 \\
0 & 0 & 0 & 0 & 0 & -1 \\
0 & 0 & 0 & 0 & 1 & 0
 \end{smallmatrix} \right)$ \vspace{4pt}\\ $\omega=f^{13}+f^{24}-f^{35}+f^{56}$} \\
\hline \xrowht{40pt}
\multirow{3}{*}{\vspace{-7em}$\frs{6.159}$} & \multirow{3}{*}{\vspace{-7em}--} & \multirow{3}{*}{\vspace{-7em}$\left( f^{24}+f^{35},0,f^{56},0,-f^{36},0 \right)$} & \mlines{Balanced} & \mlines{$Jf_1=f_6$, $Jf_2=-f_4$, $Jf_3=f_5$ \\ $\omega=f^{16}-f^{24}+f^{35}$} \\
\cline{4-5} \xrowht{40pt}
& & & LCK & \mlines{$Jf_1=f_6$, $Jf_2=f_4$, $Jf_3=f_5$ \\ $\omega=f^{16}+f^{24}+f^{35}$} \\
\cline{4-5} \xrowht{40pt}
& & & 1\textsuperscript{st}-Gauduchon  & \mlines{$Jf_1=f_6$, $Jf_2=f_4$, $Jf_3=f_5$ \\ $\omega=f^{16}-f^{23}+2f^{24}+f^{35}-f^{45}$} \\
\hline \xrowht{40pt}
$\frs{6.162}^{1}$ & -- & $\left( f^{24}+f^{35},f^{26},f^{36},-f^{46},-f^{56},0 \right)$ & \mlines{Balanced} & \mlines{$Jf_1=f_6$, $Jf_2=f_3$, $Jf_4=f_5$ \\ $\omega=f^{16}+f^{23}+f^{45}$} \\
\hline \xrowht{40pt}
\multirow{2}{*}{\vspace{-5em}$\frs{6.164}^p$} & \multirow{2}{*}{\vspace{-5em}$p>0$} & \multirow{2}{*}{\vspace{-5em}$\left( f^{24}+f^{35},pf^{26},f^{56},-pf^{46},-f^{36},0 \right)$} & \mlines{SKT} & \mlines{$Jf_1=pf_2$, $Jf_3=f_5$, $Jf_4=f_6$ \\ $\omega=pf^{12}+f^{35}+f^{46}$} \\
\cline{4-5} \xrowht{58pt}
& & & LCB & \mlines{$J=\left( \begin{smallmatrix}
0 & -\frac{1}{p} & 0 & 0 & 0 & \frac{1}{p} \\
a & 0 & 0 & 1 & 0 & 0 \\
0 & 0 & 0 & 0 & -1 & 0 \\
0 & 0 & 0 & 0 & 0 & -1 \\
0 & 0 & 1 & 0 & 0 & 0 \\
0 & 0 & 0 & 1 & 0 & 0
 \end{smallmatrix} \right)$ \vspace{4pt}\\ $\omega=pf^{12}-f^{24}+f^{35}+f^{46}$} \\
\hline \xrowht{40pt}
$\frs{6.165}^{p}$ & $p>0$ & $\left( f^{24}+f^{35},pf^{26}+f^{36},-f^{26}+pf^{36},-pf^{46}+f^{56},-f^{46}-pf^{56},0 \right)$ & \mlines{Balanced} & \mlines{$Jf_1=f_6$, $Jf_2=f_3$, $Jf_4=f_5$ \\ $\omega=f^{16}+f^{23}+f^{45}$} \\
\hline \xrowht{40pt}
\multirow{5}{*}{\vspace{-14em}$\frs{6.166}^p$} & \multirow{2}{*}{\vspace{-4em}$p=1$} & \multirow{2}{*}{\vspace{-4em}$\left( f^{24}+f^{35},f^{46},f^{56},-f^{26},-f^{36},0 \right)$} & \mlines{Balanced} & \mlines{$Jf_1=f_6$, $Jf_2=-f_4$, $Jf_3=f_5$ \\ $\omega=f^{16}-f^{24}+f^{35}$} \\
\cline{4-5} \xrowht{40pt}
& & & LCK & \mlines{$Jf_1=f_6$, $Jf_2=f_4$, $Jf_3=f_5$ \\ $\omega=f^{16}+f^{24}+f^{35}$} \\
\cline{2-5} \xrowht{40pt}
& \multirow{3}{*}{\vspace{-7em}$0<|p| \leq 1$, $p \neq 1$} & \multirow{3}{*}{\vspace{-7em}$\left( f^{24}+f^{35},f^{46},pf^{56},-f^{26},-pf^{36},0 \right)$} & \mlines{Balanced}  & \mlines{$Jf_1=f_6$, $Jf_2=-f_4$, $Jf_3=f_5$ \\ $\omega=f^{16}-f^{24}+f^{35}$} \\
\cline{4-5} \xrowht{40pt}
& & & LCK & \mlines{$Jf_1=f_6$, $Jf_2=f_4$, $Jf_3=f_5$ \\ $\omega=f^{16}+f^{24}+f^{35}$} \\
\cline{4-5} \xrowht{40pt}
& & & 1\textsuperscript{st}-Gauduchon & \mlines{$Jf_1=f_6$, $Jf_2=f_4$, $Jf_3=f_5$ \\ $\omega=|p-1|f^{16}+f^{23}+2f^{24}+f^{35}+f^{45}$} \\
\hline \xrowht{40pt}
$\frs{6.167}$ & -- & $\left( f^{24}+f^{35},f^{36},-f^{26},f^{26}+f^{56},f^{36}-f^{46},0 \right)$ & \mlines{Balanced} & \mlines{$Jf_1=f_6$, $Jf_2=f_3$, $Jf_4=f_5$ \\ $\omega=f^{16}+f^{23}+f^{45}$} \\
\hline
\end{tabular}}
\caption{Explicit Hermitian structures on six-dimensional strongly unimodular almost nilpotent Lie algebras with nilradical having one-dimensional commutator.} \label{table-exheis}
\end{center}
\end{table}

\newpage

\begin{table}[H]
\begin{center}
\addtolength{\leftskip} {-2cm}
\addtolength{\rightskip}{-2cm}
\scalebox{0.7}{
\begin{tabular}{|l|l|l|c|c|}
\hline \xrowht{15pt}
Lie algebra & Conditions & Structure equations & Structure type & Example \\ \hline
\hline \xrowht{40pt}
$\frs{6.145}^0$ & -- & $\left( f^{35}+f^{26},f^{45}-f^{16},f^{46},-f^{36},0,0 \right)$ & \mlines{Balanced} & \mlines{$Jf_1=f_2$, $Jf_3=f_4$, $Jf_5=f_6$ \\ $\omega=f^{12}+f^{34}+f^{56}$} \\
\hline \xrowht{58pt}
$\frs{6.147}^0$ & -- & $\left( f^{35}+f^{26}+f^{36},f^{45}-f^{16},f^{46},-f^{36},0,0 \right)$ & \mlines{Balanced} & \mlines{$J=\left( \begin{smallmatrix}
0 & -1 & 0 & -\frac{1}{2} & 0 & 0 \\
1 & 0 & -\frac{1}{2} & 0 & 0 & 0 \\
0 & 0 & 0 & -1 & 0 & 0 \\
0 & 0 & 1 & 0 & 0 & 0 \\
0 & 0 & 0 & 0 & 0 & -1 \\
0 & 0 & 0 & 0 & 1 & 0
 \end{smallmatrix} \right)$ \vspace{4pt}\\ $\omega=2f^{12}+f^{14}+f^{34}+f^{56}$} \\
\hline \xrowht{62pt}
$\frs{6.152}$ & -- & $\left( f^{35}+f^{26},f^{34}-f^{16}+f^{56},f^{45},-f^{56},f^{46},0 \right)$ & \mlines{LCB} & \mlines{$J=\left( \begin{smallmatrix}
0 & -1 & 0 & \frac{1}{2} & 0 & 0 \\
1 & 0 & 0 & 0 & -\frac{1}{2} & 0 \\
0 & 0 & 0 & 0 & 0 & \frac{1}{2} \\
0 & 0 & 0 & 0 & -1 & 0 \\
0 & 0 & 0 & 1 & 0 & 0 \\
0 & 0 & -2 & 0 & 0 & 0
 \end{smallmatrix} \right)$ \vspace{4pt}\\ $\omega=4f^{12}+2f^{25}-f^{36}+f^{45}$} \\
\hline \xrowht{52pt}
$\frs{6.154}^0$ & -- & $\left( f^{35}+f^{26},f^{34}-f^{16},f^{45},-f^{56},f^{46},0 \right)$ & \mlines{LCSKT \\ LCB} & \mlines{$Jf_1=f_2$, $Jf_3=f_6$, $Jf_4=-f_5$ \\ $\omega=f^{12}+f^{36}-f^{45}$ \medskip \\ $\mu=-2f^6$} \\
\hline
\end{tabular}}
\caption{Examples of special Hermitian structures on six-dimensional strongly unimodular almost nilpotent Lie algebras with nilradical having commutator of dimension at least two.} \label{table-exnew}
\end{center}
\end{table}

\end{document}